\documentclass[microtype]{gtpart}

\usepackage[usenames, dvipsnames]{xcolor}

\theoremstyle{definition} 
\newtheorem{proposition}[equation]{Proposition}
\newtheorem{construction}[equation]{Construction}
\newtheorem{definition}[equation]{Definition} 
\newtheorem{theorem}[equation]{Theorem}
\newtheorem{example}[equation]{Example} 
\newtheorem{corollary}[equation]{Corollary} 
\newtheorem{conjecture}[equation]{Conjecture} 
\newtheorem{lemma}[equation]{Lemma} 
\theoremstyle{remark} 
\newtheorem{remark}[equation]{Remark} 
\newtheorem{question}[equation]{Question} 

\usepackage{lscape}
\usepackage{setspace}
\usepackage{amssymb,amsfonts,amsmath,amsthm}
\usepackage{cite, url} 
\usepackage{tikz}
\usepackage{tikz-cd}   
\usepackage{cleveref} 
\usepackage{mathtools} 
\usepackage{xy} 
\usepackage{rotating}
\usepackage{xcolor}
\usepackage{bbm}
\usepackage{mathrsfs}  
\xyoption{all} 
\newdir{ >}{{}*!/-12pt/@{>}} 
\newdir{ -}{{}*!/-5pt/@{}} 
\newdir{- }{{}*!/-5pt/@{}} 
\newdir{ ^(}{{}*!/-5pt/@^{(}} 
\newdir{ _(}{{}*!/-5pt/@_{(}}
\newdir{ |}{{}*!/-5pt/@{|}}
\newdir{_(}{{}*!/0pt/@_{(}}
\xyoption{arc}
\xyoption{ps}
\AtBeginDocument{
  \setlength{\abovedisplayskip}{4pt}
  \setlength{\belowdisplayskip}{4pt}
}
\DeclareMathOperator{\Spec}{Spec}

\newcommand{\KMW}{\underline{\mathbf{K}}^{\mathrm{MW}}}

\newcommand{\SL}{\mathrm{SL}}

\newcommand{\Pic}{\mathrm{Pic}}

\newcommand{\bbC}{\mathbb{C}} 

\newcommand{\PP}{\mathbb{P}^1}
\newcommand{\AAA}{\mathbb{A}^1}
\renewcommand{\J}{\mathcal{J}}

\newcommand{\ato}{\mathbb{A}^2\setminus \{0\}}

\newcommand{\naif}{^\mathrm{N}}

\newcommand{\cc}{\xi}
\newcommand{\ff}{\varphi}
\newcommand{\jj}{\mathbf{j}}
\newcommand{\nc}{\chi}

\newcommand{\se}{s}
\newcommand{\wc}{\upsilon}
\newcommand{\GG}{\mathbf{G}}

\newcommand{\FF}{\mathscr{F}}
\newcommand{\XX}{X}
\newcommand{\qq}{\omega}
\newcommand{\topp}{\mathrm{top}}
\newcommand{\degt}{\deg^{\topp}}

\newcommand{\point}{*}
\newcommand{\Ho}{\mathcal{H}}
\newcommand{\Hop}{\Ho_{\point}}

\newcommand{\sPre}{\mathrm{sPre}}
\newcommand{\sPrek}{\sPre(\Sm_k)}
\newcommand{\sPrep}{\mathrm{sPre}_{\point}}
\newcommand{\sPrepk}{\sPrep(\Sm_k)}

\newcommand{\Sing}{\mathrm{Sing}}
\newcommand{\Singp}{\Sing_{\point}}
\newcommand{\pn}{\calP_n}
\newcommand{\qn}{\calQ_n}

\newcommand{\inv}{^{-1}}
\newcommand{\isom}{\cong}
\newcommand{\fd}{\rightarrow}
\newcommand{\td}{\mapsto}

\newcommand{\bbA}{\mathbb{A}}
\newcommand{\bbF}{\mathbb{F}} 
\newcommand{\bbP}{\mathbb{P}} 
\newcommand{\bbG}{\mathbb{G}} 
\newcommand{\bbZ}{\mathbb{Z}}
\newcommand{\bbN}{\mathbb{N}}
\newcommand{\bbR}{\mathbb{R}}

\newcommand{\calL}{\mathcal{L}}
\newcommand{\calF}{\mathcal{F}}
\newcommand{\calJ}{\mathcal{J}}
\newcommand{\calP}{\mathcal{P}}
\newcommand{\calQ}{\mathcal{Q}}
\newcommand{\calO}{\mathcal{O}}
\newcommand{\calY}{\mathcal{Y}}

\newcommand{\unitm}{\varepsilon}

\newcommand{\stable}{\mathfrak{s}}

\newcommand{\frakm}{\mathfrak{m}}

\newcommand{\st}{ \, \vert \, }
\newcommand{\GW}{\mathrm{GW}}
\newcommand{\MW}{\mathrm{MW}}
\newcommand{\res}{\mathrm{res}}

\newcommand{\Syl}{\mathrm{Syl}}

\DeclareMathOperator{\im}{Im}

\DeclareMathOperator{\id}{id}
\newcommand{\unit}{\mathbbm{1}}

\newcommand{\Sm}{\mathrm{Sm}}

\author{Viktor Balch Barth}
\givenname{Viktor}
\surname{Balch Barth}
\email{viktorbb@math.uio.no}
\address{Department of Mathematics, UiO, Norway.}
\author{William Hornslien}
\givenname{William}
\surname{Hornslien}
\email{william.hornslien@ntnu.no}
\address{Department of Mathematical Sciences, NTNU, Norway.}
\author{Gereon Quick}
\givenname{Gereon}
\surname{Quick}
\email{gereon.quick@ntnu.no}
\address{Department of Mathematical Sciences, NTNU, Norway.}
\author{Glen Matthew Wilson}
\givenname{Glen Matthew}
\surname{Wilson}
\email{glen.wilson@uit.no}
\address{Department of Mathematics and Statistics, UiT, Norway.}

\title{Making the motivic group structure on the endomorphisms of the projective line explicit}

\begin{document}
\date{\today}

\begin{abstract}  
We construct a group structure on the set of pointed naive homotopy classes of scheme morphisms from the Jouanolou device to the projective line. 
The group operation is defined  via matrix multiplication on generating sections of line bundles and only requires basic algebraic geometry. 
In particular, it is completely independent of the construction of the motivic homotopy category. 
We show that a particular scheme morphism, which exhibits the Jouanolou device as an affine torsor bundle over the projective line, induces a monoid morphism from Cazanave's monoid to this group. Moreover, we show that this monoid morphism is a group completion to a subgroup of the group of  scheme morphisms from the Jouanolou device to the projective line. 
This subgroup is generated by a set of morphisms that are very simple to describe. 
\end{abstract}

\maketitle

\tableofcontents

\section{Introduction}

Let $[\PP, \PP]^{\AAA}$ denote the set of self-maps of the projective line in the pointed $\AAA$-homotopy category over a field $k$ introduced by Morel and Voevodsky \cite{MV99}. 
The set $[\PP, \PP]^{\AAA}$ admits the structure of an abelian group and plays the role of the fundamental group of the circle in motivic homotopy theory.

We briefly recall how the group operation on $[\PP, \PP]^{\AAA}$ is defined. 
The standard covering of $\PP$ by two affine lines with intersection $\bbG_m$ yields an $\AAA$-weak equivalence $\PP \simeq S^1 \wedge \bbG_m$.  
The simplicial circle $S^1$ (or some suitable homotopy equivalent model of it, like $\partial \Delta^2$) admits the structure of an h-cogroup. 
The h-cogroup structure on $S^1$ makes it possible to define a group operation on $[S^1 \wedge \bbG_m, \PP]^{\AAA}$ in an analogous way that one defines the fundamental group of a topological space. 
Although the construction mimics the usual construction in algebraic topology, the set of $\bbA^1$-homotopy classes of maps $\bbP^1 \to \bbP^1$ is not simply the set of morphisms $\bbP^1 \to \bbP^1$ modulo an equivalence relation. 
It is unsettling that such an important group does not arise 
in some elementary way as a set of morphisms up to a homotopy relation with some geometrically defined group operation. 
The purpose of the present paper is to remedy this defect. 

Our results build on the work of Asok, Hoyois, and Wendt % in  \cite{AHW2} 
and the work of Cazanave. %in \cite{Caz}.  
In \cite{Caz} Cazanave defines an operation $\oplus\naif$ which turns the set $[\PP,\PP]\naif$ of pointed naive\footnote{A naive homotopy between two pointed morphisms $f, g \colon  \J \to \PP$ is given by a morphism $H \colon  \J \times \AAA \to \PP$ for which the evident restrictions satisfy $H_0 = f$ and $H_1 = g$. 
We note that $H$ must be pointed in the sense that $* \times \AAA$ maps to the basepoint of $\PP$. For more details we refer to Section \ref{sec:naive_htpy_relation}.} 
homotopy classes into a monoid and shows that the canonical map from naive to $\AAA$-homotopy classes $\nu_{\PP}\colon  [\PP,\PP]\naif \to [\PP,\PP]^{\AAA}$ is a group completion. 
However, this approach cannot yield candidates for scheme morphisms that represent inverses of $\AAA$-homotopy classes.  
In \cite{AHW2} Asok, Hoyois, and Wendt show that the set $[\PP,\PP]^{\AAA}$ is in bijection with an explicit set of maps modulo the naive homotopy relation by using the larger set  $\Sm_k(\J, \PP)$ of morphisms of smooth $k$-schemes where $\J$ denotes the Jouanolou device of $\PP$. 
Recall that the smooth affine scheme $\J$ has the concrete description as the spectrum of the ring  
\[
R = \frac{k[x,y,z,w]}{\left(x+w-1,xw-yz\right)}. 
\]
We consider $\J$ equipped with a morphism $\pi \colon  \J \to \PP$ that exhibits $\J$ as an affine bundle torsor.\footnote{See Definition \ref{def:pi} in Section \ref{subsec:def_of_J_and_pi} for the definition of  $\pi$.}  
More precisely, the work by Asok, Hoyois, and Wendt    %\cite{AHW2} 
can be used\footnote{We explain in Appendix \ref{sec:appendix} how the unpointed results of \cite{AHW1} and \cite{AHW2} imply that the canonical map $[\J,\PP]\naif \rightarrow [\J,\PP]^{\AAA}$ between sets of homotopy classes of pointed morphisms is a bijection.}  
to show that there is a bijection $\cc \colon [\calJ , \PP ]\naif \xrightarrow{\cong} [ \PP, \PP]^{\AAA}$. 
The map $\cc$ is the composite of the canonical map 
$\nu \colon [\J,\PP]\naif \to [\J,\PP]^{\AAA}$  and the inverse of the map $\pi_{\AAA}^* \colon [\PP,\PP]^{\AAA} \to [\J,\PP]^{\AAA}$ induced by a scheme morphism $\pi \colon \J \to \PP$. 
The set $[\J, \PP]\naif$ is concrete in the following sense: 
it is the set of pointed scheme morphisms $\J \to \PP$ modulo an equivalence relation generated by naive homotopies.   
This resolves the problem of a lack of candidates of morphisms which may  represent inverses in $[\PP,\PP]^{\AAA}$. 
However, it is not clear at all how the group operation on $[\PP,\PP]^{\AAA}$ or the operation $\oplus\naif$ of \cite{Caz} may be lifted.

In the present paper we define an explicit group structure on the set $[\J,\PP]\naif$ of pointed naive homotopy classes. 
The construction of this group operation is independent of the general machinery of motivic homotopy theory and only uses basic algebraic geometry.   
We then show that the induced map $\pi_{\mathrm{N}}^*\colon  [\PP,\PP]\naif \to [\J,\PP]\naif$ is a morphism of monoids where $[\PP,\PP]\naif$ has the monoid structure from Cazanave \cite[\S 3]{Caz}. 
Moreover, we show that $\pi_{\mathrm{N}}^*$ has image in a concrete subgroup $\GG$ and that the map $\pi_{\mathrm{N}}^*\colon  [\PP,\PP]\naif \to \GG$ is a group completion. 
Hence there are canonical isomorphisms between $\GG$ and $[\PP,\PP]^{\AAA}$ which are compatible with $\pi_{\mathrm{N}}^*$ and $\nu_{\PP}$. 
A key feature of the group $\GG$ is that it is defined by explicit generating scheme morphisms $\J \to \PP$ that are defined in terms of very simple $(2\times 2)$-matrices. 
Hence our result provides a concrete and simple set of scheme morphisms whose images provide generators together with their inverses for the group $[\PP,\PP]^{\AAA}$.

We will now describe our results in more detail.  
Recall that Cazanave shows in \cite[Theorem 3.22]{Caz} that there is an operation $\oplus\naif$ which provides $[\PP,\PP]\naif$ with the structure of a commutative monoid and that the canonical map $[\PP,\PP]\naif \to [\PP,\PP]^{\AAA}$ is a group completion. 
We now state our first main result: 
%shows that there is indeed a concrete group structure on $[\J,\PP]\naif$ and that the morphism $\pi \colon \J \to \PP$ induces a morphism of monoids.   

\begin{theorem}\label{thm:piN_is_monoid_morphism_intro} 
There is an operation $\oplus$ which makes $\left([\J,\PP]\naif,\oplus\right)$ an abelian group such that 
the morphism $\pi \colon \J \to \PP$ induces a morphism of commutative monoids 
\begin{equation*}
\pi_{\mathrm{N}}^* \colon \left([\PP,\PP]\naif,\oplus\naif\right) \to \left([\J,\PP]\naif,\oplus\right)
\end{equation*}
where the left-hand side denotes the monoid of \cite[\S 3]{Caz}.  
\end{theorem}

We will now outline the ideas that lead to the proof of Theorem \ref{thm:piN_is_monoid_morphism_intro}. 
First we describe the construction of the explicit group structure $\oplus$ on $[\J,\PP]\naif$.   
Recall that a morphism $ f \colon \J \to \PP$ is determined by an invertible sheaf $\calL$ over $\J$ and a choice of two generating sections $s_0, s_1 \in \Gamma(\calL, \J)$. 
The invertible sheaf $\calL$ is the pullback $f^* \calO(1)$. 
We say that a morphism $f \colon \J \to \PP$ has degree 0 if $f^*\calO(1)$ is the structure sheaf on $\J$. 
As we will show in Section \ref{sec:degree_0_maps}, the maps $\J \to \PP$ of degree 0 are exactly the maps which factor through the Hopf map $\eta \colon  \ato \to \PP$. 
Let $R$ denote the ring such that $\J = \Spec{R}$. 
The set $[\J, \ato]\naif$ has an apparent group structure: 
A morphism $\J \to \ato$ is given by a unimodular row $(A, B)$ in $R^2$, i.e., there exist $U,V \in R$ for which $AU+BV=1$. 
Any such unimodular row can be completed to a $(2\times 2)$-matrix over $R$, and the product of these matrices defines a group operation on $[\J, \ato]\naif$.  
We describe the details of this construction in Section \ref{subsec:operations_deg_zero}. 
The subgroup of degree 0 maps $\J \to \PP$ is quite large. 
In fact, the degree map  
fits into an exact sequence of pointed \emph{sets} of the form 
\begin{align*}%\label{eq:diagram_group_morphism_intro}
1 \to [\calJ, \ato]\naif \to [\J,\PP]\naif \xrightarrow{\deg}  \Pic(\calJ) \to 1.   
\end{align*}
Hence, in order to turn $[\J, \PP]\naif$ into a group, it  suffices to define an action of $[\J, \ato]\naif$ on $[\J,\PP]\naif$.  
The key idea is that any morphism $f \colon \J \to \PP$ is given by the choice of a line bundle together with two generating sections. 
We can then let a morphism $\J \to \ato$ given by a $(2\times2)$-matrix act on the sections via matrix multiplication. 
We explain the details of this operation in Section \ref{sec:action_of_0_maps_on_n_maps} 
and we complete the construction of the group structure on $[\J, \PP]\naif$ in Section \ref{sec:proposed_group_structure}, see Definition \ref{def:group_action_homotopy} and Theorem \ref{thm:group_structure} where we show that there is an isomorphism of groups $[\J,\PP]\naif \xrightarrow{\ff} [\PP,\PP]^{\AAA}$.

Next we describe the main idea for the proof that $\pi_N^*$ is a morphism of monoids.  
Let $u\in k^{\times}$. 
As in \cite{Caz} we identify a rational function $X/u$ in the indeterminate $X$ with the morphism $\PP \to \PP$ defined by $[x_0:x_1] \mapsto [x_0:ux_1]$. 
For $u,v\in k^{\times}$, let $g_{u,v} \colon \J \to \ato$ denote the morphism given by the unimodular row $\left(x+ \frac{v}{u}w, (u-v)y\right)$ in $R^2$, where $x$, $y$, $z$, and $w$ are the polynomial generators of the ring 
\[
R = \frac{k[x,y,z,w]}{\left(x+w-1,xw-yz\right)}
\]
defining $\J = \Spec{R}$. 
For the rational functions $X/u$ and $X/v$ we then have the identity 
\begin{equation}\label{eq:guv_plus_Xv_is_Xu_intro}
g_{u,v} \oplus \pi_{\mathrm{N}}^*\left(X/v\right) = \pi_{\mathrm{N}}^*\left(X/u\right)
\end{equation}
which we emphasize is an identity of morphisms not just homotopy classes. 
In the particular case where $v=1$, we have $\pi_{\mathrm{N}}^*\left(X/1\right)=\pi$ and Formula \eqref{eq:guv_plus_Xv_is_Xu_intro} reads $g_{u,1} \oplus \pi = \pi_{\mathrm{N}}^*\left(X/u\right)$. This reduces computations for $\pi_{\mathrm{N}}^*\left(X/u\right)$ to computations for $g_{u,1}$ and $\pi$. 
The key technical result needed to prove Theorem \ref{thm:piN_is_monoid_morphism_intro} is that, for every pointed morphism $f \colon \PP \to \PP$, we have an explicit naive homotopy 
\begin{align}\label{eq:piN_Xu_plus_f_is_gu1_piN_X1_f_intro}
\pi_{\mathrm{N}}^*\left(X/u \oplus\naif f \right) \simeq 
g_{u,1}  \oplus \left( \pi_{\mathrm{N}}^*\left(X/1 \oplus\naif f \right) \right). 
\end{align}
The construction of the concrete homotopy in Formula \eqref{eq:piN_Xu_plus_f_is_gu1_piN_X1_f_intro} is based on computations of the resultants of certain morphisms which we provide in Section \ref{sec:computations_for_degree_0_maps} and Appendix \ref{sec:appendix_resultants}. 
Theorem \ref{thm:piN_is_monoid_morphism_intro} then follows from the fact that the set of homotopy classes $[X/u]$ for all $u\in k^{\times}$ 
generates $[\PP,\PP]\naif$ and a successive application of Formula  \eqref{eq:piN_Xu_plus_f_is_gu1_piN_X1_f_intro}. 
The details of this argument are explained in Section  \ref{sec:compatibility_with_Cazanaves_monoid},  see Theorem \ref{thm:pi*_is_a_monoid_morphism}.

%%%%%%%%%%%%%%%

Our second main result is then based on the observation that Identity \eqref{eq:guv_plus_Xv_is_Xu_intro} implies that the image of $\pi_{\mathrm{N}}^*$ is contained in the subgroup $\GG \subseteq [\J,\PP]\naif$ generated by the homotopy classes $[g_{u,v}]$ for all $u,v \in k^{\times}$ and $[\pi]$. Theorem \ref{thm:piN_is_monoid_morphism_intro} and the work of Cazanave \cite[Theorem 3.22]{Caz} then imply that there is a unique group homomorphism $\psi \colon [\PP,\PP]^{\AAA} \to \GG$ such that $\psi \circ \nu_{\PP}=\pi_{\mathrm{N}}^*$. 
In Section \ref{sec:group_completion} we show the  following key result, see Theorem \ref{thm:canonical_isomorphism_from_GG}:

\begin{theorem}\label{thm:piN_is_group_completion_intro} 
The monoid morphism $\pi_{\mathrm{N}}^* \colon [\PP,\PP]\naif \to \GG$ is a group completion. 
There is a unique 
isomorphism $\nc \colon \GG \to [\PP,\PP]^{\AAA}$ such that the diagram below commutes, where $\nc$ and $\psi$ are mutual inverses to each other.  
\begin{align*}
\xymatrix{
\GG  \ar@/^0.8pc/[dr]^-{\nc} &  \\
[\PP,\PP]\naif \ar[u]^-{\pi_{\mathrm{N}}^*} \ar[r]_-{\nu_{\PP}} & 
\ar@{.>}[ul]^-{\psi}
[\PP,\PP]^{\AAA} 
}
\end{align*}
\end{theorem}

Theorem \ref{thm:piN_is_group_completion_intro} gives a very concrete description of all pointed endomorphisms of $\PP$ in the unstable $\AAA$-homotopy category in the following sense:  
the group $\GG$ is given by a simple set of generating morphisms, and 
the group operation $\oplus$ in $\GG$ inherited from $[\J,\PP]\naif$ is defined in basic algebro-geometric terms. 

Finally, we note that the isomorphisms $\GG \xrightarrow{\nc} [\PP,\PP]^{\AAA} \xleftarrow{\ff} [\J,\PP]\naif$ do not imply that $\GG$ equals $[\J,\PP]\naif$. 
However, we conjecture that the inclusion $\GG \subseteq [\J,\PP]\naif$ is an equality and we show in Section \ref{sec:Milnor_Witt_K_theory_and_degree_0} that this is true for all finite fields by computing the first Milnor--Witt K-theory group $K_1^{\mathrm{MW}}(\bbF_q)$, which is isomorphic to $[\PP,\ato]^{\AAA}$ and to the subgroup generated by all classes $[g_{u,v}]$ in $[\J,\ato]\naif$. \\

Below we provide a list of frequently used and important notation together with a reference where the notation is first used after the introduction. \\

%%%%%%%%%%%%%%%%%%%%%%

{\bf Acknowledgments.} 
We thank Aravind Asok, Christophe Cazanave, Marc Hoyois, Marc Levine, Kirsten Wickelgren, Ben Williams, and Paul Arne Østvær for helpful comments, suggestions and clarifications. 
We thank the anonymous referee for many comments and suggestions which helped improve the article.  

The first-named author is supported by the  RCN grant no.\,300814 Young Research Talents of Trung Tuyen Truong. 
The second-named author has received support from the project \emph{Pure Mathematics in Norway} funded by the Trond Mohn Foundation.  
The third- and fourth-named authors gratefully acknowledge the partial support by the RCN Project no.\,313472 {\it Equations in Motivic Homotopy}. 
The authors would like to thank the Centre for Advanced Study in Oslo for its hospitality where parts of the work on the paper were carried out. \\

%%%%%%%%%%%%%%%%%%%%%

\section*{List of important notation}

% We now list some notation that will be used throughout the paper: 

%\begin{center}
\begin{tabular}{l|l|l}
Notation & Brief description & First discussed \\
\hline
$k$ & a field & Section \ref{subsec:def_of_J_and_pi}\\ 
$\PP$ & projective line over $k$ pointed at $\infty$ & Section \ref{subsec:def_of_J_and_pi} \\
$R$ & The ring $R:=k[x,y,z,w]/(xw-yz,x+w-1)$ & Definition \ref{def:J}\\
$\J$ & the Joualolou device $\J = \Spec{R}$ & Definition \ref{def:J} \\
$\jj$ & basepoint $(x-1,y,z,w) \subseteq R$ of $\J$ & Definition \ref{def:J}\\
$\jj'$ & basepoint $(x-1,y,z,w) \subseteq R[T]$ of $\J \times \AAA$ & Remark \ref{rem:def_of_jj'}\\
$\calP_1$ & line bundle over $\J$, image of  $\begin{pmatrix}
    x & y \\ z & w
\end{pmatrix}$ & Section \ref{subsec:def_of_J_and_pi} \\
$\calQ_1$ & line bundle over $\J$, image of  $\begin{pmatrix}
x & z \\ y & w
\end{pmatrix}$ & Section \ref{subsec:def_of_J_and_pi} \\ 
$[s_0, s_1]$ & map $\calJ \to \PP$ $\leftrightarrow$ line bundle $\calP_n$ with sections $s_0$, $s_1$ & Construction \ref{con:J_to_P_map} \\ 
$X/u$ & morphism $\PP \to \PP$, $[x_0:x_1] \mapsto [x_0:ux_1]$ & Section \ref{sec:computations_for_degree_0_maps} \\
$\res(A,B)$ & resultant of polynomials $A$, $B$ & Proposition \ref{prop:Caz_thm_pointed_P1_endo} \\
$(a_0,a_1:b_0,b_1)_n$ & morphism  $\J \to \PP$ & Definition \ref{def:notation_maps} \\
$\deg$ & degree of $\calL$ of a morphism $(\calL, s_0,s_1) \colon \J \to \PP$ & Section \ref{sec:degree_0_maps} \\ 
$\eta$ & the Hopf map $\eta \colon \ato \to \PP$ & Section \ref{sec:degree_0_maps} \\
$\SL_2$ & $\SL_2 := \Spec \left( k[a,b,c,d]/(ad-bc-1)\right)$ & Section \ref{sec:degree_0_maps} \\ 
$M$ & an $\SL_2(R)$-matrix & Section \ref{sec:degree_0_maps} \\
$\phi $ & first column morphism $\SL_2 \to \ato$ & Definition \ref{def:def_of_phi}  \\
$g_{u,v}$ & unimodular row $\left(x+ \frac{v}{u}w, (u-v)y\right)$ in $R$ & Definition \ref{def:g_uv} \\
$m_{u,v}$ & $\SL_2(R)$-matrix $\begin{pmatrix}
x+\frac{v}{u}w & \frac{u-v}{uv}z \\ (u-v)y & x + \frac{u}{v}w
\end{pmatrix}$ %completing $g_{u,v}$ 
& Definition \ref{def:g_uv} \\
$\Sm_k$ & the category of smooth finite type $k$-schemes & Section \ref{sec:naive_htpy_relation} \\
$\Sm_k(X,Y)_*$ & set of pointed morphisms in $\Sm_k$ & Section \ref{sec:naive_htpy_relation} \\
$\simeq$ & naive homotopy relation & Definition \ref{def:pointed_naive_homotopy} \\
$[X,Y]\naif$ & set of pointed naive homotopy classes 
& Definition \ref{def:pointed_naive_homotopy} \\
$[\J,\PP]_n\naif$ & set of pointed naive homotopy classes of degree $n$ & Section \ref{sec:degree_0_maps} \\
$[X,Y]^{\AAA}$ & set of pointed $\AAA$-homotopy classes & Section \ref{sec:group_completion} \\
$\nu$ & canonical map $[\J, \PP]\naif \to [\J, \PP]^{\AAA}$ & Section \ref{sec:group_completion} \\
$\pi_{\mathrm{N}}^*$ & map $[\PP,\PP]\naif \to [\J,\PP]\naif$ induced by $\pi$ & Section \ref{sec:compatibility_with_Cazanaves_monoid} \\
$\pi^*_{\AAA}$ & map $[\PP,\PP]^{\AAA} \to [\J,\PP]^{\AAA}$ induced by $\pi$ & Section \ref{sec:group_completion} \\
$\cc$ & map $\cc = (\pi^*_{\AAA})^{-1} \circ \nu \colon  [\J, \PP]\naif \to [\PP, \PP]^{\AAA}$  & Equation \eqref{eq:jpn_ppa} \\
$\oplus$ & group operation on $[\J, \PP]\naif$ & Definition \ref{def:full_group} \\
$\oplus\naif$ & monoid operation on $[\PP, \PP]\naif$ & Section \ref{sec:compatibility_with_Cazanaves_monoid} \\
$\oplus^{\AAA}$ & group operation on $[\PP,\PP]^{\AAA}$ & Section \ref{sec:group_completion} \\
$\ff$ & group isom. $\left([\J, \PP]\naif,\oplus\right) \xrightarrow{\cong} \left([\PP, \PP]^{\AAA},\oplus^{\AAA}\right)$ & Theorem \ref{thm:group_structure}\\ 
$\deg^{\AAA}$ & Morel's $\AAA$-Brouwer degree $[\bbP^1, \bbP^1]^{\AAA} \to \GW(k)$ & Section \ref{sec:motivic_Brouwer_degree} 
\end{tabular}
%\end{center}

%%%%%%%%%%%%%%%%%%%%%

\section{The Jouanolou device and morphisms to \texorpdfstring{$\bbP^1$}{P1}}\label{sec:Jouanolou_device_and_maps}

In this section we work out the details needed about the Jouanolou device $\J$, morphisms $\J \to \PP$, and the pointed naive homotopy relation. 
We keep the terminology as elementary as possible and hope that the details provided help make our approach accessible. 

\subsection{Definition of \texorpdfstring{$\calJ$}{J},  \texorpdfstring{$\pi$}{pi}, and \texorpdfstring{$\xi$}{xi}}\label{subsec:def_of_J_and_pi}

Throughout this paper $k$ will always denote a field. 
All schemes are schemes over $\Spec{k}$.  
We denote by $\PP$ the projective line over $k$ pointed at $\infty := [1:0] \in \PP$. 
%We denote by $\Sm_k$ the category of smooth finite type $k$-schemes. 
The letter $R$ will always denote the following ring. 

\begin{definition}\label{def:J}
Let $R$ denote the ring
\begin{equation*}
   R = \frac{k[x,y,z,w]}{\left(x+w-1,xw-yz\right)}.
\end{equation*}
%We will frequently identify $R$ with the ring $k[x,y,z]/(x(1-x)-yz)$ where it is convenient.  
The Jouanolou device of $\PP$ is the smooth affine $k$-scheme $\calJ = \Spec{R}$. 
We consider $\calJ$ to be pointed at $\jj = (x-1, y, z, w) $. 
\end{definition}

The ring $R$ may be interpreted as the ring representing $(2\times 2)$-matrices with trace $1$ and determinant $0$. 
Namely, a ring homomorphism $R \to S$ is equivalent to a $(2 \times 2)$-matrix over $S$ with trace $1$ and determinant $0$.  

While we will discuss morphisms $\calJ \to \PP$ in more detail later, we point out that there are two evident morphisms that exhibit $\calJ$ as an affine torsor bundle over $\PP$.   
The matrices 
\begin{equation*}
    \begin{pmatrix}
    x & y \\ z & w
    \end{pmatrix}\,
    \text{ and } \,
    \begin{pmatrix}
    x & z \\ y & w
    \end{pmatrix} 
\end{equation*}
over $R$ are idempotent. When viewed as linear transformations from $R^2$ to $R^2$, the image of each matrix defines a projective module, denoted by $\calP_1$ and $\calQ_1 $ respectively. Both $\calP_1$ and $\calQ_1$ have rank 1 and so they yield invertible sheaves on $\calJ$. 

\begin{definition}\label{def:pi}
We define the morphism of schemes $\pi \colon \calJ \to \PP$ by selecting the invertible sheaf associated to $\calP_1$ and the generating sections 
\begin{equation*}
s_0 = \begin{pmatrix}
x \\ z
\end{pmatrix}\,
\text{ and }\,
s_1 = \begin{pmatrix}
y \\ w
\end{pmatrix}.
\end{equation*}
Similarly, we define the morphism of schemes $\tilde{\pi} \colon \calJ \to \PP$ by using $\calQ_1$ and the choice of generating sections 
\begin{equation*}
s_0 = 
\begin{pmatrix}
x \\ y
\end{pmatrix}
\text{ and } 
s_1 = 
-\begin{pmatrix}
z \\ w
\end{pmatrix}. 
\end{equation*}
\end{definition}

We intuitively understand the map $\pi$ as sending a point in $\calJ$ corresponding to a matrix $\begin{pmatrix}
x & y \\ 
z & w
\end{pmatrix}$
to either the point with homogeneous coordinates $[x : y]$ or $[z : w]$, depending on which is defined. 
When both points make sense in $\PP$, they agree, so the map is well-defined. 
A similar argument shows that $\tilde{\pi}$ is well-defined. 
Both $\pi$ and $\tilde{\pi}$ exhibit $\calJ$ as an affine torsor bundle over $\PP$, hence they are $\AAA$-homotopy equivalences. 
It follows that
\begin{equation}\label{eq:ppa_jpa}
\pi_{\AAA}^* \colon \left[\PP, \PP \right]^{\AAA} \to \left[\calJ, \PP \right]^{\AAA}
\end{equation}
is a bijection.  
We show in Proposition \ref{prop:pointed_naive_htpy_conclusion_appendix} in Appendix \ref{sec:appendix} that the  canonical map $\nu \colon [\J, \PP]\naif \to [\J, \PP]^{\AAA}$ is a bijection because $\calJ$ is affine and $\PP$ is $\AAA$-naive. 
Thus, the composition of the bijection $\nu$ and the inverse of $\pi_{\AAA}^*$ is a bijection
\begin{equation}\label{eq:jpn_ppa}
\cc \colon \left[\calJ, \PP \right]\naif \fd \left[\PP, \PP \right]^{\AAA}.
\end{equation}
This bijection may be described as follows. 
A naive pointed homotopy class of maps $[f]$ represented by the pointed scheme morphism $f \colon \calJ \to \PP$ is sent to $\cc([f]) = [f\circ \pi^{-1}]^{\AAA}$, the pointed $\AAA$-homotopy class of the zig-zag $\PP \xleftarrow{\pi} \calJ \xrightarrow{f}\PP$. 

In the following sections we will investigate the set $[\J,\PP]\naif$ of pointed naive homotopy classes of pointed morphisms $\J \to \PP$. 

%%%

\subsection{Convenient coordinates for \texorpdfstring{$\calJ$}{J}}

The map $\pi \colon \calJ \to \PP$ encourages the choice of a convenient set of coordinate charts for $\calJ$. 
For $\PP$, we use the standard notation $U_0 = \PP \setminus \{[0:1]\}$ and $U_1 = \PP \setminus \{[1:0]\}$. 
It is straightforward to verify that the preimages under $\pi$ of $U_0$ and $U_1$ are $\pi^{-1}(U_0) = D(x) \cup D(z)$ and $\pi^{-1}(U_1) = D(y) \cup D(w)$. Both of these open sets are isomorphic to $\bbA^2$ under the following maps. 

\begin{lemma}\label{lem:coords}
The open set $D(x) \cup D(z)\subseteq \calJ$ is isomorphic to $\Spec (k[a,b])$ under the map $\Psi_0 \colon \bbA^2 \to \calJ$ given by 
$x \mapsto 1-ab$, 
$y \mapsto a(1-ab)$,
$z \mapsto b$, and 
$w \mapsto ab$. 

Similarly, the open set $D(y) \cup D(w) \subseteq \calJ$ is isomorphic to $\Spec{k[c,d]}$ under the map $\Psi_1 \colon \bbA^2 \to \calJ$ given by $x \mapsto cd$, $y \mapsto d$, $z\mapsto c(1-cd)$, and $w \mapsto 1-cd$. 
\end{lemma}

\begin{proof}
The proof proceeds by studying the map locally. 
For instance, $\Psi_0$ induces an isomorphism of rings $k[a,b][(1-ab)^{-1}] \to R[x^{-1}]$ and also of $k[a,b][b^{-1}] \to R[z^{-1}]$. 
The open sets $D(1-ab)$ and $D(b)$ cover $\bbA^2$, so it follows that $\Psi_0$ maps surjectively onto $D(x) \cup D(z)$. The inverse map is obtained by gluing the maps that are defined on $D(x)$ and $D(z)$, giving the result. 
A similar argument works for $\Psi_1$. 
\end{proof}

\begin{remark}
The open affine subschemes $D(x)\cup D(z)$ and $D(y) \cup D(w)$ of $\Spec{R}$ have the odd property that  their ring of global sections is not a localization of $R$. 
\end{remark}

%%%

\subsection{Invertible sheaves on \texorpdfstring{$\calJ$}{J}}
\label{sec:invertible_sheaves}

By \cite[Theorem II.7.1]{Hartshorne}, a morphism $\J \to \PP$ is determined by an invertible sheaf $\calL$ on $\J$ and two generating global sections of $\calL$. 
We now take the time to study the invertible sheaves on $\J$ to enable our study of the morphisms $\J \to \PP$.  
We will assume familiarity with the basic terminology presented in both \cite[Chapter 1]{Kbook} and \cite{AtiyahMacdonald}. 

Since $\calJ = \Spec{R}$ is an irreducible affine scheme, the invertible sheaves on $\calJ$ correspond to projective $R$-modules of rank 1. We have already seen the projective modules $\calP_1$ and $\calQ_1$ used to define $\pi$ and $\tilde{\pi}$ above. 
Since the map $\pi \colon \calJ \to \PP$ is an $\AAA$-weak equivalence 
and the Picard group functor is homotopy invariant, the induced map on Picard groups is an isomorphism $\pi^* \colon \Pic(\PP) \to \Pic(\calJ)$. 
Since $\pi^*(\calO(1)) = \calP_1$ and $\Pic(\calJ)\cong \bbZ$, it follows that $\calP_1$ generates the Picard group of $\calJ$. For future reference, we state this as a lemma.

\begin{lemma} \label{lem:pic_J_is_Z}
The Picard group of $\calJ$ is isomorphic to $\bbZ$ and $\calP_1$ is a generator. 
\end{lemma}

Furthermore, $\calQ_1 = - \calP_1$ in $\Pic(\calJ)$ as the following proposition shows. 

\begin{proposition}\label{prop:Q_n_is_inverse_to_P_n}
There is an isomorphism $\mathcal{P}_1\otimes\mathcal{Q}_1 \cong R$.    
\end{proposition}

\begin{proof}
The $R$-module $\mathcal{P}_1 \otimes \mathcal{Q}_1$ is generated by 
\begin{equation*}
    \left\{ \begin{bmatrix}x\\z\end{bmatrix}\otimes\begin{bmatrix}x\\y\end{bmatrix}, \begin{bmatrix}x\\z\end{bmatrix}\otimes\begin{bmatrix}z\\w\end{bmatrix}, \begin{bmatrix}y\\w\end{bmatrix}\otimes\begin{bmatrix}x\\y\end{bmatrix},\begin{bmatrix}y\\w\end{bmatrix}\otimes\begin{bmatrix}z\\w\end{bmatrix} \right\}.
\end{equation*}
Consider the module homomorphism  $m \colon R^2 \otimes R^2 \longrightarrow R^2$ induced by component-wise multiplication  $m \left(\begin{bmatrix}a\\b\end{bmatrix}\otimes\begin{bmatrix}c\\d\end{bmatrix}\right) = \begin{bmatrix}ac\\bd\end{bmatrix}$. 
We restrict $m$ to $\calP_1\otimes \calQ_1$ and observe that the image of $\calP_1\otimes \calQ_1$ under $m$ is the submodule $\left\langle \begin{bmatrix}x \\ w \end{bmatrix}\right\rangle \subseteq R^2$ (use $x+w = 1$). 
This is a free $R$-module of rank 1. 
As $m \colon \calP_1\otimes \calQ_1 \to \left\langle \begin{bmatrix}x \\ w \end{bmatrix}\right\rangle$ is surjective, it follows that it is locally an isomorphism at all maximal ideals $\frakm \subseteq R$. 
Hence the map $m$ itself restricted to $\calP_1\otimes \calQ_1$ is an isomorphism onto its image.  
\end{proof}

We would like to understand the tensor powers of $\mathcal{P}_1$ and $\mathcal{Q}_1$.

\begin{definition} \label{def: pn and qn def as submodules of R2}
Let $\calP_n$ and $\calQ_n$ denote the submodules of $R^2$ generated by 
\begin{align*} 
\mathcal{P}_n & = \left\langle \begin{bmatrix}x^{n}\\z^{n}\end{bmatrix}, \begin{bmatrix}x^{n-1}y\\z^{n-1}w\end{bmatrix}, \ldots, \begin{bmatrix}y^n\\w^n\end{bmatrix}  \right\rangle ~ \text{and} ~
\calQ_n  = \left\langle  \begin{bmatrix}x^{n}\\y^{n}\end{bmatrix}, \begin{bmatrix}x^{n-1}z\\y^{n-1}w\end{bmatrix}, \ldots, \begin{bmatrix}z^n\\w^n\end{bmatrix} \right\rangle. 
\end{align*}
\end{definition}

The following lemma is useful for simplifying proofs. It shows that what we prove about $\calP_n$ by symmetry holds for $\calQ_n$.
\begin{lemma} \label{lem:suffices_to_prove_for_pn}
Define the involution $\tau \colon R \fd R$ by 
\begin{align*}
\tau( x) &= x & \tau(y) &= z \\
\tau(z) &= y & \tau(w) &= w. 
\end{align*}
Pulling back along $\tau$ gives $R$-module isomorphisms $\tau^*\pn \isom \qn $ and $\tau^*\qn \isom \pn $.
\end{lemma}
\begin{proof}
To more easily distinguish between them, we give the domain and codomain of $\tau $ different names and write $\tau \colon R \to R'$.
Pulling back the $R'$-module $\pn$, we get the $R$-module $\tau^*\pn$, where the multiplication is defined by $r\cdot_R p = \tau(r)\cdot_{R'} p $. 
The map $\tau^*\pn \fd\qn $ is defined on basis elements by
\begin{equation*}
\begin{bmatrix}x^{n-i}y^i\\z^{n-i}w^i\end{bmatrix} \td \begin{bmatrix}x^{n-i}z^i\\y^{n-i}w^i\end{bmatrix}.
\end{equation*}
It is easily checked that $f$ is bijective and $R$-linear and hence an $R$-module isomorphism. 

To see that $\tau^*\qn \isom \pn $, we pull back the isomorphism $\qn \isom \tau^*\pn$, which we just proved,  along $\tau$ on both sides. 
Since $\tau \circ \tau = \id$, this simplifies to $\tau^*\qn \cong \tau^* \tau^* \pn = \pn$. 
\end{proof}

\begin{proposition}
\label{prop:span_prop}
The $R$-modules $\calP_n$ and $\calQ_n$ are also generated in the following way 
\begin{align*} 
\mathcal{P}_n = \left\langle \begin{bmatrix}x^{n}\\z^{n}\end{bmatrix}, \begin{bmatrix}y^n\\w^n \end{bmatrix} \right\rangle 
~ ~ \text{and} ~ ~ 
\calQ_n  = \left\langle \begin{bmatrix}x^{n}\\y^{n}\end{bmatrix}, \begin{bmatrix}z^n\\w^n \end{bmatrix} \right\rangle. 
\end{align*}
\end{proposition}
\begin{proof}
We only prove the claim for $\calP_n$, as the proof for $\calQ_n$ is analogous by Lemma \ref{lem:suffices_to_prove_for_pn}. 
Containment in one direction is clear by definition of $\calP_n$.  
Now fix $n$ and pick a number $0\leq i \leq n$. We then have
\begin{equation*}
    \begin{bmatrix}x^{n-i}y^i\\z^{n-i}w^i\end{bmatrix} =
    \left(x+w\right)^n \begin{bmatrix}x^{n-i}y^i\\z^{n-i}w^i\end{bmatrix}
    = \sum_{d=0}^n \binom{n}{ d} x^{n-d}w^d
    \begin{bmatrix}x^{n-i}y^i\\z^{n-i}w^i\end{bmatrix}.
\end{equation*}
For each $d$, one of the following hold:
\begin{equation*}
\begin{array}{ll}
			x^{n-d}w^d \begin{bmatrix}x^{n-i}y^i\\z^{n-i}w^i\end{bmatrix} = x^{n-i-d}y^iw^d\begin{bmatrix}x^{n}\\z^{n}\end{bmatrix} & \mbox{if } i+d \leq n, \\
			& \\
			x^{n-d}w^d \begin{bmatrix}x^{n-i}y^i\\z^{n-i}w^i\end{bmatrix} = x^{n-d}z^{n-i}w^{d+i-n} \begin{bmatrix}y^{n}\\w^{n}\end{bmatrix} & \mbox{if } i+d > n.    
\end{array}
\end{equation*}
This completes the proof. 
\end{proof}

We will now show that the $R$-modules $\calP_n$ and $\calQ_n$ are algebraic line bundles, that is, finitely generated $R$-modules of constant rank 1. We check locally and in the process give a description of $\calP_n$ and $\calQ_n$ with open patching data  \cite[2.5 on page 14]{Kbook}.

\begin{proposition}
\label{prop:local_patching_data}
The $R$-modules $\calP_n$ and $\calQ_n$ are algebraic line bundles over $R$, that is, finitely generated $R$-modules of constant rank 1, and determine invertible sheaves on $\J$. They are described in terms of open patching data by 
\begin{align*}
\calP_n &\cong \left\{\left(f_x, f_w\right) \in R[x^{-1}] \times R[w^{-1}] \st   f_w = \left(\frac{z}{x}\right)^n f_x \right\}, \\
\calQ_n &\cong \left\{\left(f_x, f_w\right) \in R[x^{-1}] \times R[w^{-1}] \st   f_w = \left(\frac{y}{x}\right)^n f_x\right\}.
\end{align*}
\end{proposition}
\begin{proof}
The canonical projections %$p_1 \colon\calP_n[x^{-1}] \to R[x^{-1}]$ and $p_2 \colon \calP_n[w^{-1}] \to R[w^{-1}]$
$\calP_n[x^{-1}] \to R[x^{-1}]$ and $\calP_n[w^{-1}] \to R[w^{-1}]$ 
are isomorphisms. Since $D(x) \cup D(w) = \calJ$, we conclude $\calP_n$ is locally free of constant rank 1 and is thus an algebraic line bundle \cite[Lemma 2.4]{Kbook}. For $\begin{bmatrix}f \\ g \end{bmatrix} \in \calP_n[x^{-1},w^{-1}]$, one checks that $(z/x)^nf = g$, which determines $\calP_n$ in terms of open patching data. A similar calculation works for $\calQ_n$.
\end{proof}

\begin{remark}\label{rem:pair_of_secs_as_map_to_P1}
Proposition \ref{prop:local_patching_data} shows us how to interpret an element 
\begin{equation*}
s_0 = a_0 \begin{bmatrix} x^n \\ z^n \end{bmatrix} + a_1 \begin{bmatrix} y^n \\ w^n \end{bmatrix}, ~ \text{with} ~ a_0, a_1 \in R
\end{equation*}
that is a global section of the invertible sheaf associated to $\calP_n$. 
Namely, the global section $s_0$ restricted to $D(x)$ is described by $a_0 x^n + a_1 y^n$, while on $D(w)$ the section is $a_0 z^n + a_1 w^n$. On the overlap, the two sections agree when compared using the appropriate transition functions. 

Combining Propositions 
\ref{prop:span_prop} and \ref{prop:local_patching_data} it follows that any tuple $(f_x, f_w) \in R[x^{-1}] \times R[w^{-1}]$ satisfying $f_w = \left(\frac{z}{x}\right)^n f_x$ can be expressed as
\begin{equation*}
\begin{bmatrix}f_x \\ f_w \end{bmatrix} 
= a_0 \begin{bmatrix} x^n \\ z^n \end{bmatrix} + a_1 \begin{bmatrix} y^n \\ w^n \end{bmatrix}, ~ \text{for some} ~ a_0, a_1 \in R. 
\end{equation*} 
\end{remark}

The algebraic line bundles $\calP_n$ and $\calQ_n$ may also be described as the image of an idempotent $(2\times 2)$-matrix of rank 1. 
For $n\geq 1$, let $A=\sum_{i=0}^{n-1} \binom{2n-1}{i}x^{n-1-i}w^i$ and $B=\sum_{i=n}^{2n-1} \binom{2n-1}{i}x^{2n-1-i}w^{i-n}$. Then $x^nA + w^n B= (x+w)^{2n-1}=1$.
Define
\begin{equation}
    \label{eq:Mn}
    M_n = \begin{pmatrix}x^n A  & y^n B \\ z^n A & w^n B\end{pmatrix} \text{ and }  
    M_n' = \begin{pmatrix}x^n A  & z^n B \\ y^n A & w^n B\end{pmatrix}.
\end{equation}

\begin{proposition}
For every $n\geq 1$, the matrices $M_n$ and $M_n'$ are idempotent of rank 1. When viewed as linear transformations from $R^2$ to $R^2$, the image of $M_n$ and $M_n'$ is $\calP_n$ and $\calQ_n$ respectively. 
\end{proposition}

\begin{proof}
It is straightforward to verify that $M_n$ is idempotent using the relation $1=x^nA + w^n B$ and that $\im(M_n) \subset \calP_n $. Note that 
 \begin{equation*}
    x^n\begin{bmatrix} x^n A\\ z^n A\end{bmatrix} + z^n\begin{bmatrix} y^n B\\ w^nB \end{bmatrix} = (x^nA + w^nB) \begin{bmatrix} x^n \\ z^n \end{bmatrix} =  \begin{bmatrix} x^n \\ z^n \end{bmatrix}
 \end{equation*}
 and similarly, 
 \begin{equation*}
    y^n\begin{bmatrix} x^n A\\ z^n A\end{bmatrix} + w^n\begin{bmatrix} y^n B\\ w^nB \end{bmatrix} = (x^nA + w^nB) \begin{bmatrix} y^n \\ w^n \end{bmatrix} =  \begin{bmatrix} y^n \\ w^n \end{bmatrix}.
 \end{equation*}
So $\calP_n \subset \im(M_n)$, and the image is equal to $\calP_n$. 
The argument for $M_n'$ and $\calQ_n$ is similar. 
\end{proof}

\begin{proposition}\label{prop:tensor_of_P1_is_Pn}
For every $n\ge 1$, the morphisms $\calP_n \otimes \calP_1 \to \calP_{n+1}$ and $\calP_1^{\otimes n} \to \calP_n$ obtained from component-wise multiplication are isomorphisms. 
A similar statement holds for $\calQ_n$. 
\end{proposition}
\begin{proof}
Consider the $R$-module map $m \colon R^2 \otimes R^2 \to R^2$ induced by component-wise multiplication. 
By the description of the generators of the modules $\calP_n$, it is clear that $m$ restricts to a map $m \colon \calP_n \otimes \calP_1 \to \calP_{n+1}$ and this map is surjective. 
As both $\calP_n\otimes \calP_1$ and $\calP_{n+1}$ are algebraic line bundles, the map $m$ is surjective locally at every maximal ideal $\frakm\subseteq R$ and hence an isomorphism. 
Thus $m \colon \calP_n \otimes \calP_1 \to \calP_{n+1}$ is itself an isomorphism by \cite[Proposition 3.9]{AtiyahMacdonald}. 
This proves the first claim. 
The second claim now follows by induction. 
\end{proof}

We now have a complete description of the isomorphism $\bbZ \cong \Pic(\calJ)$.

\begin{proposition}\label{prop:P_n_and_Q_n_generate_Pic(J)}
Under the isomorphism $\Pic(\calJ) \cong \Pic(\PP) \cong \bbZ$ arising from the $\AAA$-homotopy equivalence $\pi \colon  \J \to \PP$, the modules $\calP_n$ and $\calQ_n$ correspond to $n$ and $-n$, respectively, while the trivial invertible sheaf $\calO$ corresponds to $0$. 
\end{proposition}
\begin{proof}
By Lemma \ref{lem:pic_J_is_Z}, $\Pic (\J) = \bbZ$, and $\calP_1$ generates the Picard group. 
By Proposition \ref{prop:Q_n_is_inverse_to_P_n}, the inverse of $\calP_1$ is$\calQ_1$. 
By Proposition \ref{prop:tensor_of_P1_is_Pn}, the modules $\calP_n$ and $\calQ_n$ correspond to $n$ and $-n$ in the Picard group.
\end{proof}

%%%

\subsection{Pointed morphisms \texorpdfstring{$\PP \to \PP$}{P1 to P1} and \texorpdfstring{$\J \to \PP$}{J to P1}} \label{sec:pointed_maps_p1_to_p1_and_j_to_p1}

We will now study morphisms to $\PP$ in more detail. 
By \cite[Theorem II.7.1]{Hartshorne}, for a smooth $k$-scheme $X$, the data needed to give a morphism $f \colon  X \to \PP$ are an invertible sheaf $\calL$ over $X$ and the choice of two global sections $s_0, s_1 \in \Gamma(X, \calL)$ that generate the invertible sheaf $\calL$. 
That is, at every point $p \in X$, the stalks of the sections $(s_0)_p$ and $(s_1)_p$ generate the local ring $\calL_p$. 
We then write $[s_0, s_1]$ for the morphism $X \to \PP$ given by the data above, where we usually omit the invertible sheaf $\calL$ from the notation. 
We note that throughout the paper we use the terms \emph{morphism} and \emph{map} interchangeably. 

The scheme $\PP$ is pointed at $\infty = [1:0]$.  
A pointed morphism $f \colon  \PP \to \PP$ by definition is a morphism satisfying $f(\infty) = \infty$. 
A pointed morphism $f \colon \PP \to \PP$ given by the invertible sheaf $\calO(n)$ on $\PP$ with two generating sections $\sigma_0, \sigma_1 \in k[x_0,x_1]_{(n)}$ has the following special form by work of Cazanave \cite{Caz}.  

\begin{proposition}\label{prop:Caz_thm_pointed_P1_endo} \cite[Proposition 2.3]{Caz}  
A pointed $k$-scheme morphism $f \colon \bbP^1 \to \bbP^1$ corresponds uniquely to the data of a natural number $n$ and a choice of two polynomials, $A = \sum_{i=0}^n a_iX^i$ and $B = \sum_{i=0}^{n-1}b_iX^i$ in $k[X]$ for which $a_n = 1$ and the resultant $\res(A,B)$ is non-zero. 
The integer $n$ is called the degree of $f$ and is denoted $\deg(f)$; the scalar
$\res(f) = \res(A, B) \in  k^\times$ is called the resultant of $f$. 
We recall that $\res(A,B) =\det \left(\Syl (A,B) \right) \in k$, 
where $\Syl(A,B)$ is the Sylvester matrix of the pair of polynomials $(A,B)$ which we recall in Definition \ref{def:sylvesterAndRes}.
\end{proposition}

\begin{remark}\label{rem:from_fPP_to_sections}
One easily translates from the morphism $f \colon \PP \to \PP$ given by $n$, $A$, and $B$ in Proposition \ref{prop:Caz_thm_pointed_P1_endo} to the morphism given by the invertible sheaf $\calO(n)$ and the choice of global sections $\sigma_0 = \sum_{i=0}^n a_ix_0^ix_1^{n-i}$ and $\sigma_1 = \sum_{i=0}^n b_ix_0^ix_1^{n-i}$ where we understand $b_n = 0$. 
The resultant condition guarantees that these global sections generate $\calO(n)$. 
The condition $a_n=1$ is a normalizing condition to give a bijective correspondence between morphisms and the data $n$, $A$, and $B$. 
We will find it more convenient to use the data $[\sigma_0, \sigma_1] \colon \PP \to \PP$ and $\calO(n)$ to describe a pointed map in what follows.
\end{remark}

We will now explain in detail how we can use this perspective to describe morphisms via line bundles and generating sections in the special case $\J \to \PP$. 

\begin{proposition}\label{prop:compose_caz_with_pi}
Consider a pointed map $[\sigma_0, \sigma_1] \colon \PP \to \PP$ with invertible sheaf $\calO(n)$, $\sigma_0 = \sum_{i=0}^n a_ix_0^ix_1^{n-i}$ and $\sigma_1 = \sum_{i=0}^n b_ix_0^ix_1^{n-i}$. 
The composition $[\sigma_0, \sigma_1] \circ \pi$ is the map $[s_0, s_1] \colon \calJ \to \PP$ with invertible sheaf $\calP_n$ and global sections 
\begin{align}
    \label{eq:caz_section}     
    s_0 = \sum_{i=0}^n a_i \begin{bmatrix}x^iy^{n-i} \\ z^iw^{n-i} \end{bmatrix} ~ \text{and} ~ 
    s_1 & = \sum_{i=0}^n b_i \begin{bmatrix}x^ny^{n-i} \\ z^iw^{n-i} \end{bmatrix}.
\end{align}
\end{proposition}

\begin{proof}
This is a straightforward calculation. 
The condition on the resultant ensures that the sections $s_0$ and $s_1$ generate $\calP_n$.  
\end{proof}

\begin{remark}
We note that the difference between a general map $[s_0, s_1] \colon \calJ \to \PP$ with invertible sheaf $\calP_n$ and a map $\calJ \to \PP$ which factors as $f \circ \pi$ with $f\colon \PP \to \PP$ is that the coefficients $a_i$ and $b_i$ in the expressions of the sections in Equation \eqref{eq:caz_section}
are in the field $k$ when the map factors, but in general the coefficients are in $R$. 
\end{remark}

%%%%%%%%%%%%%%%%%%%%%

We now look at the data needed to describe a general morphism $\J \to \PP$ and also see what condition pointedness imposes. 

\begin{construction}
\label{con:J_to_P_map} 
A morphism  $f \colon \calJ \to \bbP^1$ is determined by the following data: 
an invertible sheaf $\calL$ on $\J$ and the choice of two global sections $s_0, s_1 \in \Gamma(\J, \calL)$ that generate $\calL$ \cite[Theorem II.7.1]{Hartshorne}. 
Since $\Pic(\J) \cong \bbZ$, the invertible sheaf $\calL$ may be chosen to be either $\calP_n$, $\calQ_n$, or $\calO$. 
We call the integer corresponding to the class of $\calL$ in $\Pic(\J)\cong \bbZ$ the degree of $f$.  

We will now make the assignment $(\calL,s_0,s_1) \mapsto f$ explicit. 
We will study only the case of $\calP_n$, as $\calQ_n$ is handled in the same way by Proposition \ref{prop:span_prop}. 
The case of $\calO$ is discussed later in Section \ref{sec:degree_0_maps}. 

For $\calL = \calP_n$, two generating sections $s_0, s_1 \in \Gamma(\J, \calP_n)$ may be chosen to be of the form
\begin{equation*}
s_0 = a_0\begin{bmatrix}x^n \\ z^n \end{bmatrix}+a_1\begin{bmatrix}y^n \\ w^n\end{bmatrix}\qquad 
s_1 = b_0\begin{bmatrix}x^n \\ z^n \end{bmatrix}+b_1\begin{bmatrix}y^n \\ w^n\end{bmatrix}.
\end{equation*}
Define $D(s_i) = \{ p \in \J \st (s_i)_p \not\in \frakm_p(\calP_n)_p \}$. The map $[s_0, s_1]$ is defined on the open set $D(s_i)$ to map into $U_i = \left\{ [x_0, x_1] \st x_i \neq 0 \right\}$. Here $U_0 \cong \Spec{k[y_1]}$ and $U_1 \cong \Spec{k[y_0]}$, where $y_0 = x_0/x_1$ and $y_1 = x_1/x_0$. 
The map $D(s_i) \to U_i$ is given by the corresponding map of rings $k[y_j] \to \calP_n[s_i^{-1}]$ determined by $y_j \mapsto s_j/s_i$. 
This requires some explanation due to the description of the sheaf $\calP_n$. 
Proposition \ref{prop:local_patching_data} shows that the components of each section $s_i$ describe the section on the open sets $D(x)$ and $D(w)$. 
Hence there are four cases to consider to get a description of the map in concrete terms of affine open sets.
\begin{enumerate}
\item $D(x) \cap D(s_0)$: Here $s_0$ is described by $a_0 x^n + a_1 y^n$ in the ring $R[x^{-1}]$ and $s_1$ is given by $b_0 x^n + b_1 y^n$ in the ring $R[x^{-1}]$. Hence on $D(s_0)$ the corresponding ring map $k[y_1] \to R[x^{-1}, (a_0x^n + a_1y^n)^{-1}]$ is given by $y_1 \mapsto \frac{b_0 x^n + b_1 y^n}{a_0 x^n + a_1y^n}$. 
\item $D(x) \cap D(s_1)$: Here $s_0$ is described by $a_0 x^n + a_1 y^n$ in the ring $R[x^{-1}]$ and $s_1$ is given by $b_0 x^n + b_1 y^n$ in the ring $R[x^{-1}]$. Hence on $D(s_1)$ the corresponding ring map $k[y_0] \to R[x^{-1}, (b_0x^n + b_1y^n)^{-1}]$ is given by $y_0 \mapsto \frac{a_0 x^n + a_1y^n}{b_0 x^n + b_1 y^n}$. 
\item $D(w) \cap D(s_0)$:  Here $s_0$ is described by $a_0 z^n + a_1 w^n$ in the ring $R[w^{-1}]$ and $s_1$ is given by $b_0 z^n + b_1 w^n$ in the ring $R[w^{-1}]$. Hence on $D(s_0)$ the corresponding ring map $k[y_1] \to R[w^{-1}, (a_0z^n + a_1w^n)^{-1}]$ is given by $y_1 \mapsto \frac{b_0 z^n + b_1 w^n}{a_0 z^n + a_1w^n}$. 
\item $D(w) \cap D(s_1)$: Here $s_0$ is described by $a_0 z^n + a_1 w^n$ in the ring $R[w^{-1}]$ and $s_1$ is given by $b_0 z^n + b_1 w^n$ in the ring $R[w^{-1}]$. Hence on $D(s_1)$ the corresponding ring map $k[y_0] \to R[w^{-1}, (b_0z^n + b_1w^n)^{-1}]$ is given by $y_0 \mapsto \frac{a_0 z^n + a_1w^n}{b_0 z^n + b_1 w^n}$. 
\end{enumerate}

This information can be consolidated into the two maps $D(x) \to \PP$ and $D(w) \to \PP$ given in terms of the pair of sections $[a_0 x^n + a_1 y^n , b_0 x^n + b_1 y^n]$ and $[a_0 z^n + a_1 w^n , b_0 z^n + b_1 w^n]$ respectively. Written in this form, we see that a map $\J \to \PP$ given by the invertible sheaf $\calP_n$ with two generating sections $s_0, s_1$ should be interpreted as giving a map to $\PP$ on the open sets $D(x)$ and $D(w)$ according to the first component of the sections $s_0, s_1$ on $D(x)$ and according to the second component of the sections $s_0, s_1$ on $D(w)$.
\end{construction}

\begin{remark}
Recall that $\J$ is pointed at $\jj = (x-1,y,z,w)$ and $\PP$ is pointed at $\infty = [1:0]$. A map $f \colon  \J \to \PP$ is pointed if $f(\jj)=\infty$. If $f = [s_0, s_1]$ with line bundle $\calL$ and generating sections $s_0$, $s_1$, pointedness can be verified by checking that the stalk $s_1(\jj)$ satisfies $s_1(\jj)=0$ in the local ring $\calL_\jj$. For us, it suffices to work on $D(x)$ where our line bundles are trivial, and verify that modulo $\jj$ the section $s_1$ vanishes. 
\end{remark}

We give a concrete criterion for checking pointedness of a map $f \colon  \J \to \PP$ with line bundle $\calP_n$. 
The case of $\calQ_n$ is similar.

\begin{proposition}
\label{prop:pointed_map}
A map $[s_0, s_1]:\J \to \PP$ with invertible sheaf $\calP_n$ and generating sections
\begin{equation*}
s_0 = a_0\begin{bmatrix}x^n \\ z^n \end{bmatrix}+a_1\begin{bmatrix}y^n \\ w^n\end{bmatrix},\qquad 
s_1 = b_0\begin{bmatrix}x^n \\ z^n \end{bmatrix}+b_1\begin{bmatrix}y^n \\ w^n\end{bmatrix}
\end{equation*}
is pointed if and only if $b_0 \in \jj$, i.e., $b_0(\jj)=0$.  
\end{proposition}

\begin{proof}
First, assume the map $[s_0, s_1]$ is pointed. Construction \ref{con:J_to_P_map} gives a description of the map in local coordinates. Note that for $\jj$ to map to $\infty\in U_0$, it is necessary that $\jj \in D(s_0)$. Since $\jj \in D(x)\cap D(s_0)$, the map in local coordinates is obtained by taking $\Spec$ of the ring map $g \colon 
k[y_1] \to R[x^{-1}, (a_0x^n + a_1y^n)^{-1}]$ which is given by $g(y_1) = \frac{b_0 x^n + b_1 y^n}{a_0 x^n + a_1y^n}$.
The condition for pointedness is then that the preimage of $\jj$ under $g$ is the maximal ideal $(y_1)$.  
This is equivalent to the condition that $y_1$ maps into the ideal $(x-1,y,z,w) \subseteq R[x^{-1}, (a_0 x^n + a_1y^n)^{-1}]$. 
By the definition of $g$, the requirement is that $\frac{b_0 x^n + b_1y^n}{a_0 x^n + a_1 y^n} \in (x-1,y,z,w)$, which is equivalent to $b_0 x^n + b_1 y^n \in (x-1,y,z,w)$. 
Since $y \in (x-1,y,z,w)$ and $x$ is invertible, this condition is met when $b_0 \in \jj$. Thus when $[s_0,s_1]$ is pointed, $\jj\in D(s_0)$ and $b_0 \in \jj$. 

Now assume that $b_0 \in \jj$. This implies $\jj \in D(s_0)$, since the sections $s_0, s_1$ generate $(\calP_n)_{\jj}$, and $b_0\in\jj$ implies $s_1(\jj)=0$. 
Here we can use the same construction above, since $\jj \in D(x)\cap D(s_0)$. 
The algebra above shows that when $b_0 \in \jj$ the preimage of $\jj$ under $\se$ is $(y_1)$, i.e., the map $[s_0, s_1]$ is pointed. 
\end{proof}

\begin{proposition}
\label{prop:normalized_map}
Let $f=[s_0, s_1] \colon \J \to \PP$ be a pointed map with invertible sheaf $\calP_n$. 
If $r = s_0(\jj)$, then $\left[\frac{s_0}{r}, \frac{s_1}{r}\right] \colon \J \to \PP$ is a pointed map that is equal to $f$. 
Thus any pointed map with line bundle $\calP_n$ may be represented by a pair of generating global sections $[s_0, s_1]$ where $s_0(\jj)=1$ and $s_1(\jj)=0$.
\end{proposition}

\begin{proof}
Proposition \ref{prop:pointed_map} has established that $s_1(\jj)=0$ and $s_0(\jj) = r$ is a unit. We verify that the maps $[s_0, s_1]$ and $\left[\frac{s_0}{r}, \frac{s_1}{r}\right]$ are equal in local coordinates by Construction \ref{con:J_to_P_map}, where the constants $\frac{1}{r}$ cancel out in every local coordinate chart. 
\end{proof}

\begin{proposition}\label{prop:morphism_ideal_condition}
Let $s_0$ and $s_1$ be the following sections in $\calP_n$
\begin{equation*}
s_0 = a_0\begin{bmatrix}x^n \\ z^n \end{bmatrix}+a_1\begin{bmatrix}y^n \\ w^n\end{bmatrix},\quad 
s_1 = b_0\begin{bmatrix}x^n \\ z^n \end{bmatrix}+b_1\begin{bmatrix}y^n \\ w^n\end{bmatrix}.
\end{equation*}
The sections $s_0,s_1$ generate $\calP_n$ if and only if there exist $U_x,V_x,U_w,V_w \in R$ such that \begin{equation*}
    U_x(x^na_0+y^na_1) + V_x(x^nb_0+y^nb_1) + U_w(z^na_0+w^na_1) + V_w(z^nb_0+w^nb_1) = 1.
\end{equation*}

Employing similar notation, sections $s_0,s_1$ generate $\calQ_n$ if and only if there exist $U_x,V_x,U_w,V_w \in R$ such that \begin{equation*}
    U_x(x^na_0+z^na_1) + V_x(x^nb_0+z^nb_1) + U_w(y^na_0+w^na_1) + V_w(y^nb_0+w^nb_1) = 1.
\end{equation*}
\end{proposition}

\begin{proof}
By Lemma \ref{lem:suffices_to_prove_for_pn}, it suffices to prove this for $\calP_n$. 
Assume $s_0,s_1$ generate $\calP_n$. Then there exist $U,V$ such that $Us_0 + Vs_1 = \begin{bmatrix}x^n \\ z^n \end{bmatrix}$. 
The first component of this identity gives
\begin{equation*}
    \left(U(x^na_0+y^na_1) + V(x^nb_0+y^nb_1)\right) = x^n.
\end{equation*} 
Similarly, there exist $U',V'$ such that $U's_0 + V's_1 = \begin{bmatrix}y^n \\ w^n \end{bmatrix}$. 
This gives the relation 
\begin{equation*}
    \left(U'(z^na_0+w^na_1) + V'(z^nb_0+w^nb_1)\right) = w^n.
\end{equation*}
We need to show that the ideal $(x^n,w^n)$ is the unit ideal. 
However, the equation $1=(x+w)^{2n} = \Sigma_{i=0}^{2n} \binom{2n}{i}x^{2n-i}w^i$ demonstrates that $1$ can be expressed as a linear combination of $x^n$ and $w^n$ over $R$, since each summand $\binom{2n}{i}x^{2n-i}w^i$ can be written as $r_ix^n$ with $r_i = \binom{2n}{i}x^{n-i}w^i \in R$ or $r'_iw^n$ with $r'_i = \binom{2n}{n+i}x^{n-i}w^{i} \in R$ for $i=0,1, \ldots,n$.  
Thus, $(x^n,w^n)$ is the unit ideal. 
Now we assume that there exist elements  $U_x,V_x,U_w,V_w \in R$ such that 
\begin{equation*}
U_x(x^na_0+y^na_1) + V_x(x^nb_0+y^nb_1) + U_w(z^na_0+w^na_1) + V_w(z^nb_0+w^nb_1) = 1.
\end{equation*}
A straight forward computation yields $(U_xx^n + U_wz^n)s_0 + (V_xx^n + V_wz^n)s_1 = \begin{bmatrix}x^n \\ z^n \end{bmatrix}$, and $(U_xy^n + U_ww^n)s_0 + (V_xy^n + V_ww^n)s_1 = \begin{bmatrix}y^n \\ w^n \end{bmatrix}.$ These two elements generate $\calP_n$, thus $[s_0,s_1]$ do as well.  
\end{proof}

For brevity, we write maps $\J \to \PP$ of nonzero degree using the following notation. 

\begin{definition}\label{def:notation_maps}
Let $n$ be a positive integer. 
Let $(a_0,a_1:b_0,b_1)_n$ denote the map $\J \to \PP$ with invertible sheaf $\calP_n$ and generating sections
\begin{equation*}
s_0 = a_0\begin{bmatrix}x^n \\ z^n \end{bmatrix}+a_1\begin{bmatrix}y^n \\ w^n\end{bmatrix},\quad 
s_1 = b_0\begin{bmatrix}x^n \\ z^n \end{bmatrix}+b_1\begin{bmatrix}y^n \\ w^n\end{bmatrix} 
~ \text{with} ~ a_0,a_1,b_0,b_1 \in R.
\end{equation*}
Similarly, let $(a_0,a_1:b_0,b_1)_{-n}$ denote the map $\J \to \PP$ with invertible sheaf $\calQ_n$ and generating sections 
\begin{equation*}
s_0 = a_0\begin{bmatrix}x^n \\ y^n \end{bmatrix}+a_1\begin{bmatrix}z^n \\ w^n\end{bmatrix},\quad 
s_1 = b_0\begin{bmatrix}x^n \\ y^n \end{bmatrix}+b_1\begin{bmatrix}z^n \\ w^n\end{bmatrix} 
~ \text{with} ~ a_0,a_1,b_0,b_1 \in R. 
\end{equation*}
\end{definition}

%%%%%%%%%%%%%%%%%%%%%%%%%

\subsection{Detecting morphisms \texorpdfstring{$\J \to \PP$}{J to P1} via resultants} \label{sec:maps_j_to_p1_via_resultant}

For later purposes, we extend the definition of the resultant to homogeneous polynomials in two variables. 
We collect some further facts about resultants in Appendix \ref{sec:appendix_resultants}. 
The main goal of this subsection is to prove Proposition \ref{prop:resultant_unit_morphism}. 
Motivated by the observation of Remark \ref{rem:from_fPP_to_sections} make the following definition.  

\begin{definition}\label{def:sigma}
Let $R [\alpha, \beta]$ denote the polynomial ring over $R$ in variables  $\alpha$ and $\beta$, 
and let $R [\alpha,\beta]_{(n)}$ denote the $R$-submodule of homogeneous polynomials of degree $n$. 
For every $n\geq 1$, 
the map $\se \colon R [\alpha,\beta]_{(n)} \to \calP_n$, defined by 
$\se(\alpha^{i}\beta^{n-i}) = \begin{bmatrix}x^i y^{n-i} \\ z^i w^{n-i} \end{bmatrix}$ for all $0\leq i \leq n$
is a surjective morphism of $R$-modules. 
\end{definition}

\begin{definition}\label{def:res_for_pairs}
The resultant of a pair of homogeneous polynomials $(F_0, F_1)=\left(\sum_{i=0}^n a_i\alpha^i\beta^{n-i} , \sum_{i=0}^n b_i\alpha^i\beta^{n-i}\right) \in (R [\alpha,\beta]_{(n)})^2 $
is defined to be the resultant of the associated univariate polynomials $(\FF_0,\FF_1) = \left(\sum_{i=0}^n a_i\XX^i , \sum_{i=0}^n b_i\XX^i\right)$ in the indeterminate $\XX$. That is,
\begin{equation*}
\res(F_0,F_1):= \res(\FF_0,\FF_1) = \res \left( \sum_{i=0}^n a_i\XX^i , \sum_{i=0}^n b_i\XX^i \right) = \det \left(\Syl (\FF_0,\FF_1) \right),
\end{equation*}
where $\Syl(\FF_0,\FF_1)$ is the Sylvester matrix of the pair of polynomials $(\FF_0,\FF_1)$ in $R[\XX]$. See Definition \ref{def:sylvesterAndRes} for a definition of the Sylvester matrix.
\end{definition}

\begin{proposition}
\label{prop:resultant_unit_morphism}
Consider a pair $(F_0,F_1)$ of homogeneous polynomials of degree $n\geq 1$ in $R[\alpha, \beta]_{(n)}$. %, denoted by $F_0 = \sum_{i=0}^n a_i \alpha^{n-i}\beta^{i}$ and $F_1 = \sum_{i=0}^n b_i \alpha^{n-i}\beta^{i}$. 
If $\res(F_0, F_1)$ is a unit, then the pair of 
sections $(\se(F_0), \se(F_1))$ generates $\calP_n$ and defines a morphism $[\se(F_0), \se(F_1)] \colon  \J \to \PP$.
\end{proposition}
\begin{proof} 
Consider $(F_0, F_1) = \left(\sum_{i=0}^n a_ix^iy^{n-i} , \sum_{i=0}^n b_ix^iy^{n-i}\right)$ with unit resultant. It suffices to show that $(s_0,s_1) = (\se(F_0), \se(F_1))$ generate $\calP_n$ on the open patches $D(x)$ and $D(w)$. 
On $D(x)$, this requires showing that the ideal $\left(\sum_{i=0}^n a_ix^iy^{n-i} , \sum_{i=0}^n b_ix^iy^{n-i}\right)$ is the unit ideal in $R[x^{-1}]$. 
The ideal is the same as the ideal $\left(\sum_{i=0}^n a_i(\frac{y}{x})^{n-i}, \sum_{i=0}^n b_i(\frac{y}{x})^{n-i}\right)$ which corresponds to a pair of polynomials of degree $n$ in the variable $\frac{y}{x}$. 
By Lemma \ref{lem:reverse_order_resultant}, this pair of polynomials has unit resultant. 
Since the resultant is a unit,  there exists $U_x, V_x \in R[x^{-1}]$ by Lemma \ref{lem:resultant_bezout} giving a B\'ezout relation $U_x\sum_{i=0}^n a_i(\frac{y}{x})^{n-i} + V_x\sum_{i=0}^n b_i(\frac{y}{x})^{n-i} = 1$. 
On $D(w)$ we need to prove that the ideal $\left( \sum_{i=0}^na_iz^iw^{n-i}, \sum_{i=0}^n b_iz^iw^{n-i} \right)$ is the unit ideal in $R[w^{-1}]$. 
The ideal is equal to $\left( \sum_{i=0}^n a_i(\frac{z}{w})^i , \sum_{i=0}^nb_i(\frac{z}{w})^i \right)$. 
This pair of polynomials has unit resultant by assumption. 
By Lemma \ref{lem:resultant_bezout}, unit resultant implies existence of a B\'ezout relation $U_w\sum_{i=0}^n a_i(\frac{z}{w})^i + V_w \sum_{i=0}^nb_i(\frac{z}{w})^i = 1$ in $R[w^{-1}]$. 
This proves that $[s_0,s_1]$ defines a morphism $\J \to \PP$.  
\end{proof}

\begin{remark}\label{rem:connecting_sections_sigma_and_sections_s}
Let $[\sigma_0, \sigma_1] \colon \PP \to \PP$ be a pointed map given by invertible sheaf $\calO(n)$, and sections $\sigma_0 = \sum_{i=0}^n a_ix_0^ix_1^{n-i}$ and $\sigma_1 = \sum_{i=0}^n b_ix_0^ix_1^{n-i}$. 
Then the pair of homogeneous polynomials $F_0 = \sum_{i=0}^n a_i \alpha^{n-i}\beta^{i}$ and $F_1 = \sum_{i=0}^n b_i \alpha^{n-i}\beta^{i}$ in $R[\alpha, \beta]_{(n)}$ has unit resultant. 
By Proposition \ref{prop:resultant_unit_morphism}, the pair $(\se(F_0),\se(F_1))$ defines a morphism $[\se(F_0), \se(F_1)] \colon  \J \to \PP$. 
This morphism is equal to the morphism $[s_0,s_1]$, constructed from $[\sigma_0, \sigma_1]$ in Proposition \ref{prop:compose_caz_with_pi}.  
\end{remark}

\begin{remark}
We note that there exist pairs of polynomials $(F_0, F_1), (F'_0, F'_1)$ such that $(\se(F_0), \se(F_1)) = (\se(F'_0),\se(F'_1))$, while  $\res(F_0, F_1) \neq \res(F'_0, F'_1)$. 
An example is given by $(x\alpha + z\beta, \beta)$ and $(\alpha, \beta)$. 
We calculate 
\begin{equation*}
 (\se(x\alpha + z\beta), \se(\beta)) =
\left(x\begin{bmatrix}x \\ z \end{bmatrix} + z\begin{bmatrix}y \\ w \end{bmatrix}, \begin{bmatrix}y \\ w \end{bmatrix}\right)=
\left(\begin{bmatrix}x \\ z \end{bmatrix}, \begin{bmatrix}y \\ w \end{bmatrix}\right) 
= (\se(\alpha),\se(\beta)).
\end{equation*}
The resultants are 
\begin{equation*}
\res(x\alpha + z\beta, \beta) = x \neq 1 = \res(\alpha, \beta).
\end{equation*}
\end{remark}

%%%%%%%%%%%%%%%%%

\subsection{The pointed naive homotopy relation}\label{sec:naive_htpy_relation}

Naive homotopy theory for schemes is a generalization of the homotopy theory of rings in classical algebra, see \cite{Gersten} for a definition. 
Naive homotopy classes of maps between schemes do not generally have the good properties one expects from a homotopy theory. 
In our case, however, the work of Asok, Hoyois, and Wendt in \cite{AHW2} shows that naive homotopy classes behave sufficiently well.  
We denote by $\Sm_k$ the category of smooth finite type $k$-schemes. 
We denote the set of morphisms between objects $X,Y \in \Sm_k$ by $\Sm_k(X,Y)$. 
If $X$ and $Y$ are pointed, we denote by $\Sm_k(X,Y)_*$ the set of \emph{pointed} morphisms in $\Sm_k$.

\begin{definition} \label{def:elem_naive_htpy}
Let $X$ and $Y$ be smooth schemes finite type $k$-schemes.  
For $a \in k$, let $i_a= \id_X \times a$ be the map obtained by taking the Cartesian product of $\id_X$ and the inclusion map $a \colon \Spec k \to \AAA$ given by the ring map $k[t] \to k$ sending $t$ to $a$.  
An elementary homotopy between two morphisms $f \colon X \to Y$ and $g \colon X \to Y$ is given by a 
morphism $H(T) \colon X \times \AAA \to Y$ satisfying $H(0) = f$ and $H(1) = g$, i.e., $H(0) = H(T) \circ i_0$ and $H(1) = H(T) \circ i_1$. 
We say that $f$ and $g$ are elementarily homotopic and write $f \sim g$.

The relation of morphisms being elementarily homotopic is symmetric and reflexive, but not transitive. To obtain an equivalence relation on the set of morphisms $\Sm_k(X,Y)$, we take the transitive closure of $\sim$.  
That is, we define two morphisms $f,g \in \Sm_k(X,Y)$ to be naively homotopic if there is a finite sequence of elementary homotopies $H_i(T) \colon X\times \AAA \to Y$, for $0\leq i \leq n$ with $H_0(0) = f$, $H_n(1) = g$, and for all $0\leq i < n$ $H_i(1) = H_{i-1}(0)$. We write $f \simeq g$ in this case. 
The relation $\simeq$ is now an equivalence relation on $\Sm_k(X,Y)$, so we can study the set of naive homotopy classes of morphisms from $X$ to $Y$. 
\end{definition}

For our constructions, we will work with pointed maps and pointed naive homotopies.  
We define the latter next. 

\begin{definition}
\label{def:pointed_naive_homotopy}
If $X$ and $Y$ are smooth $k$-schemes, pointed at $k$-points $x$ and $y$ respectively, we say that an elementary homotopy $H(T) \colon X\times \AAA \to Y$ is pointed if the generic point of $\{x\} \times \AAA$ maps to $y$. 
Said another way, the points $x$ and $y$ correspond to morphisms $x \colon \Spec{k} \to X$ and $y \colon \Spec{k} \to Y$, 
and we require that $H(T) \circ (x \times \id_{\AAA}) = y\circ p_1$ where $p_1 \colon \Spec{k} \times \AAA \to \Spec{k}$ is the projection onto the first factor. 

As in the unpointed case, the relation on the set of pointed morphisms $\Sm_k(X,Y)_*$ given by pointed elementary homotopies is not an equivalence relation. We say that pointed morphisms $f, g \in \Sm_k(X,Y)_*$ are naively homotopic, and write $ f \simeq g$, if there is a chain of pointed elementary homotopies from $f$ to $g$. The naive homotopy relation is an equivalence relation on pointed morphisms. We write $[X,Y]\naif = \Sm_k(X,Y)_*/\simeq$ for the set of equivalence classes.
\end{definition}

\begin{remark}\label{rem:def_of_jj'}
For us, the most important case is when $X = \J = \Spec{R}$ with basepoint $\jj = (x-1,y,z,w)$. 
This ideal extends to $\jj' = (x-1,y,z,w) \subseteq R[T]$. 
The condition that a homotopy $H(T) \colon \J \times \AAA \to Y$ be pointed is simply that $H(T)(\jj')=y$, where $y$ is the basepoint of $Y$.
\end{remark}

We now formulate a criterion which will help us to construct homotopies of the form $\J \times \AAA \to \PP$. 
Let $p_1 \colon \J \times \AAA \to \J$ denote the projection to the first factor. 
Similar to Definition 
\ref{def:sigma}, we will use the following notation. 
Let $(R[T])[\alpha, \beta]$ be the polynomial ring over $R[T]$ in variables  $\alpha$ and $\beta$, 
and let $(R[T])[\alpha,\beta]_{(n)}$ denote the $R[T]$-submodule of homogeneous polynomials of degree $n$. 
For every $n\geq 1$, 
we consider the map $\se \colon (R[T])[\alpha,\beta]_{(n)} \to p_1^*\calP_n$, defined by 
$\se(\alpha^{i}\beta^{n-i}) = \begin{bmatrix}x^i y^{n-i} \\ z^i w^{n-i} \end{bmatrix}$ for all $0\leq i \leq n$. 

\begin{proposition}
\label{prop:resultant_unit_morphism_homotopy}
Let $F_0,F_1 \in (R[T])[\alpha, \beta]_{(n)}$ be a pair of homogeneous polynomials of degree $n\geq 1$ over the ring $R[T]$. 
If $\res(F_0, F_1)$ is a unit, then the pair of 
sections $(\se(F_0), \se(F_1))$ generates the line bundle $p_1^*\calP_n$ and defines a morphism $[\se(F_0), \se(F_1)] \colon  \J \times \AAA \to \PP$.
\end{proposition}
\begin{proof}
The proof is analogous to the proof of Proposition \ref{prop:resultant_unit_morphism} after replacing the ring $R$ with $R[T]$. 
\end{proof}

%%%%%%%%%%%%%%%%%%%%%%

\subsection{Morphisms \texorpdfstring{$\calJ \to \ato$}{from J to A2-0}}
\label{sec:degree_0_maps}

We write $\deg  \colon  [\calJ, \PP]\naif \to \Pic(\calJ)\cong \bbZ$ for the map that sends a map $f$ to $f^*\calO(1)$. Our choices thus far set $\deg(\pi) = 1$ and $\deg(\tilde{\pi}) = -1$. 
Write $[\calJ, \PP]\naif_n$ for the set of naive homotopy classes of maps $\calJ \to \PP$ with degree $n$. 
Our goal for this section is to describe the maps $\J \to \PP$ of degree $0$. 
We consider the scheme $\ato = \Spec \left( k[t_0,t_1]
\right) \setminus \{(t_0,t_1)\}$ to be pointed at $(t_0-1, t_1)$ and write $\eta  \colon  \ato \to \PP$  for the Hopf map given by the trivial algebraic line bundle $\calO_{\ato}$ with the choice of sections $\eta_0=t_0$, $\eta_1=t_1$.  
A scheme morphism $\calJ \to \ato$ is given by a morphism $\calJ \to \bbA^2$ that does not have $\{0\}$ in the image. 
Thus, a morphism $\J \to \ato$ is given by a pair $(s_0,s_1) \in R^2$ such that the ideal $(s_0,s_1)$ generates $R$, i.e., there are $U,V \in R$ for which $s_0 U + s_1 V = 1$. 
In other words, a morphism $\calJ \to \ato$ is given by unimodular row $(s_0,s_1)$ in $R^2$.

\begin{proposition}
\label{prop:a2_lift}
Consider a map $f  \colon  \calJ \to \PP$. Then we have $\deg(f) = 0$ if and only if $f$ factors through the Hopf map $\eta  \colon  \ato \to \PP$. 
\end{proposition}

\begin{proof}
First, since $\Pic(\ato)=0$, it follows that any map that factors as $\calJ \to \ato \xrightarrow{\eta} \PP$ has degree 0. 
Second, assume that the morphism $f\colon  \calJ \to \PP$ satisfies $\deg(f)=0$, i.e.,  $f^*\calO(1) = \calO_R$. 
%As we remarked at the beginning of Section \ref{sec:pointed_maps_p1_to_p1_and_j_to_p1}, 
We recall from Construction \ref{con:J_to_P_map} that 
then $f$ is given by global sections $s_0, s_1 \in R= \Gamma(\J,\calO_R)$ that generate $\calO_R$, i.e., the ideal $(s_0, s_1)$ generates $R$. 
As we explained above, this shows that $(s_0,s_1)$ determines a morphism $\calJ \to \ato$. 
Since $\eta$ is given by the trivial bundle, the composition $\J \xrightarrow{(s_0,s_1)} \ato \xrightarrow{\eta} \PP$ corresponds to the trivial algebraic line bundle $\calO_R$ on $\J$ with global sections $s_0,s_1$. 
Thus, the composition $(s_0,s_1)\circ \eta$ equals $f$ which finishes the proof. 
\end{proof}

\begin{corollary}
\label{cor:a2_lift_through_hopf}
Let $f \colon \calJ \to \PP$ be a pointed map of degree 0. 
Then there exists a unique pointed map $f' \colon \J \to \ato$ such that $f = f' \circ \eta$.
\end{corollary}
\begin{proof}
Let $(s_0, s_1)\colon \J \to \ato$ be a factorization of $f$ through the Hopf map. 
Note that $r=s_0(\jj)$ need not be $1$, although $r$ is a unit of $k$. 
Instead, the map $f' = \left(\frac{1}{r} s_0, \frac{1}{r}s_1\right)$ is pointed and satisfies $f = f' \circ \eta$.  

To show uniqueness, let $(s_0', s_1')\colon \J \to \ato$ be another pointed map that factors $f$ through $\eta$. 
That is, we assume $[s_0, s_1] = [s_0',s_1']$. Note that in this case, $D(s_0)=D(s_0')=D(f^*x_0)$ and $D(s_1)=D(s_1')=D(f^*x_1)$.
Working locally in $D(s_1)=D(s_1')$, we have $s_0/s_1 = s_0'/s_1'$ in $R[s_1^{-1}]$ by construction.
We may write $s_0'= c_0s_0$ for $c_0=s_1'/s_1 \in R[s_1^{-1}]$. 
Similarly, in $D(s_0)=D(s_0')$, we obtain $s_1' = c_1s_1$ for $c_1 =s_0'/s_0$. 
In the intersection $D(s_0)\cap D(s_1)$ we have $s_0'/s_1'= s_0/s_1$, which implies $s_0'/s_0= s_1'/s_1$. 
This is exactly the equation $c_0=c_1$. 
The elements $c_1 \in R[s_0^{-1}]$ and $c_0 \in R[s_1^{-1}]$ therefore glue together to an element $c\in R$. 
Hence $c$ satisfies $s_0' = cs_0$ and $s_1' = cs_1$.   
Observe that $c(s_0u' + s_1 v') = 1$, that is, $c \in R^\times=k^\times$. 
The pointedness assumption forces $c(\jj)=1$, hence, $c=1$ with which we conclude $(s_0,s_1)=(s_0',s_1')$.
\end{proof}

\begin{remark}
The previous proposition says, in other words, that a map $\calJ \to \ato$ is equivalent to a unimodular row $(A,B)$ of length two in $R$. 
Furthermore, a pointed map $\calJ \to \ato$ is equivalent to a unimodular row $(A,B)$ of length two in $R$ that also satisfies $A(\jj)=1$ and $B(\jj)=0$. 
\end{remark}

Pointed elementary homotopies between maps of degree $0$ can also be lifted to a pointed elementary homotopy of maps $\J\to\ato$.

\begin{proposition}
\label{prop:lift_homotopies_ato}
Let  $H(T)=[s_0(T), s_1(T)]\colon \calJ\times \AAA \to \PP $ be a pointed elementary homotopy between maps $H(0)$ and $H(1)$ which have degree 0. There is a pointed elementary homotopy $H'(T) \colon \J \times \AAA \to \ato$ between the lifts $H'(0)$ and $H'(1)$. 
\end{proposition}

\begin{proof}
Since $H(0)$ and $H(1)$ have degree 0, the homotopy $H(T)$ is degree 0 too, that is, the line bundle it determines is the trivial one $\calO_{\J\times\AAA}$. 
We can use the two generating global sections $s_0(T), s_1(T) \in R[T]$ to build a map $(s_0(T), s_1(T)) \colon \J \times \AAA \to \ato$. Note that since $s_0(T)$ and $s_1(T)$ generate $R[T]$, there are $u(T), v(T) \in R[T]$ for which $s_0(T)u(T) + s_1(T)v(T)=1$. 
Since $H(T)$ is pointed, $s_1(T)(\jj')= 0$ in $R[T]/\jj'$. 
This implies that $s_0(T)(\jj')u(T)(\jj') = 1$ in $R[T]/\jj'$. The ring $R[T]/\jj'$ is easily seen to be isomorphic to $k[T]$. 
Hence $r=s_0(T)(\jj')$ is a unit of $k[T]$, and the units of $k[T]$ are exactly the units of $k$. 
Thi shows that the map $\left(\frac{1}{r}s_0(T), \frac{1}{r}s_1(T)\right) \colon \J \times \AAA \to \ato$ is a pointed homotopy between $H'(0)$ and $H'(1)$. 
\end{proof}

Let $\SL_2$ denote the affine scheme $\Spec \left( k[a,b,c,d]/(ad-bc-1)\right)$ pointed at the ideal $(a-1,b,c,d-1)$. Intuitively, this is the scheme of $(2\times 2)$-matrices with determinant 1, pointed at the identity matrix. 
Let $(A,B)$ be a unimodular row in $R$. 
That is, there exist $U, V \in R$ for which $AU + BV =1$. 
Thus the data of a map $\calJ \to \ato$ can be used to produce a matrix $\begin{pmatrix}A & -V \\ B & U \end{pmatrix} \in \SL_2(R)$, in other words, a map $\J\to \SL_2$.  

\begin{lemma} \label{lem:liftsofpointedmapstoSL2}
A pointed map $(A,B)\colon \calJ \to \ato$ can be lifted to a pointed map $\begin{pmatrix}A & -V \\ B & U \end{pmatrix}\colon \J\to \SL_2$.
\end{lemma}
\begin{proof}
Let $\begin{pmatrix}A & -V_1 \\ B & U_1 \end{pmatrix}$ be an arbitrary lift of $(A,B)$. 
Note that $A(\jj)=1$, $B(\jj)=0$, and $U_1(\jj)=1$, but $V_1(\jj)=v$ for some $v\in k$.  
For any $d\in R$, we can construct a different lift by setting $U_2=U_1+Bd$ and $V_2=V_1-Ad$.
Set $d=v$. Then $U_2(\jj)=1$, and $V_2(\jj)=0$, so $\begin{pmatrix}A & -V_2 \\ B & U_2 \end{pmatrix}$ is pointed.
\end{proof}

\begin{remark}\label{rem:liftsofmapstoSL2}
The pointed lift of Lemma \ref{lem:liftsofpointedmapstoSL2} is not unique in general. For example, the unimodular row $(1,0)$ lifts to the pointed maps
\begin{equation*}
\begin{pmatrix}
        1 & 0 \\ 0 & 1 
    \end{pmatrix} \quad \text{or} \quad \begin{pmatrix}
        1 & y \\ 0 & 1 
    \end{pmatrix}.
\end{equation*}
\end{remark}

However, we now show that any two pointed lifts of a 
pointed unimodular row $(A,B)$ are naively homotopic.

\begin{lemma}\label{lemma:liftsofmapstoSL2_unique}
Let $(A,B)\colon \calJ \to \ato$ be a pointed map. 
Any two lifts of $(A,B)$ to a pointed map $\J\to \SL_2$ are naively homotopic. 
\end{lemma}
\begin{proof}
Let 
\begin{equation*}
\tilde{f}_i = \begin{pmatrix}
A & -V_i \\
B & U_i
\end{pmatrix}
\end{equation*}
for $i\in\{1,2\}$ be two pointed lifts of $(A,B)$. 
A pointed elementary homotopy between $\tilde{f}_1$ and $\tilde{f}_2$ is given by
\begin{equation*}
%\label{eq:f_thomotopy}
\tilde{f}_t = \begin{pmatrix} A & -(TV_1 + (1-T)V_2)\\ B & TU_1 + (1-T)U_2\end{pmatrix}
\end{equation*}
which proves the claim. 
\end{proof}

\begin{proposition}
\label{prop:lift_homotopies}
Every pointed elementary homotopy 
\begin{equation*}
H(T) = (s_0(T),s_1(T))\colon \calJ\times \AAA \to \ato
\end{equation*}
can be lifted to a pointed elementary homotopy 
\begin{equation*}
\begin{pmatrix}
s_0(T) & -V(T) \\ s_1(T) & U(T) \end{pmatrix}\colon \J\times \AAA\to \SL_2.
\end{equation*}
\end{proposition}

\begin{proof}
Recall $\jj' = (x-1,y,z,w)$ must map to the basepoint for the homotopy $H(T)$ to be pointed. The sections $s_0(T)$ and $s_1(T)$ generate the unit ideal, hence there exist $u(T)$ and $v(T)$ in $R[T]$ for which $s_0(T)u(T) + s_1(T)v(T) = 1$. 
The pointedness assumption gives the relation among ideals $(s_0(T)-1, s_1(T)) \subseteq \jj' \subseteq R[T]$.  

With these data, we construct the matrix 
\begin{equation*}
\begin{pmatrix}
s_0(T) & - v(T) \\ s_1(T) & u(T)
\end{pmatrix}\in
\SL_2(R[T]).
\end{equation*}
This matrix determines a map $\J\times\AAA \to \SL_2$ that lifts the unimodular row $(s_0(T), s_1(T))$ we started with. 
This homotopy need not be a pointed homotopy. However, for any choice of $d(T) \in R[T]$, the matrix 
\begin{equation*}
\begin{pmatrix}
s_0(T) & -v(T)+s_0(T)d(T) \\ 
s_1(T) & u(T)+s_1(T)d(T)
\end{pmatrix}
\end{equation*}
is also a lift of $(s_0(T),s_1(T))$. 
We will now show that, for $d(T) =v(T)$, the map this matrix determines is a pointed homotopy. 
Write $u_2(T)=u(T)+s_1(T)v(T)$ and $v_2(T)=v(T)-s_0(T)v(T)$.
Our assumption that $(s_0(T), s_1(T))$ is pointed gives us $(s_0(T)-1, s_1(T)) \subseteq \jj' \subseteq R[T]$. We must show that $(v_2(T), u_2(T)-1) \subseteq \jj'$ too. Since $(s_0(T)-1)\in \jj'$, we have $v_2(T) = -v(T)(s_0(T)-1)\in \jj'$. Observe that $u_2(T)-1 \in \jj'$ if $u(T)-1 \in \jj'$ since $s_1(T)\in \jj'$. Since $s_0(T)u(T) + s_1(T)v(T) = 1$, it follows that $s_0(T)u(T)-1 \in \jj'$. This can be rewritten as $s_0(T)u(T)-1 = (s_0(T)-1)u(T) + u(T) -1$. Since $s_0(T)-1 \in \jj'$ it follows that $u(T)-1 \in \jj'$ too. 
\end{proof}

\begin{definition}\label{def:def_of_phi}
Let $\phi \colon \SL_2 \to \ato$ be the morphism determined by the ring map $f \colon k[t_0,t_1] \to k[a,b,c,d]/(ad-bc-1)$ given by $f(t_0) = a$ and $f(t_1) = c$. Intuitively, this is the morphism that extracts the first column from a matrix in $\SL_2$. As given, this map has codomain $\mathbb{A}^2$, but it is clear from the relation $ad-bc = 1$ that $\phi$ maps into $\ato$.
\end{definition}

\begin{proposition}\label{prop:sequence_of_bijections}
The maps $\phi \colon \SL_2\to \ato$ and $\eta \colon \ato \to \PP$ induce bijections of naive homotopy classes of pointed maps 
\begin{equation*}
[\J, \SL_2]\naif \xrightarrow{\phi_*} [\J, \ato]\naif \xrightarrow{\eta_*} [\J, \PP]\naif_0.
\end{equation*}
\end{proposition}
\begin{proof}
The map $\phi_*$ is surjective by Lemma \ref{lem:liftsofpointedmapstoSL2}, 
and $\phi_*$ is injective by Lemma \ref{lemma:liftsofmapstoSL2_unique}. 
Corollary \ref{cor:a2_lift_through_hopf} shows that $\eta_*$ is bijective on the level of pointed morphisms. 
This shows that $\eta_*$ is surjective. To show that $\eta_*$ is injective, it suffices to show that a pointed elementary homotopy $H(T) \colon \J \times \AAA \to \PP$ between degree 0 maps lifts to a pointed elementary homotopy $H'(T) \colon \J \times \AAA \to \ato$, which is done in Proposition \ref{prop:lift_homotopies_ato}.
\end{proof}

%%%

\section{Operations on naive homotopy classes of morphisms} 
\label{sec:operations}

\subsection{Group structure on maps of degree \texorpdfstring{$0$}{0}}\label{subsec:operations_deg_zero}

We may now define a binary operation on naive homotopy classes of morphisms $\J\fd\ato $. 
This is analogous to Cazanave's naive sum of pointed rational functions. 
The group structure is obtained by lifting maps $f,g  \colon  \calJ \to \ato$ to $\tilde{f}, \tilde{g}  \colon  \calJ \to \SL_2$, multiplying the two resulting maps using the group structure on $\SL_2$, then mapping back down to $\ato$ via the map $\phi  \colon \SL_2 \to \ato$.

A morphism $f \colon \calJ \to \SL_2$ is equivalent to the data of a matrix 
$M \in \SL_2(R)$. 
A matrix $M \in \SL_2(R)$ corresponds to a pointed morphism if upon evaluation at $\jj$, the resulting matrix is the identity matrix. The set of pointed maps corresponds to a subgroup of $\SL_2(R)$. 
The operation of matrix multiplication respects the naive homotopy relation for pointed maps and therefore defines a group operation on $[\calJ, \SL_2]\naif$, the set of pointed naive homotopy classes of morphisms. It suffices to prove the following proposition, given that the naive homotopy relation is the transitive closure of pointed elementary homotopies.

\begin{proposition}\label{prop:product_of_htpy_is_htpic}
Let $M(T) \in \SL_2(R[T])$ be a pointed elementary homotopy between the matrices $M_0 = M(0)\in \SL_2(R)$ and $M_1 = M(1) \in \SL_2(R)$ corresponding to pointed morphisms. 
Let $M'(T)\in \SL_2(R[T])$ be another elementary homotopy where similar notation is employed. The pointed morphisms corresponding to $M_0\cdot M'_0$ and $M_1 \cdot M'_1$ are elementarily homotopic.
\end{proposition}

\begin{proof}
All that needs to be verified is that the morphism corresponding to the matrix product $M(T)\cdot M'(T)$ is 
pointed. Since both $M(T)$ and $M'(T)$ are pointed, evaluation at $\jj'$ gives the identity matrix in $\SL_2(k[T])$. It's clear then that the product $M(T)\cdot M'(T)$ will evaluate to the identity matrix at $\jj'$ too. 
\end{proof}

\begin{definition}
\label{defn:M-sum}
Consider two pointed naive homotopy classes $[(A_i, B_i)  \colon  \calJ \to \ato]$ for $i = 1, 2$ represented by the unimodular rows $(A_i, B_i) \in R^2$. Pick completions of the unimodular rows to matrices corresponding to pointed maps, as guaranteed by Lemma \ref{lem:liftsofpointedmapstoSL2}:
\begin{equation*}
\begin{pmatrix}A_1 & -V_1\\ B_1 & U_1\end{pmatrix},\begin{pmatrix}A_2 & -V_2\\ B_2 & U_2\end{pmatrix} \in \SL_2(R).
\end{equation*}
We define $[(A_1,B_1)] \oplus [(A_2,B_2)]$ to be the naive homotopy class $[(A_3,B_3)]$ where $(A _3, B_3)$ is the unimodular row obtained from the matrix product
\begin{equation*}
\begin{pmatrix}A_3 & -V_3\\ B_3 & U_3\end{pmatrix}  =
\begin{pmatrix}A_1 & -V_1\\ B_1 & U_1\end{pmatrix} \cdot 
\begin{pmatrix}A_2 & -V_2\\ B_2 & U_2\end{pmatrix}.
\end{equation*}
\end{definition}

\begin{proposition}\label{prop:oplus_well_defined_on_SL_2}
The operation $\oplus$ of Definition \ref{defn:M-sum} is well-defined and gives the set $[\calJ, \ato]\naif$ the structure of a group. 
\end{proposition}

\begin{proof}
We first show that the operation does not depend on the particular completion to a matrix in $\SL_2(R)$. 
Let $M_1$ and $M_1'$ be two pointed completions of $(A_1, B_1)$, and similarly let $M_2$ and $M_2'$ be two pointed completions of $(A_2, B_2)$. 
There are two representatives for the product $[(A_1,B_1)]\oplus[(A_2,B_2)]$ from these choices. 
They are $(A_3, B_3)$, taken from the first column of $M_1\cdot M_2$ and $(A_3', B_3')$, the first column of $M_1'\cdot M_2'$. 
Any two completions of a unimodular row to a matrix in $\SL_2(R)$ are homotopic by Lemma \ref{lemma:liftsofmapstoSL2_unique}, 
hence $[M_1]=[M_1']$ and $[M_2]=[M_2']$ in $[\calJ, \SL_2]\naif$. 
By Proposition \ref{prop:product_of_htpy_is_htpic}, the products $M_1\cdot M_2$ and $M_1'\cdot M_2'$ are  homotopic as maps $\J \to \SL_2$. 
% By the proof of Proposition \ref{prop:lift_homotopies} 
% there is an elementary pointed naive homotopy between $M_1\cdot M_2$ and $M_1'\cdot M_2'$. 
Extracting the first column of this homotopy gives a homotopy between the resulting unimodular rows defining the resulting map. 

We now show that the operation does not depend on the representative of the naive homotopy class chosen. 
Let $(A_1(T), B_1(T))$ and $(A_2(T), B_2(T))$ be pointed elementary homotopies.
These can be completed to matrices $M_1(T) \in \SL_2(R[T])$ and $M_2(T)\in\SL_2(R[T])$ by Proposition \ref{prop:lift_homotopies}. 
The first column of the product $M_1(T)\cdot M_2(T)$ provides the homotopy between the two possible representations of the product. 
We conclude that the operation is well-defined on the set $[\calJ, \ato]\naif$ of pointed naive homotopy classes. 

The identity for the operation is given by the unimodular row $(1,0)  \colon  \calJ \to \ato$. 
Associativity of $\oplus$ follows from the associativity of matrix multiplication.  
Finally, let $(A,B)  \colon  \calJ \to \ato$ be given by the unimodular row $(A, B) \in R^2$ and complete it to a matrix $\begin{pmatrix} A & -V \\ B & U\end{pmatrix} \in \SL_2(R)$ giving a pointed map. 
The inverse of this matrix in $\SL_2(R)$ is the matrix $\begin{pmatrix}U  & V \\ -B & A\end{pmatrix}$, and the first column of this matrix represents the inverse of $(A,B)$ in $[\calJ, \ato]\naif$ for $\oplus$. 
That is, $-[(A,B)] = [(U, -B)]$. 
\end{proof}

\begin{lemma}
\label{lem:sl2_ato_naive_iso}
The map $\phi \colon \SL_2 \to \ato$ induces an isomorphism of groups  
\[
\phi_* \colon [\J, \SL_2]\naif \xrightarrow{\cong} [\J, \ato]\naif.
\]
\end{lemma}
\begin{proof}
The map $\phi_*$ is a group homomorphism by our definition of $\oplus$ on $[\J, \ato]\naif$ in terms of matrix multiplication.  We have shown in Proposition \ref{prop:sequence_of_bijections} that $\phi_*$ is bijective, hence the result. 
\end{proof}

For the next result, we recall that the cogroup structure of $\bbP^1 \cong S^1 \wedge \bbG_m$ in the pointed $\AAA$-homotopy category endows $[\bbP^1,X]^{\AAA}$ with a group operation for any motivic space $X$.  
%First we recall the conventional group operation $\oplus^{\AAA}$ on $[\PP, \PP]^{\AAA}$, the set of maps in the pointed $\AAA$-homotopy category over a field $k$. 
It is a simple exercise to produce an $\AAA$-weak equivalence $\PP \simeq S^1 \wedge \bbG_m$ using the standard covering of $\PP$ by two affine lines with intersection $\bbG_m$. 
The simplicial circle $S^1$ (or some suitable homotopy equivalent model of it, like $\partial \Delta^2$) admits the structure of an h-cogroup, or just a cogroup in the homotopy category. 
Explicitly, the pointed simplicial set $\partial \Delta^2 \simeq S^1$ admits two maps: 
a pinch map $p \colon S^1 \to S^1 \vee S^1$ and an inverse map $S^1 \to S^1$. 
These operations fit into homotopy commutative diagrams that give the expected algebraic properties, like associativity and the definition of the inverse \cite[Chapter 2]{Switzer}. 
These two observations together allow us to define a group operation on $[S^1 \wedge \bbG_m, \PP]^{\AAA}$ as follows. 
Given two maps $f, g \colon S^1 \wedge \bbG_m \to \PP$ in the $\AAA$-homotopy category, the composition below represents the sum $f \oplus^{\AAA} g$ of the maps $f$ and $g$. 
\begin{equation}
\label{eq:group_operation}
\xymatrix{
S^1 \wedge \bbG_m \ar[r]^-{p \wedge 1} & (S^1 \vee S^1 ) \wedge \bbG_m \ar[r]^-{\cong} & (S^1 \wedge \bbG_m) \vee (S^1 \wedge \bbG_m) \ar[r]^-{f \vee g} & \bbP^1
}
\end{equation}
Note that the morphism $f \vee g$ exists by the universal property of wedge sums. 
One must take the time to verify that the operation defined above does indeed make $[\bbP^1, \bbP^1]^{\AAA}$ into a group, but the pleasant properties of the $\AAA$-homotopy category make this doable. 
We refer to the operation $\oplus^{\AAA}$ as the \emph{conventional group structure}.

\begin{definition}
Let $\cc_0 \colon [\J, \ato]\naif \to [\bbP^1, \ato]^{\AAA}$ denote the composition of the natural map  
$\nu_0 \colon [\J, \ato]\naif \to [\J, \ato]^{\AAA}$ and the bijection $(\pi_{\AAA}^*)^{-1} \colon  [\J, \ato]^{\AAA} \to [\bbP^1, \ato]^{\AAA}$ that is given by the inverse of the bijection $\pi_{\AAA}^*$. 
We note that $\nu_0$ is a bijection by Proposition \ref{prop:appendix_main_result} since $\ato$ is $\AAA$-naive. 
\end{definition}

\begin{theorem}
\label{thm:degree_zero_iso} 
The map $\cc_0$ is an isomorphism of groups between the group $[\J, \ato]\naif$ with operation $\oplus$ and the group $[\bbP^1, \ato]^{\AAA}$ with the conventional group operation. 
\end{theorem}

\begin{proof}
Let $\cc'_0 \colon [\J, \SL_2]\naif \to [\bbP^1, \SL_2]^{\AAA}$ denote the composition of the canonical map $[\J, \SL_2]\naif \to [\J, \SL_2]^{\AAA}$ and the bijection $[\J, \SL_2]^{\AAA} \to [\bbP^1, \SL_2]^{\AAA}$ which is given by the inverse of the bijection $\pi_{\AAA}^*$.
Since both groups $[\bbP^1,\SL_2]^{\AAA}$ and $[\bbP^1,\ato]^{\AAA}$ inherit their operation from the cogroup structure of $\bbP^1$, 
the $\bbA^1$-weak equivalence $\SL_2 \xrightarrow{\phi} \ato$ induces an isomorphism of groups $\phi_* \colon [\PP, \SL_2]^{\AAA} \to [\PP, \ato]^{\AAA}$.
By Lemma \ref{lem:sl2_ato_naive_iso} the map $\phi_* \colon [\J, \SL_2]\naif \to [\J, \ato]\naif$ is a group isomorphism. 
We then have the following commutative diagram. 
\begin{align*}
\xymatrix{
[\J, \ato]\naif \ar[r]^-{\cc_0} & [\PP, \ato]^{\AAA} \\
[\J, \SL_2]\naif \ar[r]_-{\cc'_0} \ar[u]_-{\cong}^-{\phi_*} & [\bbP^1,\SL_2]^{\AAA} \ar[u]^-{\cong}_-{\phi_*} 
}
\end{align*}

Hence, in order to establish that $\cc_0$ is a group isomorphism, it suffices to show that $\cc'_0$ is a group isomorphism. 
Since $\J$ is affine and $\SL_2$ is $\AAA$-naive by \cite[Theorem 4.2.1]{AHW2}, we know that $\cc'_0$ is a bijection by Proposition \ref{prop:appendix_main_result}. 
Hence it suffices to show that $\cc'_0$ is a group homomorphism.

Again, because $\SL_2$ is $\AAA$-naive, the canonical map $[\J, \SL_2]\naif \to [\J, \SL_2]^{\AAA}$ is a bijection by Proposition \ref{prop:appendix_main_result}. 
This bijection is a group isomorphism because the operation on both sets is defined using the same construction, that is, the sum of two maps is given by
\begin{equation*}
\J \xrightarrow{\Delta} \J \times \J \xrightarrow{f \times g} \SL_2 \times \SL_2 \xrightarrow{m} \SL_2, 
\end{equation*}
where $m \colon \SL_2 \times \SL_2 \to \SL_2$ is the multiplication on $\SL_2$. 
In other words, the group structure is induced by the group object structure on $\SL_2$.

Similarly, the set $[\PP, \SL_2]^{\AAA}$ also obtains the structure of a group using that $\SL_2$ is a group object in the pointed $\AAA$-homotopy category.
The Eckmann--Hilton argument given in \cite[Proposition 2.25]{Switzer} can be applied in this scenario to show that this group structure coincides with the conventional group structure, see also \cite[Proposition 2.2.12]{arkowitz2011introduction}. 
Hence we may assume that the group operation on $[\PP, \SL_2]^{\AAA}$ is induced by the group object structure on $\SL_2$. 
Combining these observations shows that the composition 
\[
[\J, \SL_2]\naif \to [\J, \SL_2]^{\AAA} \to [\PP,\SL_2]^{\AAA}
\]
is a group homomorphism. 
This is the map $\cc'_0$ which proves the assertion.  
\end{proof}

\begin{corollary}\label{cor:J_to_ato_naive_is_abelian}
The group $[\J, \ato]\naif$ is abelian.
\end{corollary}
\begin{proof}
Since $[\J, \ato]\naif $ is isomorphic to $[\PP, \SL_2]^{\AAA}$, the Eckmann--Hilton argument shows that this group is abelian.
\end{proof}

\begin{remark}\label{rem:Morel_shows_K1MW_gives_degree_0_maps}
Morel shows in \cite[\S 7.3]{Morel12} that the group $[\PP,\ato]^{\AAA}$ is isomorphic to $K_1^{\mathrm{MW}}(k)$, the first Milnor--Witt $K$-theory group of the field $k$ (see Definition \ref{def:Milnor_Witt_K} below). 
In short, the computation \cite[Theorem 7.13]{Morel12} and the $\AAA$-weak equivalence between $\SL_2$ and $\ato$ gives $\pi_1^{\AAA}(\ato) \cong \KMW_2$. The contraction of this sheaf evaluated at $\Spec{k}$ then computes $[\PP, \ato]^{\AAA}$:
\begin{equation*}
[\PP, \ato]^{\AAA} \cong \pi_1^{\AAA}(\ato)_{-1}(\Spec{k}) \cong (\KMW_2)_{-1}(\Spec{k}) \cong K^{\mathrm{MW}}_1(k).
\end{equation*}
Hence our results show that there is an isomorphism $[\J, \SL_2]\naif \cong K_1^{\mathrm{MW}}(k)$. 
We make this isomorphism explicit in Section \ref{sec:Milnor_Witt_K_theory_and_degree_0}. 
\end{remark}

\begin{remark} \label{q:explicit_matrix_commutation}
For any two pointed matrices $M,M' \in \SL_2(R)$, which represent pointed morphisms $\J \to \SL_2$, there is a chain of elementary homotopies connecting $M\cdot M'$ and $M' \cdot M$. 
We do not know of a general algorithm to construct this chain of homotopies explicitly. 
\end{remark}

The following explicit naive homotopies will be used in the later sections.

\begin{lemma}
\label{lem:ABeqUV}
Consider a matrix $M = 
\begin{pmatrix}
A & -V \\ 
B & U 
\end{pmatrix} 
\in \SL_2(R)$.
Then $M$ and $(M^{-1})^T$ are naively homotopic. 
Thus, the unimodular rows $(A,B)$ and $(U,V)$ are naively homotopic. 
\end{lemma}
\begin{proof}
Consider the matrix $H = 
\begin{pmatrix}
1-T^2 & -T \\
T(2-T^2) & 1-T^2 
\end{pmatrix}\in\SL_2(R[T])$. 
The matrix $H$ defines an unpointed homotopy from the identity matrix to 
$\begin{pmatrix}
    0 & -1 
    \\1 & 0 
\end{pmatrix}$. 
It is straightforward to verify that the product $H M H^{-1}$ is a pointed homotopy between $M$ and $(M^{-1})^T$ as claimed.  
\end{proof}

\begin{lemma}
\label{lem:scaling_lemma}
Consider a matrix 
$\begin{pmatrix}
A & -V \\ 
B & U 
\end{pmatrix} \in SL_2(R)$ and let $u \in k^\times$. Then there is an elementary homotopy
\begin{equation}\label{eq:elementary_htpy_unimodulr_row_and_usqaure}
\begin{pmatrix}
A & -V \\ 
B & U 
\end{pmatrix} \simeq \begin{pmatrix}
A & -\frac{1}{u^2}V \\ 
u^2B & U 
\end{pmatrix}.
\end{equation}
Thus, the unimodular row $(A,B)$ is naively homotopic to the unimodular row $(A,u^2B)$.
\end{lemma}
\begin{proof}
The matrix on the right-hand side of Equation \eqref{eq:elementary_htpy_unimodulr_row_and_usqaure} can be written as the following product    
\begin{equation*}
\begin{pmatrix}
A & -\frac{1}{u^2}V \\ u^2B & U 
\end{pmatrix}
= \begin{pmatrix}
\frac{1}{u} & 0 \\0 & u
\end{pmatrix}\begin{pmatrix}
A & -V \\ B & U 
\end{pmatrix}\begin{pmatrix}
 u & 0 \\0 & \frac{1}{u}
\end{pmatrix}.
\end{equation*}
The diagonal matrices can be decomposed to a product of elementary matrices, which are all homotopic to the identity.
\end{proof}

%%%

\subsection{Action of degree \texorpdfstring{$0$}{0} maps on degree \texorpdfstring{$n$}{n} maps}\label{sec:action_of_0_maps_on_n_maps}

Recall that we write $[\calJ, \PP]\naif_n$ for the set of naive homotopy classes of maps $\calJ \to \PP$ with degree $n$. 
We define a group action of $[\calJ, \ato]\naif\cong [\calJ, \SL_2]\naif\cong [\calJ, \PP]\naif_0$ on $[\calJ, \PP]\naif_n$ for all $n \neq 0$. We start by first defining an operation on actual morphisms, and then show that the operation respects the naive homotopy equivalence relation.

\begin{definition} \label{def:group_action_SL2}
Let $M \colon \calJ \to \SL_2$ be a morphism with corresponding matrix 
\begin{equation*}
\begin{pmatrix}A & -V 
    \\B & U 
\end{pmatrix}
\end{equation*}
and consider a map $[s_0, s_1]  \colon  \calJ \to \PP$ determined by $n \in \bbN$, the algebraic line bundle $\calP_n$ or $\calQ_n$, and generating global sections $s_0, s_1$. 

We define $M \oplus [s_0, s_1] \colon \calJ \to \PP$ to be the morphism determined by the same algebraic line bundle with the generating global sections $M\oplus [s_0,s_1] = [As_0-Vs_1, Bs_0+Us_1]$ which are obtained from the following matrix multiplication
\begin{equation*}
\begin{pmatrix}A & -V 
    \\B & U 
\end{pmatrix}
\begin{pmatrix}
    s_0 \\
    s_1
    \end{pmatrix}
      = 
      \begin{pmatrix}
      As_0 - Vs_1 \\ Bs_0 + U s_1
      \end{pmatrix}. 
\end{equation*}
\end{definition}

\begin{proposition}\label{prop:oplus_is_a_morphism}
Given a map $[s_0, s_1] \colon \calJ \to \PP$ with algebraic line bundle $\calL$ (either $\calP_n$ or $\calQ_n$) and a map $M \colon \calJ \to \SL_2$, the construction $M \oplus [s_0, s_1]$ is a morphism from $\calJ $ to $\bbP^1$. 
If both maps are pointed, the result is also pointed. Furthermore, the operation of Definition \ref{def:group_action_SL2} defines a left group action of $[\J,\SL_2]\naif$ on the set $\Sm_k(\J,\PP)$. 
\end{proposition}

\begin{proof}
The morphism $M \colon \calJ \to \SL_2$ is described by a matrix 
\begin{equation*}
\begin{pmatrix} A & -V \\ B & U \end{pmatrix} \in \SL_2(R).
\end{equation*}
We observe that $U(As_0 -Vs_1) +V(Bs_0 + Us_1) = s_0$ and $-B(As_0 -Vs_1) +A(Bs_0 + Us_1) = s_1$. 
By assumption, the sections $s_0$, $s_1$ generate the algebraic line bundle $\calL$. 
Hence the pair of sections $As_0 -Vs_1$, $Bs_0 + Us_1$ generate $\calL$ as well. This proves the first assertion.

That the map $[s_0, s_1]$ is pointed means that $s_1 \in \jj \subseteq R$, or equivalently, $s_1(\jj)=0$ in $R/\jj$.
That $M$ is pointed means $M(\jj)$ is the identity matrix. To verify $M\oplus [s_0,s_1]$ is pointed, we must check that $B(\jj)s_0(\jj)+U(\jj)s_1(\jj)=0$, but this is clear as $B(\jj)=0$ and $s_1(\jj)=0$ from our assumptions. 

The fact that the operation is a left group action follows from the associativity of matrix multiplication and the definition of the group structure on maps $\calJ \to \SL_2$.
\end{proof}

The next theorem employs the notation of Definition \ref{def:notation_maps} for morphisms $\J \to \PP$. 

\begin{theorem}
\label{thm:deg0action_factorization}
Let $f \colon \calJ \to \PP$ be a map of degree $n$. Then there exists a matrix $M\in\SL_2(R)$ such that $f = M \oplus (1,0:0,1)_n $.
\end{theorem}
\begin{proof}
We prove the assertion for $n>0$. 
The proof for $n <0$ is similar. 
Let $f = (a_0,a_1:b_0,b_1)_n$ with the notation introduced in Definition \ref{def:notation_maps}.  
For $c,c',d,d'\in R$, consider the matrix 
\begin{equation}\label{eq:matrix_to_get_f=M_plus_pi}
M = \begin{pmatrix}
a_0 + y^n c +w^n c' & a_1 -x^nc - z^nc' \\ b_0 - y^nd - w^nd' & b_1 +x^n d +z^n d'
\end{pmatrix}.
\end{equation}
The matrix $M$ can be written as the sum
\begin{equation*}
    M = \begin{pmatrix}
        a_0 & a_1 \\
        b_0 & b_1
    \end{pmatrix} + \begin{pmatrix}
        y^n c +w^n c' & -x^nc - z^nc' \\
        - y^nd - w^nd' & x^n d +z^n d'
    \end{pmatrix}.
\end{equation*}
By definition of $\oplus$ and the notation in Definition \ref{def:notation_maps} we compute 
\begin{align*}
& \begin{pmatrix}
        y^n c +w^n c' & -x^nc - z^nc' \\
        - y^nd - w^nd' & x^n d +z^n d'
\end{pmatrix} \oplus (1,0:0,1)_n \\
= & 
\begin{pmatrix}
        y^n c +w^n c' & -x^nc - z^nc' \\
        - y^nd - w^nd' & x^n d +z^n d'
\end{pmatrix}
\begin{pmatrix}
    \begin{bmatrix}
    x^n \\ z ^n
    \end{bmatrix} \\ 
    \begin{bmatrix}
    y^n \\ w ^n
    \end{bmatrix}
    \end{pmatrix}\\
= & 
\begin{bmatrix}
(y^n c + w^n c') \begin{bmatrix}
    x^n \\ z ^n
    \end{bmatrix}  + (-x^nc - z^nc') \begin{bmatrix}
    y^n \\ w ^n
    \end{bmatrix} \\
    (- y^nd - w^nd') \begin{bmatrix}
    x^n \\ z ^n
    \end{bmatrix} + 
    (x^n d +z^n d')\begin{bmatrix}
    y^n \\ w ^n
    \end{bmatrix}
    \end{bmatrix}
    = \begin{bmatrix} 0 \\ 0 \end{bmatrix}
\end{align*}

The last equality follows from the relations 
$y^n \begin{bmatrix}
    x^n \\ z ^n
    \end{bmatrix} = x^n \begin{bmatrix}
    y^n \\ w ^n
    \end{bmatrix}$ and  $w^n \begin{bmatrix}
    x^n \\ z ^n
    \end{bmatrix} = z^n \begin{bmatrix}
    y^n \\ w ^n
\end{bmatrix}.$ 
% $xw-yz=0$ in $R$.  
%
Hence, for any choice of $c,c',d,d'\in R$, we have 
\begin{align*}
M\oplus (1,0:0,1)_n & = M\oplus\left[
\begin{bmatrix}
    x^n \\ z ^n
    \end{bmatrix}, 
    \begin{bmatrix}
    y^n\\ w^n
    \end{bmatrix}
    \right] 
    =  
    \left[
a_0 \begin{bmatrix}
    x^n \\ z ^n
    \end{bmatrix} + a_1  
    \begin{bmatrix}
    y^n\\ w^n
    \end{bmatrix}, 
    b_0 \begin{bmatrix}
    x^n \\ z ^n
    \end{bmatrix} + b_1  
    \begin{bmatrix}
    y^n\\ w^n
    \end{bmatrix}
    \right] \\
    & = (a_0,a_1:b_0,b_1)_n.
    \end{align*}
We now show that there always exist $c,c',d,d'$ such that $M \in SL_2(R)$. 
The determinant of $M$ is given by the formula  
\begin{equation*}
\det(M) = a_0b_1-a_1b_0 + c(x^nb_0 + y^n b_1) + c'(z^nb_0 + w^nb_1) + d(x^na_0 + y^na_1) + d'(z^na_0 + w^na_1).  
\end{equation*} 
Since $(a_0,a_1:b_0,b_1)_n$ determines a morphism of schemes, 
it follows from Proposition \ref{prop:morphism_ideal_condition} that the ideal $I:=(x^na_0+y^na_1,z^na_0+w^na_1,x^nb_0+y^nb_1,z^nb_0+w^nb_1)$ is the unit ideal. 
Thus $1-a_0b_1-a_1b_0$ is in $I$, and there exist elements $c,c',d,d'$ such that $\det(M)=1$. 
This shows there exists $M \in SL_2(R)$ satisfying $f = M \oplus (1,0:0,1)_n $ as desired.
\end{proof}

Theorem \ref{thm:deg0action_factorization} implies that the action $\oplus$ is transitive. More concretely, the theorem has the following consequence: 

\begin{corollary}
Let $f,g \colon \calJ \to \PP$ be two morphisms of degree $n$. 
Then there exists a matrix $M\in\SL_2(R)$ such that $M \oplus f= g$. 
\end{corollary}
\begin{proof}
By Theorem \ref{thm:deg0action_factorization}, there exist $M'$ and $M''$ such that $M'\oplus (1,0:0,1)_n = f$ and  $M'' \oplus (1,0:0,1)_n  = g$. 
The desired matrix is given by $M = M''\cdot (M')^{-1}$. 
\end{proof}

\begin{remark}\label{rem:M_may_not_be_pointed}
The $\SL_2(R)$-matrix $M$ constructed in the proof of Theorem \ref{thm:deg0action_factorization} is not always pointed, even if the map $f$ we started with is pointed. 
For example, following the construction for the map $f=(1,1:0,1)_1$ yields the matrix $M=\begin{pmatrix}
1 & 1 \\ 0 & 1
\end{pmatrix}$ which is not pointed, since $M(\jj)$ is not the identity matrix. 
\end{remark}

Remark \ref{rem:M_may_not_be_pointed} shows that we have to improve our argument in order to get an action on pointed homotopy classes. 
We will now prove the necessary adjustments. 

\begin{proposition}
\label{prop:deg0action_pointed_factorization_new}
Let $f \colon  \calJ \to \PP$ be a pointed map of degree $n\neq 0$.
Then there is a pointed naive homotopy between $f$ and a map of the form $M \oplus (1,0:0,1)_n $ for some pointed matrix $M \in \SL_2(R)$. 
\end{proposition}

\begin{proof}
Let $f = (a_0,a_1:b_0,b_1)_n$, where we may assume $a_0(\jj)=1$ by Proposition \ref{prop:normalized_map}. 
By Theorem \ref{thm:deg0action_factorization} we can find a matrix $M' \in \SL_2(R)$ such that $M' \oplus (1,0:0,1)_n = f$. 
However, $M'$ may not be pointed.  
We can replace $M'$ with a pointed map $M$ as follows. 
Assuming $M'$ is of the form given in Equation \eqref{eq:matrix_to_get_f=M_plus_pi} we get $b_1(\jj)+d(\jj)=1$. Moreover, this implies that there is an element $e\in k$ such that  $M'(\jj) = \begin{pmatrix}
1 & e \\ 0 & 1
\end{pmatrix}$ and $e = a_1(\jj) - c(\jj)$.  
Define $M$ to be 
$M = \begin{pmatrix}
1 & -e \\ 0 & 1
\end{pmatrix}M'$. 
We compute $M \oplus (1,0:0,1)_n  = (a_0 - eb_0,a_1-eb_1:b_0,b_1)_n$. 
The assertion now follows from the fact that the morphism $(a_0 - Teb_0,a_1-Teb_1:b_0,b_1)_n$ is a pointed homotopy between $M \oplus (1,0:0,1)_n$ and $f$, where $T$ denotes the parameter for the homotopy. 
\end{proof}

\begin{corollary}
\label{cor:decomposition}
Let $f,g \colon \calJ \to \PP$ be two pointed morphisms of degree $n$. 
There exists a pointed map $M \colon \J \to \SL_2$ such that $M \oplus f$ is pointed naively homotopic to $g$. 
\end{corollary}

Since the line bundle corresponding to the morphisms $M\oplus f$ and $f$ are the same by definition of $\oplus$, it is clear that $\oplus$ preserves degrees of morphisms. 
Hence we make the following definition. 

\begin{definition}\label{def:group_action_homotopy}
Let $[(A,B)]\in [\J,\ato]\naif\cong [\J, \PP]\naif_0$ be a pointed naive homotopy class represented by the map with unimodular row $(A,B)$ in $R$. 
Let $[f] \in [\J, \PP]\naif_n$ be a pointed naive homotopy class of degree $n$ with $n\neq 0$ represented by a pointed morphism $f \colon \calJ \to \PP$. 
We define $[(A,B)]\oplus [f] := [M\oplus f]$ where $M$ is a completion of $(A,B)$ to a matrix in $\SL_2(R)$ corresponding to a pointed map.   
\end{definition}

It remains to show that $\oplus$ respects naive homotopy classes. 

\begin{theorem}\label{thm:operation_descends_to_htpy_classes}
The operation of Definition \ref{def:group_action_homotopy} is well-defined and for each $n\in \Z$ provides the set $[\J, \PP]\naif_n$ with a left-action by the group $[\J,\ato]\naif$. 
\end{theorem}
\begin{proof}
First, consider a pointed map $f=[s_0,s_1]$. We show that $[(A,B)]\oplus [f]$ is independent of the choice of completion of $(A,B)$
 to a matrix in $\SL_2(R)$. 
So let 
\begin{align*}
M=\begin{pmatrix}A & -V\\ B & U\end{pmatrix}, ~ 
M'=\begin{pmatrix}A & -V'\\ B & U'\end{pmatrix}
\end{align*}
be two completions to matrices in $\SL_2(R)$ which correspond to pointed maps. 
Then the naive homotopy $H(T)=\begin{pmatrix}A & -(TV+(1-T)V' \\ B & TU+(1-T)U'\end{pmatrix}$ is pointed independently of $T$, 
and $H(T)\oplus f$ is a pointed homotopy between $M\oplus f$ and $M'\oplus f$. 

Now we show that $[(A,B)]\oplus [f]$ is independent of the choice of the representing unimodular row $(A,B)$. 
Suppose we have a pointed elementary homotopy $(A(T), B(T))$ between two unimodular rows. 
Proposition \ref{prop:lift_homotopies} shows that we can lift it to a pointed elementary homotopy $M(T) = \begin{pmatrix}A(T) & -V(T) \\ B(T) & U(T) \end{pmatrix}\in \SL_2(R[T])$. 
Then $M(T)\oplus f$ is a pointed homotopy between $(A,B)\oplus f$ and $(A',B')\oplus f$. 
Now we consider a unimodular row $(A,B)$ and let $M = \begin{pmatrix}A & -V\\ B & U\end{pmatrix}$ be a lift to a matrix in $\SL_2(R)$. 
Let $f_0, f_1 \colon \J \to \PP$ be two pointed morphisms which are homotopic via a pointed naive homotopy. 
Let $f(T) \colon \J \times \AAA \to \PP$ be a pointed naive homotopy. 
We let $\calL'$ denote the line bundle $\J \times \AAA$ which corresponds to the morphism $f(T)$. 
We define the map $H(T):= M \oplus f(T) \colon \J \times \AAA \to \PP$ with the same algebraic line bundle $\calL'$ on $\J \times \AAA$ and global sections 
\[
\begin{pmatrix}A & -V\\ B & U\end{pmatrix} \cdot \begin{pmatrix} s_0(T) \\ s_1(T)\end{pmatrix} 
= \begin{pmatrix}A s_0(T) -V s_1(T) \\ 
B s_0(T) + Us_1(T) \end{pmatrix}.
\]
We note that $H(T)$ thus defined is in fact a morphism $\J \times \AAA \to \PP$, since we have $U(As_0(T) -Vs_1(T)) + V(Bs_0(T) + Us_1(T)) = s_0(T)$, and $-B(As_0(T) -Vs_1(T)) +A(Bs_0(T) + Us_1(T)) = s_1(T)$. 
By assumption, the sections $s_0(T)$, $s_1(T)$ generate the line bundle $\calL'$.  
Hence $(As_0(T) -Vs_1(T), Bs_0(T) + Us_1(T))$ generate $\calL'$ as well. 
This shows that $H(T)$ defines a morphism.  
We now verify that $H(T)$ is pointed by showing $Bs_0(T) + Us_1(T) \in \jj'$ . 
Pointedness of $[s_0(T), s_1(T)]$ means that $s_1(T)(\jj') = 0$ in $R[T]/\jj'$. 
Pointedness of $(A,B)$ means $M(\jj)$ is the identity matrix. We calculate 
\begin{equation*}
B(\jj')s_0(T)(\jj') + U(\jj')s_1(T)(\jj') =0\cdot s_0(T)(\jj') + 1 \cdot 0 = 0
\end{equation*}
which completes the verification.
This shows that $\oplus$ is independent of the choice of representatives in both naive homotopy classes and completes the proof of the first assertion. 
The second assertion then follows from Proposition  \ref{prop:oplus_is_a_morphism}. 
\end{proof}

\begin{remark}\label{rem:variation_of_operation}
There are several variations to the operation given in Definition \ref{def:group_action_homotopy} that produce valid group actions. 
For $M \in \SL_2(R)$, the operation in Definition \ref{def:group_action_SL2} is given by the matrix multiplication $M\cdot (s_0 \; s_1)^T$. 
We could have taken equally well 
either $M^T \cdot (s_0\; s_1)^T$ or $M^{-1} \cdot (s_0\; s_1)^T$, although this would give a right-action rather than a left-action on maps. 
Up to homotopy, the latter two choices in fact agree, since $M^T$ is homotopic to  $M^{-1}$ by Lemma \ref{lem:ABeqUV}.  
Thus there are two natural choices for this action, one of which applies the inverse operation to the morphism in $\SL_2(R)$ before acting. 
In Appendix \ref{sec:evidence} we will use real realization to check which of these operations can represent the group operation on $[\J, \bbP^1]^{\AAA}$ induced from Morel's group structure on $[\bbP^1, \bbP^1]^{\AAA}$ via $\pi_{\AAA}^*$. 
In fact, in Examples \ref{ex:realization_F_acts_on_pi_tilde} and \ref{ex:group_action_justification} we show that only the choice of Definitions \ref{def:group_action_SL2} and \ref{def:group_action_homotopy} can be compatible.  
\end{remark}

%%%

\section{The group structure on \texorpdfstring{$[\J,\PP]\naif$}{[J,P1]}}
\label{sec:proposed_group_structure}

\subsection{The definition of the group structure}

In this section we define an explicit group structure on $[\J,\PP]\naif$. 
We will then discuss some alternative approaches and open questions. 

\begin{definition}\label{def:minus_pi_and_n_pi}
Let $-[\id]$ denote the additive inverse of $[\id \colon \PP \to \PP]$ under the conventional group structure on $[\PP, \PP]^{\AAA}$. 
Define $-\pi \colon  \J \to \PP$ to be a morphism which represents the $\AAA$-homotopy class $-[\id \colon  \PP \to \PP] \in [\PP, \PP]^{\AAA}$ under the bijection $\cc \colon  [\J, \PP]\naif \to [\PP, \PP]^{\AAA}$ of Equation \eqref{eq:jpn_ppa}. 
More generally, for any integer $n$, let $n\pi$ denote a morphism $n\pi \colon  \J \to \PP$ which represents the $\AAA$-homotopy class $n[\id \colon  \PP \to \PP]$ under the bijection $\cc \colon  [\J, \PP]\naif \to [\PP, \PP]^{\AAA}$.
\end{definition}

We are now ready to define a group operation on $[\J, \PP]\naif$.  

\begin{definition}
\label{def:full_group}
Let $f \colon \J \to \PP$ and $g \colon \J \to \PP$ be morphisms of degrees $n$ and $m$ respectively. 
By Corollary \ref{cor:decomposition} there are degree $0$ maps $f_0 \colon \J \to  \PP$ and $g_0 \colon \J \to \PP$ for which $f \simeq f_0 \oplus n\pi$ and $g \simeq g_0 \oplus m\pi$. 
We define the sum of $[f]$ and $[g]$ to be 
\begin{align*}
[f] \oplus [g] & = ([f_0] \oplus [n\pi] ) \oplus ([g_0] \oplus [m \pi]) \\
& = ([f_0] \oplus [g_0]) \oplus [(n+m) \pi].
\end{align*}
The term $[f_0] \oplus [g_0]$ is calculated by matrix multiplication via Definition \ref{defn:M-sum}. 
The group action of Definition \ref{def:group_action_homotopy} is used to compute $(f_0\oplus g_0)\oplus (n+m)\pi$.
\end{definition}

\begin{remark}
It follows from Theorem  \ref{thm:operation_descends_to_htpy_classes} that the operation $\oplus$ of Definition \ref{def:full_group} is well-defined. 
We also note that, for $n>0$, the proofs of Theorem \ref{thm:deg0action_factorization} and Proposition \ref{prop:deg0action_pointed_factorization_new} may be used to write down a concrete algorithm to find a map $f_0$ such that $f \simeq f_0 \oplus n\pi$ for any degree $n$ map $f$. 
\end{remark}

\begin{remark}\label{rem:construction_of_npi}
For $n > 0$, we may construct morphisms $n\pi$ by using Cazanave's group operation on morphisms $[\PP,\PP]\naif$ and lift it to an element in $[\J,\PP]\naif$. A recursive description of the maps $n\pi$ for $n>0$ can be given as follows: 
Set $G_0=1$ and $G_1=\begin{bmatrix}
        x \\ z
    \end{bmatrix}$. 
For $n>0$, we define $G_{n+1}$ recursively by setting 
\begin{align*}
        G_{n+1}= \begin{bmatrix}
        x \\ z
    \end{bmatrix}\cdot G_n - \begin{bmatrix}
        y^2 \\ w^2
    \end{bmatrix}\cdot G_{n-1}
    \end{align*}
where we recall that multiplication of sections is induced by component-wise multiplication in $R$. 
For $n>0$, the morphism $n\pi$ is given by sections \begin{equation*}
    \left[ G_n, \begin{bmatrix}
        y \\ w
\end{bmatrix}\cdot G_{n-1}\right].
\end{equation*}
We note that  $[(1,0:0,1)_n]$ is in general not equal to $[n\pi]$ for $n>1$. 
We will explain this observation in Remark \ref{rem:KW2_computed_Xn/1} using Morel's motivic Brouwer degree and the work of Kass and Wickelgren.  
\end{remark}

We are now ready to prove the following important result. 

\begin{theorem}\label{thm:group_structure}
The operation $\oplus$ turns the set $[\J, \PP]\naif$ into an abelian group. 
Moreover, there is an isomorphism of groups  $\ff \colon \left([\J, \PP]\naif,\oplus\right) \xrightarrow{\cong} \left([\PP, \PP]^{\AAA},\oplus^{\AAA}\right)$. 
\end{theorem}

\begin{proof}
We observe that the set $\{[n\pi] : n \in \Z\}$ inherits the structure of an abelian group from $\Z$. 
In Definition \ref{def:full_group} we construct the group $\left([\J, \PP]\naif,\oplus\right)$ as the direct product of the two groups $\{[n\pi] : n \in \Z\}$ and $[\J,\ato]\naif$. 
Both are abelian by Corollary \ref{cor:J_to_ato_naive_is_abelian}. 
This implies the first assertion. 

By definition of the operation $\oplus$, the group $\left([\J, \PP]\naif, \oplus\right)$ 
fits into the short exact sequence displayed in the top row of Diagram \eqref{eq:diagram_group_morphism} below. 
By the work of Morel in \cite[\S7.3]{Morel12}, the group $\left([\PP, \PP]^{\AAA}, \oplus^{\AAA}\right)$ 
fits into  the short exact sequence displayed in the bottom row. 
\begin{align}\label{eq:diagram_group_morphism}
\xymatrix{
1 \ar[r] & [\calJ, \ato]\naif  \ar[r] \ar[d]^-{\cong}_-{\xi_0} & [\J,\PP]\naif \ar@{.>}[d]_-{\ff} \ar[r]^-{\deg} &  \bbZ \ar[d]^-{\cong}  \ar[r] & 1  \\
1 \ar[r] & [\PP, \ato]^{\AAA} \ar[r] & [\PP, \PP]^{\AAA} \ar[r]^-{\deg} &  \bbZ \ar[r] & 1 
}
\end{align}
By Theorem \ref{thm:degree_zero_iso}, the vertical map $\cc_0$ on the left-hand side is an isomorphism. 
The vertical map $q$ on the right-hand side is an isomorphism as well. 
We define $\ff$ to be the unique group homomorphism satisfying $\ff([\pi]) = [\id]$ and $\ff([f_0]) = \cc_0([f_0])$ for all $[f_0] \in [\J,\ato]\naif$.  
The diagram commutes by our definition of $\ff$. 
Since $\cc_0$ and the right-hand vertical map in Diagram \eqref{eq:diagram_group_morphism} are isomorphisms, $\ff$ is an isomorphism of groups as well by the five-lemma. 
\end{proof}

%%%%%%%%%%%%%%%%%%%%%%%%%

\subsection{Open questions and potential alternative approaches}

Recall the map $\cc \colon [\J,\PP]\naif \to [\PP,\PP]^{\AAA}$ which is the composite of the canonical map 
$\nu \colon [\J,\PP]\naif \to [\J,\PP]^{\AAA}$  and the inverse of the induced map $\pi_{\AAA}^* \colon [\PP,\PP]^{\AAA} \to [\J,\PP]^{\AAA}$. 
Unfortunately, Theorem \ref{thm:group_structure} does not imply that the bijection $\cc$ is a group isomorphism. 
However, since $\cc$ restricts to an isomorphism on the subgroups $[\J,\ato]\naif$ and $\{[n\pi] \st n \in \bbZ \}$, we do believe that the bijection $\cc$ is a group isomorphism, which we state as a conjecture below. 

\begin{conjecture}\label{conj:main_conjecture}
The bijection $\cc \colon [\J,\PP]\naif \to [\PP,\PP]^{\AAA}$ is a group isomorphism and equals $\ff$. 
\end{conjecture}

%%%

One obstacle to prove Conjecture \ref{conj:main_conjecture} is that, for $n<0$, we do not know which morphism $\J \to \PP$ is sent to $n[\id]$ under $\cc$. 
In particular, we do not know which morphism $\J \to \PP$ is mapped to the motivic homotopy class $-[\id : \PP \to \PP]$. 
A potential candidate for $-\pi$ may be the map $\tilde{\pi}=(1,0:0,-1)_{-1}$ determined by the line bundle $\calQ_1$ and generating sections 
\begin{equation*}
s_0 = \begin{pmatrix}
x \\ y
\end{pmatrix}
\text{ and } ~ 
s_1 = -\begin{pmatrix}
z \\ w
\end{pmatrix}. 
\end{equation*}

\begin{question}\label{question:tilde_pi} 
Is $\tilde{\pi}$ the inverse of $\pi$ for $\oplus$, i.e., is $\tilde{\pi}$ naively homotopic to $-\pi$? 
\end{question}

In Appendix \ref{sec:evidence} we present further evidence for Conjecture \ref{conj:main_conjecture}. 
We use the real realization functor for fields $k \subset \R$ and Morel's theorem which states that the signature of the motivic Brouwer degree equals the topological Brouwer degree under real realization. 
This provides a potential  
obstruction to the compatibility of $\oplus^{\AAA}$ and the action of $[\J,\ato]\naif$ on $[\J,\PP]\naif$. 
We then compute  concrete examples and show that other choices for the action of $[\J,\ato]\naif$ on $[\J,\PP]\naif$ are not compatible with $\oplus^{\AAA}$, while our choice of operation in Definition \ref{def:group_action_homotopy} is compatible with $\oplus^{\AAA}$ after real realization in the chosen examples.

In Proposition \ref{prop:image_of_100u} we show that the naive homotopy class of $\pi$ is mapped to the class $\left(\langle 1\rangle,1\right)$ in $\GW(k) \times_{k^\times/k^{\times 2}} k^\times$ as expected if $\cc$ is a group homomorphism.  
Based on the computations in Appendix  
\ref{sec:evidence} 
we prove in Theorem \ref{thm:class_of_tilde_pi} that the image of $[\tilde{\pi}]$ under the motivic Brouwer degree is the class $-\langle 1\rangle$ in $\GW(k)$. 
This brings us very close to a positive answer to Question \ref{question:tilde_pi}. 
We are, however, not able to compute the resultant, i.e., the image of $\tilde{\pi}$ in $k^{\times}$.

%%%%%%%%%%%%%%%%

We end this section with comments on potential alternative approaches: 

\begin{remark}\label{rem:J_to_J_maps}
Since $\pi$ is an $\AAA$-weak equivalence, it induces a bijection $\pi_* \colon [\J,\J]\naif \to [\J,\PP]\naif$. 
Hence there is a bijection between $[\PP,\PP]^{\AAA}$ and the set of pointed naive homotopy classes $[\J,\J]\naif$.  
In \cite[page 31]{cazanave2009theorie} Cazanave speculates whether $[\J,\J]\naif$ can be used to study the group structure on $[\PP,\PP]^{\AAA}$.  
A morphism $\J \to \J$ corresponds to a ring homomorphism $R \to R$, or equivalently, the data of a $(2\times 2)$-matrix with entries in $R$ and with trace $1$ and determinant $0$. 
For every map $f \colon \J \to \PP$ we can find a map $F \colon \J \to \J$ such that $f=\pi \circ F$. 
We will refer to such a map $F$ as a lift of $f$. 
There is a particularly nice way to construct a lift in the case $f \colon \J \to \PP$ has degree $0$. 
Assume that $f$ is given by a unimodular row $(A,B)$. 
Let $U, V \in R$ be such that $\begin{pmatrix}
        A & -V \\ B & U
    \end{pmatrix}$ has determinant $1$.  
Then $F$ is given by the matrix  $\begin{pmatrix}
        AU & BU \\ AV & BV
    \end{pmatrix}$
which has trace $1$ and determinant $0$. 
Composing the map with $\pi$ yields the $\J \to \PP$ map given by either $[AU : BU]$ or $[AV: BV]$, whenever they are defined, which coincides with the map corresponding to the unimodular row $(A,B)$. 
If $f$ has non-zero degree, there is also a concrete procedure to find a lift of $f$, which we leave to the reader. 

Since morphisms $\J \to \J$ can be represented by matrices, it may seem plausible that one can find a suitable operation on $[\J,\J]\naif$ which may help to describe the group $([\PP,\PP]^{\AAA},\oplus^{\AAA})$. 
However, neither addition nor multiplication of matrices equip the set $[\J,\J]\naif$ with an operation which is compatible with the conventional group structure on $[\PP,\PP]^{\AAA}$. 
We have verified in examples that composition of maps in $[\J,\J]\naif$ descends to the operation $\circ$ on $[\PP,\PP]\naif$ of \cite[Definition 4.5]{Caz}. 
As pointed out in \cite[Remark 4.7]{Caz} the latter does not distribute over the conventional group structure on $[\PP,\PP]^{\AAA}$. 
We were not able to make a reasonable guess which other operation on $[\J,\J]\naif$ might work. 
We have therefore not pursued this path further. 
\end{remark}

\begin{remark}\label{rem:J_wedge_J}
An alternative approach to construct a group structure on $[\J,\J]\naif$ may be the following.   
One can hope to construct a cogroup structure on $\J$. 
However, this is not so easy, even though Asok and Fasel have done much of the work to make this possible. 
In \cite{AsokFasel}, Asok and Fasel give an explicit construction of a smooth scheme $\widetilde{\J \vee \J}$ that is $\AAA$-weak equivalent to the wedge sum $\J \vee \J$. 
We have constructed an explicit map
$\J \to \widetilde{\J \vee \J}$ that conjecturally represents the pinch map $\PP \to \PP \vee \PP$. 
We also have a candidate for a map representing the inverse map $\J \to \PP$, 
but unfortunately both of these claims have proven too difficult to verify. 
We therefore decided not to include this construction in this paper. 
\end{remark}

%%%%%%%%%%%%%%%%%%%%%%%%%%%%%%%%

\section{Compatibility with Cazanave's monoid structure}\label{sec:compatibility_with_Cazanaves_monoid}

The goal of this section is to show that the map 
\[
\pi_{\mathrm{N}}^* \colon \left([\PP,\PP]\naif,\oplus\naif\right) \to \left([\J,\PP]\naif,\oplus \right)
\]
is a morphism of monoids, 
where $\oplus\naif$ denotes the monoid operation defined by Cazanave in \cite{Caz}. 
We will achieve this goal in Theorem \ref{thm:pi*_is_a_monoid_morphism}. 

%%%

\subsection{Compatibility with certain degree \texorpdfstring{$0$}{0} maps} \label{sec:computations_for_degree_0_maps}

We first study an important family of degree $0$ morphisms and their compatibility with $\oplus$, $\oplus\naif$, and $\pi_{\mathrm{N}}^*$. 

\begin{definition}\label{def:g_uv}
For $u, v \in k^\times$, we write $g_{u,v}$ for the pointed morphism $\J\to \ato$  
given by the unimodular row $\left(x+ \frac{v}{u}w, (u-v)y\right)$ in $R$. 
This unimodular row can be completed to the $\SL_2(R)$-matrix 
\begin{equation*}
m_{u,v}=\begin{pmatrix}
x+\frac{v}{u}w & \frac{u-v}{uv}z \\ (u-v)y & x + \frac{u}{v}w
\end{pmatrix}.
\end{equation*}
\end{definition}

We now prove some basic properties of the maps $g_{u,v}$ which will be necessary to show that $\pi_{\mathrm{N}}^*$ is a monoid morphism. 

\begin{lemma}\label{lemma:guv_plus_gvu}
For all $u, v, s \in k^\times$, we have the identity  $g_{u,v} \oplus g_{v,s} = g_{u,s}$. 
In particular, we have $g_{u,v} \oplus g_{v,u} = (1,0)$ and $g_{u,v}\oplus g_{v,1}=g_{u,1}$. 
\end{lemma}
\begin{proof}
A direct computation, using $xw=yz$, shows
\begin{equation*}
m_{u,v}\cdot m_{v,s}
=
\begin{pmatrix}
x^2+ \frac{uv+vs}{uv}xw+ \frac{s}{u}w^2&
\frac{uv-sv}{usv}xz+ \frac{-vs+uv}{usv}zw \\
(u-s)xy+(u-s)wy &
x^2+ \frac{uv+sv}{sv}xw+ \frac{u}{s}w^2
\end{pmatrix},
\end{equation*}
and since $x+w=1$,
this simplifies to the matrix $m_{u,s}$.
Then $g_{u,v} \oplus g_{v,u} = g_{u,u}=(1,0)$ and $g_{u,v}\oplus g_{v,1}=g_{u,1}$ are special cases for respectively $s=u$ and $s=1$. 
\end{proof}

\begin{lemma}
\label{lem:square_scale}
Let $u,v,c \in k^\times$. 
Then $[g_{u,v}] = [g_{c^2u,c^2v}]$.
\end{lemma}
\begin{proof}
By Lemma \ref{lem:scaling_lemma}, we have 
\begin{align*}
g_{u,v} = \left( x+\frac{v}{u}w, (u-v)y \right) 
\simeq \left( x+\frac{v}{u}w, c^2(u-v)y \right) = g_{c^2u,c^2v}.
\end{align*}
\end{proof}

\begin{lemma}
\label{lem:gu1_summing}
Let $u,v \in k^{\times}$ and let $v$ be a square. 
Then we have $[g_{u,1}] \oplus [g_{v,1}] =  [g_{uv,1}]$. 
\end{lemma}
\begin{proof}
Lemma \ref{lem:square_scale} and Lemma \ref{lemma:guv_plus_gvu} imply 
\begin{equation*}
    g_{u,1}\oplus g_{v,1} \simeq g_{u,1}\oplus g_{1,1/v}=g_{u,1/v} \simeq g_{uv,1},
\end{equation*} 
and hence the claim. 
\end{proof}

Now we study the relationship of the maps $g_{u,v}$ with $\oplus$ and $\pi_{\mathrm{N}}^*$. 
We adopt the following notation from Cazanave \cite{Caz}:  
For $u\in k^\times$, we identify a rational function $X/u$ in the indeterminate $X$ with the morphism $\PP \to \PP$ defined by $[x_0:x_1] \mapsto [x_0:ux_1]$. 
We then have 
\begin{align*}
\pi = \pi_{\mathrm{N}}^*\left(\frac{X}{1}\right) 
= \left[
\begin{bmatrix}
        x\\z
    \end{bmatrix}, 
    \begin{bmatrix}
        y\\w
    \end{bmatrix}
\right] \quad 
\text{and} \quad  
\pi_{\mathrm{N}}^*\left(\frac{X}{u}\right) 
= \left[
\begin{bmatrix}
        x\\z
    \end{bmatrix}, 
    u \begin{bmatrix}
        y\\w
    \end{bmatrix}
\right] 
\end{align*}
as maps $\J \to \PP$. 
Our next goal is to prove Proposition \ref{prop:gu1_oplus_acts_on_caz}. 
To do so, we need some preparation. 

\begin{lemma}\label{lem:g_u1_oplus_pi}
For every $u \in k^\times$, we have $g_{u,1} \oplus \pi = \pi_{\mathrm{N}}^*\left(\frac{X}{u}\right)$. 
\end{lemma}
\begin{proof}
A direct computation using the facts that $x\begin{bmatrix}
        y\\w
    \end{bmatrix}=y\begin{bmatrix}
        x\\z
    \end{bmatrix}$ and $z\begin{bmatrix}
        y\\w
    \end{bmatrix}=w\begin{bmatrix}
        x\\z
    \end{bmatrix}$ shows
\begin{align*}
g_{u,1} \oplus \pi & 
= \begin{pmatrix}
x+\frac{1}{u}w & \frac{u-1}{u}z \\ 
(u-1)y & x + uw
\end{pmatrix}
\oplus
\left[
\begin{bmatrix}
        x\\z
    \end{bmatrix}, 
    \begin{bmatrix}
        y\\w
    \end{bmatrix}
\right] \\
& = 
\left[
\left(x+\frac{1}{u}w\right) \begin{bmatrix}
    x \\ z 
    \end{bmatrix}  
    + \left(\frac{u-1}{u}z\right) \begin{bmatrix}
    y \\ w 
    \end{bmatrix} ,
    (u-1)y \begin{bmatrix}
    x \\ z 
    \end{bmatrix} + 
    \left(x + uw\right)\begin{bmatrix}
    y \\ w 
    \end{bmatrix}
\right] \\
& = \left[
\begin{bmatrix}
        x\\z
    \end{bmatrix}, 
    u \begin{bmatrix}
        y\\w
    \end{bmatrix}
\right]
\end{align*}
and hence the result by definition of the maps involved. 
\end{proof}

\begin{lemma}\label{lem:g_uv_oplus_pi_v}
For all $u,v \in k^\times$, we have the identity 
\begin{align*}
g_{u,v} \oplus \pi_{\mathrm{N}}^*\left(\frac{X}{v}\right) = \pi_{\mathrm{N}}^*\left(\frac{X}{u}\right).
\end{align*}
\end{lemma}
\begin{proof}
Using Lemmas \ref{lemma:guv_plus_gvu}, \ref{lem:g_u1_oplus_pi}, and Definition \ref{def:full_group} we get 
\begin{equation*}
g_{u,v} \oplus \pi_{\mathrm{N}}^*\left(\frac{X}{v}\right)  = (g_{u,v} \oplus g_{v,1}) \oplus \pi = g_{u,1} \oplus \pi = \pi_{\mathrm{N}}^*\left(\frac{X}{u}\right)
\end{equation*}
and hence the result. 
\end{proof}

%%%

We are now ready to prove a key result for the compatibility of $\pi_{\mathrm{N}}^*$ with the monoid operations. 

\begin{proposition}
\label{prop:gu1_oplus_acts_on_caz}
Let $u \in k^{\times}$ and $f \colon \PP \to \PP$ be a pointed morphism. 
Then there is a pointed naive homotopy
\begin{align*}
\pi_{\mathrm{N}}^*\left(\frac{X}{u} \oplus\naif f \right) \simeq 
g_{u,1}  \oplus \left( \pi_{\mathrm{N}}^*\left(\frac{X}{1} \oplus\naif f \right) \right). 
\end{align*}
\end{proposition}
\begin{proof} 
We begin by clearly describing the maps under consideration before defining our naive homotopy. 
A pointed morphism $f \colon \PP \to \PP$ of degree $n$ can be expressed as a rational function $\frac{A}{B} = \frac{X^n + a_{n-1}X^{n-1}+\ldots + a_0}{b_{n-1}X^{n-1} + \ldots + b_0}$ where $\res(A,B)$ is a unit by Proposition \ref{prop:Caz_thm_pointed_P1_endo}. If we define
\begin{equation*}
(F_0,F_1)=(\alpha^n + a_{n-1}\alpha^{n-1}\beta + \ldots +a_0\beta^n, b_{n-1}\alpha^{n-1}\beta + \ldots + b_0\beta^n) \in 
\left(R [\alpha,\beta]_{(n)}\right)^2,
\end{equation*}
then $\pi_N^*(f) = [\se(F_0) ,\se(F_1)]$ with respect to the invertible sheaf $\calP_n$ by Proposition \ref{prop:resultant_unit_morphism} and Remark \ref{rem:connecting_sections_sigma_and_sections_s}. Note $\res(F_0, F_1) = \res(A,B)$ is a unit. Write $f_0 = s(F_0)$ and $f_1 = s(F_1)$, so that $\pi_N^*(f) = [f_0, f_1]$ to simplify the notation. 

The map $\pi_{\mathrm{N}}^*\left(\frac{X}{u} \oplus\naif f \right)$ is described as the following matrix product by the calculation of Cazanave \cite[Example 3.3]{Caz} when interpreted in the notation of Remark  \ref{rem:connecting_sections_sigma_and_sections_s}:
\begin{equation*}
   \pi_{\mathrm{N}}^*\left(\frac{X}{u} \oplus\naif f \right) =\begin{pmatrix}\begin{bmatrix}  x \\ z \end{bmatrix} & -\frac{1}{u} \begin{bmatrix}  y \\ w \end{bmatrix} \\ u \begin{bmatrix}  y \\ w \end{bmatrix} & 0 \end{pmatrix}
  \cdot \begin{pmatrix}
        f_0 \\ f_1
    \end{pmatrix}.
\end{equation*}
Note that here we use the isomorphism $\calP_n \otimes \calP_1 \xrightarrow{\cong} \calP_{n+1}$ of  Proposition \ref{prop:tensor_of_P1_is_Pn} 
to identify the product of a pair of column vectors with its image in $\calP_{n+1}$, where we recall that multiplication of sections is induced by component-wise multiplication of elements in $R$. 
The pair of generating sections that determine the morphism $g_{u,1} \oplus \left( \pi^*\left(\frac{X}{1} \oplus\naif f \right) \right)$ is given by the matrix product 
\begin{equation*}
   g_{u,1}  \oplus \left( \pi_{\mathrm{N}}^*\left(\frac{X}{1} \oplus\naif f \right) \right) =\begin{pmatrix}
       x + \frac{1}{u}w & \frac{u-1}{u}z \\ (u-1)y & x+uw
   \end{pmatrix} \cdot 
   \begin{pmatrix}\begin{bmatrix}  x \\ z \end{bmatrix} & - \begin{bmatrix}  y \\ w \end{bmatrix} \\  \begin{bmatrix}  y \\ w \end{bmatrix} & 0 \end{pmatrix} \cdot 
   \begin{pmatrix}
        f_0 \\ f_1
    \end{pmatrix}.
\end{equation*}
Note that we have the following equality of matrix products 
\begin{align*}
    \begin{pmatrix}
       x + \frac{1}{u}w & \frac{u-1}{u}z \\ (u-1)y & x+uw
   \end{pmatrix} \cdot 
   \begin{pmatrix}\begin{bmatrix}  x \\ z \end{bmatrix} & -\begin{bmatrix}  y \\ w \end{bmatrix} \\  \begin{bmatrix}  y \\ w \end{bmatrix} & 0 \end{pmatrix} & = \begin{pmatrix}\begin{bmatrix}  x \\ z \end{bmatrix} & -(x+\frac{1}{u}) \begin{bmatrix}  y \\ w \end{bmatrix} \\ u \begin{bmatrix}  y \\ w \end{bmatrix} & (1-u)y \begin{bmatrix}  y \\ w \end{bmatrix}\end{pmatrix} \\
   & = \begin{pmatrix}\begin{bmatrix}  x \\ z \end{bmatrix} & -\frac{1}{u} \begin{bmatrix}  y \\ w \end{bmatrix} \\ u \begin{bmatrix}  y \\ w \end{bmatrix} & 0 \end{pmatrix} 
   \cdot 
   \begin{pmatrix}
       1 & \frac{u-1}{u}y \\ 0 & 1
   \end{pmatrix}.
\end{align*}
This motivates the definition of the matrix $h(T)$ defined by the following product. 
\begin{align*}
    h(T) & = \begin{pmatrix}\begin{bmatrix}  x \\ z \end{bmatrix} & -\frac{1}{u} \begin{bmatrix}  y \\ w \end{bmatrix} \\ u \begin{bmatrix}  y \\ w \end{bmatrix} & 0 \end{pmatrix}
    \cdot 
    \begin{pmatrix}
       1 & \frac{u-1}{u}yT \\ 0 & 1
   \end{pmatrix} 
   \cdot 
   \begin{pmatrix}
        f_0 \\ f_1
    \end{pmatrix} \\
    & =\begin{pmatrix}
    f_0\begin{bmatrix}  x \\ z \end{bmatrix}  +\frac{u-1}{u}yT f_1 \begin{bmatrix}  x \\ z \end{bmatrix} - \frac{1}{u}f_1 \begin{bmatrix}  y \\ w \end{bmatrix} \\
    u f_0  \begin{bmatrix}  y \\ w \end{bmatrix} + (u-1)yT f_1  \begin{bmatrix}  y \\ w \end{bmatrix} 
    \end{pmatrix}
\end{align*}

We claim that the rows of the matrix $h(T)$ provide generating sections of the line bundle $p_1^*\calP_{n+1}$ on $\J \times \AAA$, where $p_1 \colon \J \times \AAA \to \J$ is the projection to the first factor. 
This implies that $h$ defines a morphism $h \colon \J \times \AAA \to \PP$, and that $h$ is a pointed homotopy from $h(0)=\pi_{\mathrm{N}}^*\left(\frac{X}{u} \oplus\naif f \right) $ to $h(1)=
g_{u,1}  \oplus \left( \pi_{\mathrm{N}}^*\left(\frac{X}{1} \oplus\naif f \right) \right)$. 

Define $H(T)$ by
\begin{align*}
    H(T) & = \begin{pmatrix}\alpha & -\frac{1}{u} \beta \\ u \beta & 0 \end{pmatrix}
    \cdot 
    \begin{pmatrix}
       1 & \frac{u-1}{u}yT \\ 0 & 1
   \end{pmatrix} 
   \cdot 
   \begin{pmatrix}
        F_0 \\ F_1
    \end{pmatrix} \\
    & =\begin{pmatrix}
    F_0\alpha +\frac{u-1}{u}yT F_1 \alpha - \frac{1}{u}F_1 \alpha \\
    u F_0  \beta + (u-1)yT F_1 \beta
    \end{pmatrix}.
\end{align*}
Since $\res(F_0,F_1)$ is a unit, Lemma \ref{lem:resultant_conservation_elementary_action} implies that the pair of polynomials $\left(F_0 + \frac{u-1}{u}yTF_1,F_1\right)$ over the ring $R[T]$ has unit resultant.  
Lemma \ref{lem:resultantconservation} then implies that the resultant %of $H(T)$ 
\begin{align*}
\res\left(F_0\alpha +\frac{u-1}{u}yT F_1 \alpha - \frac{1}{u}F_1 \alpha, 
u F_0  \beta + (u-1)yT F_1 \beta \right)
\end{align*}
is a unit as well. 
Finally, Proposition \ref{prop:resultant_unit_morphism_homotopy} shows that $h$ is a morphism and thus the desired pointed naive homotopy. 
\end{proof}

To give a concrete example of the homotopy constructed in the proof of Proposition \ref{prop:gu1_oplus_acts_on_caz}, we look at the  special case $f=X/1$:

\begin{example}\label{ex:pi_Xu_plus_pi_X1}
For every $u\in k^{\times}$, the morphism $H$ defined by 
\begin{equation*}
H =  \left[ \begin{bmatrix}
        x^2 \\ z^2
    \end{bmatrix} + T\frac{u-1}{u}y \begin{bmatrix}
        xy \\ zw
    \end{bmatrix} - \left(x+\frac{1}{u}w\right)
    \begin{bmatrix}
        y^2 \\ w^2 
    \end{bmatrix},  u \begin{bmatrix}
        xy \\ zw
    \end{bmatrix} + (T(u-1)-(u-1)y)\begin{bmatrix}
        y^2 \\ w^2 
    \end{bmatrix}\right]
\end{equation*}
is a homotopy between %$H(0)=g_{u,1}\oplus 2\pi$ 
$H(0)=g_{u,1}\oplus\pi_{\mathrm{N}}^*\left(\frac{X}{1} \oplus\naif \frac{X}{1} \right)$ 
and 
$H(1)=\pi_{\mathrm{N}}^*\left(\frac{X}{u} \oplus\naif \frac{X}{1} \right)$. 
\end{example}

%%%

\subsection{The map \texorpdfstring{$\pi_{\mathrm{N}}^*$}{0} is a monoid morphism}\label{sec:compatibility_with_oplusN}

We will now prove that the map $\pi_{\mathrm{N}}^* \colon \left([\PP,\PP]\naif,\oplus\naif\right) \to \left([\J,\PP]\naif,\oplus \right)$ induced by $\pi$ is a morphism of monoids. 

\begin{lemma}\label{lemma:pi_n_times_X1}
We have 
\begin{align*}
\pi_{\mathrm{N}}^*\left(\frac{X}{1} \oplus\naif \cdots \oplus\naif \frac{X}{1} \right) \simeq 
\pi \oplus \cdots \oplus \pi, 
\end{align*}
where there are $n$ summands on both sides. 
\end{lemma}
\begin{proof}
Since $\nu_{\PP}\left(\frac{X}{1}\right) = \id_{\PP}$ and $\nu_{\PP}$ is a morphism of monoids by \cite[Proposition 3.23]{Caz}, we have the equality $n[\id] = \left[\frac{X}{1} \oplus\naif \cdots \oplus\naif \frac{X}{1}\right]$ in $[\PP, \PP]^{\AAA}$. 
Thus, $\pi_{\mathrm{N}}^*\left(\frac{X}{1} \oplus\naif \cdots \oplus\naif \frac{X}{1} \right)$ is naively homotopic to $n\pi$. 
By definition, $[\pi] \oplus \cdots \oplus [\pi] = [n\pi]$, hence the result follows. 
\end{proof}

\begin{proposition}\label{prop:pi*_Xu1_to_X_un_and_oplus}
For $u_1,\ldots,u_n \in k^{\times}$ we have 
\begin{align*}
\pi_{\mathrm{N}}^*\left(\frac{X}{u_1} \oplus\naif \frac{X}{u_2} \oplus\naif \cdots \oplus\naif \frac{X}{u_n} \right) 
\simeq \pi_{\mathrm{N}}^*\left(\frac{X}{u_1}\right) \oplus \pi_{\mathrm{N}}^*\left(\frac{X}{u_2}\right) \oplus \cdots \oplus \pi_{\mathrm{N}}^*\left(\frac{X}{u_n} \right). 
\end{align*}
\end{proposition}
\begin{proof}
Since both $\oplus\naif$ and $\oplus$ are commutative and  $\oplus$ is associative, we may apply Proposition  \ref{prop:gu1_oplus_acts_on_caz} and Lemma \ref{lemma:pi_n_times_X1}  to compute
\begin{align*}
\pi_{\mathrm{N}}^*\left(\frac{X}{u_1} \oplus\naif \cdots \oplus\naif \frac{X}{u_n} \right) & \simeq 
g_{u_1,1} \oplus \pi_{\mathrm{N}}^*\left(\frac{X}{1} \oplus\naif \frac{X}{u_2} \oplus\naif \cdots \oplus\naif \frac{X}{u_n}\right) \\
& \simeq  
g_{u_1,1} \oplus g_{u_2,1} \oplus \pi_{\mathrm{N}}^*\left(\frac{X}{1} \oplus\naif \frac{X}{1} \oplus\naif \frac{X}{u_3} \oplus\naif \cdots \oplus\naif \frac{X}{u_n}\right) \\
& \simeq 
g_{u_1,1} \oplus \cdots \oplus g_{u_n,1} \oplus n\pi \\ 
& \simeq 
(g_{u_1,1} \oplus \pi) \oplus  (g_{u_2,1} \oplus \pi) \oplus \cdots \oplus (g_{u_n,1} \oplus \pi) \\ 
& \simeq \pi_{\mathrm{N}}^*\left(\frac{X}{u_1}\right) \oplus \cdots \oplus \pi_{\mathrm{N}}^*\left(\frac{X}{u_n}\right). 
\end{align*}    
The final step follows from Lemma \ref{lem:g_u1_oplus_pi}. 
\end{proof}

\begin{theorem}\label{thm:pi*_is_a_monoid_morphism}
The map $\pi_{\mathrm{N}}^* \colon \left([\PP,\PP]\naif,\oplus\naif\right) \to \left([\J,\PP]\naif,\oplus \right)$ induced by $\pi$ is a morphism of monoids. 
\end{theorem}
\begin{proof}
Let $f,g \colon  \PP \to \PP$ be two pointed morphisms. 
By \cite[Lemma 3.13]{Caz}, $[\PP,\PP]\naif$ is generated by elements in degree $1$. 
Hence we can assume $f \simeq \frac{X}{u_1} \oplus\naif \frac{X}{u_2} \oplus\naif \cdots \oplus\naif \frac{X}{u_n}$ and $g \simeq \frac{X}{v_1} \oplus\naif \frac{X}{v_2} \oplus\naif \cdots \oplus\naif \frac{X}{v_m}$ for some $u_1,\ldots,u_n, v_1,\ldots,v_m \in k^{\times}$. 
Then Proposition \ref{prop:pi*_Xu1_to_X_un_and_oplus} implies the identity 
\begin{equation*}
\pi_{\mathrm{N}}^*\left([f] \oplus\naif [g]\right) = \left[\pi_{\mathrm{N}}^*(f)\right] \oplus \left[\pi_{\mathrm{N}}^*(g)\right] 
\end{equation*}
and hence the result.  
\end{proof}

%%%

\section{Group completion}\label{sec:group_completion}

The morphism $\pi \colon \J \to \PP$ induces the following commutative diagram of solid arrows.

\begin{align}\label{eq:diagram_group_completion_setup}
\xymatrix{
[\J,\PP]\naif  \ar[r]^-{\nu_{\J}}  & [\J,\PP]^{\AAA} \ar@/^/[d]^-{(\pi_{\AAA}^*)^{-1}}   \\
[\PP,\PP]\naif \ar[u]^-{\pi_{\mathrm{N}}^*} \ar[r]_-{\nu_{\PP}} & 
\ar@{.>}[ul]^-{\psi}
[\PP,\PP]^{\AAA} \ar[u]^-{\pi_{\AAA}^*} 
}
\end{align}

Recall that $\cc= (\pi^*_{\AAA})^{-1} \circ \nu_{\J}$ denotes the bijection $\cc \colon  [\J, \PP]\naif \to [\PP, \PP]^{\AAA}$ of Equation \eqref{eq:jpn_ppa}. 

\begin{lemma}\label{lem:cc_piN*_is_a_monoid_morphism} 
We have the identity of morphisms $\cc \circ \pi_{\mathrm{N}}^* = \nu_{\PP}$. 
\end{lemma}
\begin{proof}
Since the outer square in Diagram \eqref{eq:diagram_group_completion_setup} commutes, we have 
$(\pi^*_{\AAA})^{-1} \circ \nu_{\J} \circ \pi_{\mathrm{N}}^* = \nu_{\PP}$.  
Since $\cc = (\pi^*_{\AAA})^{-1} \circ \nu_{\J}$ by definition, this shows 
\begin{align*}
\cc \circ \pi_{\mathrm{N}}^* = \nu_{\PP}
\end{align*}
as desired. 
\end{proof}

In \cite[Theorem 3.22]{Caz} Cazanave proves that the canonical map $\nu_{\PP} \colon \left([\PP, \PP]\naif, \oplus\naif\right) \to \left([\PP, \PP]^{\AAA}, \oplus^{\AAA}\right)$ is a group completion. 
Hence there exists a unique group homomorphism 
\begin{align*}
\psi \colon \left([\PP, \PP]^{\AAA},\oplus^{\AAA}\right) \to \left([\J, \PP]\naif, \oplus\right)
\end{align*} 
making the lower triangle in Diagram \eqref{eq:diagram_group_completion_setup} commute.

We will show in this section that $\pi_{\mathrm{N}}^*$ has image in a certain subgroup and induces a group completion. 
Together with Cazanave's result this implies that we have a canonical isomorphism between the two group completions induced by $\pi_{\mathrm{N}}^*$ and $\nu_{\PP}$, respectively. 
The main result is proven in Theorem \ref{thm:canonical_isomorphism_from_GG}. 

\subsection{Motivic Brouwer degree}
\label{sec:motivic_Brouwer_degree}

In \cite{Morel_ICM} Morel describes the analog of the topological Brouwer degree map in $\AAA$-homotopy theory. 
For pointed endomorphisms of $\PP$ it defines a homomorphism 
\[
\deg^{\AAA} \colon [\bbP^1, \bbP^1]^{\AAA} \to \GW(k).
\]
We recall that by the work of Cazanave \cite[Corollary 3.10]{Caz} and Morel \cite[Theorem 7.36]{Morel12} the map given by 
\begin{align*}
f \mapsto \left(\deg^{\AAA}(f), \res(f)\right),
\end{align*}
where $\res(f)$ denotes the resultant of $f$ as in  \cite{Caz}, which we recalled in Proposition \ref{prop:Caz_thm_pointed_P1_endo}, 
induces an isomorphism of groups 
\begin{equation}\label{eq:Mor_PtoP_GW}
\rho \colon [\bbP^1, \bbP^1]^{\AAA}  \xrightarrow{\cong} \GW(k) \times_{k^\times / k^{\times2}} k^\times.
\end{equation}
Since our definition of $\deg$ is compatible with the notion of degree of a rational function used by Cazanave in \cite{Caz}, 
the work of Cazanave and Morel implies that, for every pointed morphism $f \colon \PP \to \PP$ we have 
\begin{align*}
\deg([f]) = \mathrm{rank}\left(\deg^{\AAA}([f])\right),
\end{align*}
where $\mathrm{rank}$ denotes the homomorphism $\GW(k) \to \Z$ induced by the rank of a quadratic form. 
For a pointed morphism $g \colon \J \to \PP$ with $\cc([g])=[f]$, we have 
\begin{align*}
\deg([g]) = \mathrm{rank}\left(\deg^{\AAA}(\cc[g])\right).   
\end{align*}
Hence we have the commutative diagram 
\begin{align}\label{eq:diagram_rank_and_deg_and_GW}
\xymatrix{
1 \ar[r] & [\calJ, \ato]\naif \ar[d]_-{\cc_0}^-{\cong} \ar[r] & [\J,\PP]\naif \ar[d]^-{\mathrm{bijection}}_-{\cc} \ar[r]^-{\deg} & \Pic(\calJ) \ar[d]^-{\cong} \ar[r] & 1 \\
1 \ar[r] & [\PP, \ato]^{\AAA} \ar[d]^-{\cong}  \ar[r] & [\PP, \PP]^{\AAA}   \ar[r]^-{\deg} \ar[d]_-{\rho}^-{\cong} & \Pic(\PP) \ar[d]^-{\cong} \ar[r] & 1 \\
1 \ar[r] & \GW(k)_0 \times_{k^\times/k^{\times 2}} k^\times \ar[r] & \GW(k) \times_{k^\times / k^{\times2}} k^\times \ar[r]^-{\mathrm{rank}} & \Z \ar[r] & 1
}
\end{align} 
where $\GW(k)_0 \times_{k^\times/k^{\times 2}} k^\times$ denotes the kernel of the rank homomorphism.

%%%

\begin{proposition}
\label{prop:pi_N_injective}
The map $\pi_{\mathrm{N}}^*$ is injective. 
\end{proposition}
\begin{proof}
By Lemma \ref{lem:cc_piN*_is_a_monoid_morphism} we know $\cc \circ \pi_{\mathrm{N}}^* = \nu_{\PP}$. 
Since $\cc$ is a bijection, it suffices to show that $\nu_{\PP}$ is injective. 
The isomorphism $\rho$ fits into the commutative diagram 
\begin{align*}
\xymatrix{
[\PP,\PP]\naif \ar[d]^-{\cong} \ar[r]^-{\nu_{\PP}} & [\PP,\PP]^{\AAA} \ar[d]_-{\cong}^-{\rho} \\
\MW^s(k) \times_{k^\times / k^{\times2}} k^\times \ar[r]_-{\hat{\wc}} & \GW(k) \times_{k^\times / k^{\times2}} k^\times
}
\end{align*}
where $\MW^s(k)$ denotes the stable monoid of symmetric bilinear forms as in \cite[Definition 3.8]{Caz} and $\hat{\wc}$ is the group completion induced by the group completion $\wc \colon \MW^s(k) \to \GW(k)$. 
Since the vertical maps are isomorphisms by \cite[Corollary 3.10]{Caz} and \cite[\S7.3]{Morel12}, $\nu_{\PP}$ is injective if and only if $\hat{\wc}$ is injective. 
To show that $\hat{\wc}$ is injective, it suffices to show that the group completion $\wc \colon \MW^s(k) \to \GW(k)$ is injective. 
Since $\MW^s(k)$ satisfies the cancellation property by the definition of $\MW^s(k)$ in \cite[Definition 3.8]{Caz}, respectively by Witt's cancellation theorem, the map $\wc$ is indeed injective.  
This proves the assertion. 
\end{proof}

%%%

\begin{proposition}
\label{prop:gu1_distinct}
Let $u,v \in k^\times$. 
If $u \neq v$, then $[g_{u,1}] \neq  [g_{v,1}]$.
\end{proposition}
\begin{proof} 
Assume that $u\neq v \in k^{\times}$. 
By Cazanave's work \cite[Corollary 3.10]{Caz}, this implies $[X/u] \neq [X/v]$. 
By Proposition \ref{prop:pi_N_injective} this implies $\pi_{\mathrm{N}}^*([X/u]) \neq \pi_{\mathrm{N}}^*([X/v])$.  
By Proposition \ref{prop:gu1_oplus_acts_on_caz} we have $\pi_{\mathrm{N}}^*([X/u]) = [g_{u,1}] \oplus \pi$ and $\pi_{\mathrm{N}}^*([X/v]) = [g_{v,1}] \oplus \pi$. 
Since $[\J,\PP]\naif$ is a group, this implies $[g_{u,1}] \neq  [g_{v,1}]$.  
\end{proof}

\begin{proposition}\label{prop:image_of_100u}
For every $u\in k^{\times}$ we have 
\[
\left(\deg^{\AAA}(\cc[\pi_{\mathrm{N}}^*(X/u)]),\res (\cc[\pi_{\mathrm{N}}^*(X/u)])\right) = \left(\langle u \rangle,u\right) ~ \text{in} ~ \GW(k) \times_{k^\times/k^{\times 2}} k^\times.
\]
In particular, for $\pi_{\mathrm{N}}^*(X/1)=\pi$, we get
\[
\left(\deg^{\AAA}(\cc[\pi]),\res (\cc[\pi])\right) = \left(\langle 1 \rangle,1\right) ~ \text{in} ~ \GW(k) \times_{k^\times/k^{\times 2}} k^\times.
\]
\end{proposition}
\begin{proof}
By Lemma \ref{lem:cc_piN*_is_a_monoid_morphism} we know $\cc[\pi_{\mathrm{N}}^*(X/u)] = \nu_{\PP}([X/u])$. 
In \cite[3.4]{Caz} Cazanave shows that the image of $ \nu_{\PP}([X/u])$ in $\GW(k)\times_{k^\times / k^{\times 2}}k^\times$ is $\left(\langle u \rangle,u\right)$ by assigning it to the rank $1$ symmetric matrix $[u]$,  
which has determinant $u$ and corresponds to the quadratic form $\langle u \rangle$. 
\end{proof}

\begin{remark}\label{rem:KW2_computed_Xn/1}
We note that, since $[(1,0:0,u)_n] = X^n/u$, the work of Kass and Wickelgren in \cite[Lemma 5]{KassWickelgren2} implies that we have  $\deg^{\AAA}(\cc([(1,0:0,u)_n)])) = \langle u \rangle + \frac{n-1}{2}\langle 1,-1 \rangle$ for $n$ odd, and 
$\deg^{\AAA}(\cc([(1,0:0,u)_n)])) = \frac{n}{2}\langle 1,-1 \rangle$ for $n$ even.  
On the other hand, 
by the choice of the morphism $n\pi$ in Definition \ref{def:minus_pi_and_n_pi}
we have $\deg^{\AAA}(\cc([n\pi])) = 
\deg^{\AAA}(n[\id]) = n \langle 1 \rangle$. 
In particular, this implies that 
$[(1,0:0,1)_n]$ and $[n\pi]$ are in general not equal for $n>1$ which explains our comment at the end of Remark \ref{rem:construction_of_npi}. 
\end{remark}

%%%%%%%%%%%%%%%%%%%%%%%%%

In light of Question \ref{question:tilde_pi} 
we would like to show that $\rho(\cc[\tilde{\pi}])$ is the class $(-\langle 1 \rangle, 1)$ in $\GW(k) \times_{k^\times/k^{\times 2}} k^\times$.   
We are not able to confirm this yet, since we do not know how to compute the resultant of $\cc[\tilde{\pi}]$. 
We can, however, show the following fact based on the computations of topological degrees in Appendix \ref{sec:evidence}. 
We thank Kirsten Wickelgren for mentioning to us the idea to use the arguments of \cite{BHeta} and \cite{BWEuler} to reduce the computation to the Grothendieck--Witt group of the integers.

\begin{theorem}\label{thm:class_of_tilde_pi} 
For every field $k$, 
we have 
\begin{align*}
\deg^{\AAA}\left(\cc[\tilde{\pi}]\right) = - \langle 1 \rangle  ~ \text{in} ~ \GW(k).
\end{align*}
\end{theorem}
\begin{proof} 
First we assume $k = \bbF_2$. 
By \cite[Lemma 3.13]{Caz}, $[\PP,\PP]\naif$ is generated by elements in degree $1$, i.e., the class of $X/1$ generates $[\PP,\PP]\naif$. 
Since $\nu_{\PP} \colon [\PP,\PP]\naif \to [\PP,\PP]^{\AAA}$ is a group completion, 
this implies that 
$\deg_{\bbF_2}^{\AAA} \colon [\PP,\PP]^{\AAA} \xrightarrow{\cong} \GW(\bbF_2) = \bbZ$ is an isomorphism, where we refer to \cite{MilnorHusemoller} for the Grothendieck group $\GW(\bbF_2)$ of symmetric bilinear forms over $\bbF_2$ and the isomorphism $\GW(\bbF_2) = \bbZ$ (see \cite[Lemma B.5]{BWEuler}, \cite[III Remark (3.4)]{MilnorHusemoller}).   
Since the right-hand side of Diagram \eqref{eq:diagram_rank_and_deg_and_GW} commutes, 
the fact that we have $\deg\left([\tilde{\pi}]\right) = -1$ implies $\deg^{\AAA}\left(\cc[\tilde{\pi}]\right) = - \langle 1 \rangle$ in $\GW(k)$.

Next we let $k$ be a field of characteristic $2$. 
Then the canonical morphism $\Spec{k} \to \Spec{\bbF_2}$ induces a commutative diagram of group homomorphisms 
\begin{align*}
\xymatrix{
[\PP,\PP]_{\bbF_2}^{\AAA} \ar[d]_-{\deg_{\bbF_2}^{\AAA}} \ar[r] & [\PP,\PP]_{k}^{\AAA} \ar[d]^-{\deg_{k}^{\AAA}} \\
\GW(\bbF_2) \ar[r] & \GW(k).
}
\end{align*}
Since $\tilde{\pi}$ is defined over $\bbF_2$ 
and $\deg_{\bbF_2}^{\AAA}(\cc[\tilde{\pi}]) = -\langle 1 \rangle$ by the first case, 
this implies $\deg_{k}^{\AAA}(\cc[\tilde{\pi}]) = -\langle 1 \rangle$ in $\GW(k)$. 

Now we assume that $k$ is a field of characteristic $\ne 2$. 
The proof for this case is also based on the fact that $\tilde{\pi}$ is already defined over $\bbZ$ and not just $k$. 
To make the argument work, however, requires a bit more effort. 
For a ring $S$, let $\mathcal{SH}(S)$ denote the stable motivic homotopy category over $\Spec{S}$.  
Let $KO_k\in \mathcal{SH}(k)$ denote the motivic spectrum over $\Spec{k}$ which represents Hermitian K-theory. 
It is equipped with a unit morphism $\unitm_k \colon \unit_k \to KO_k$ in $\mathcal{SH}(k)$. 
Let $\stable \colon [\PP,\PP]^{\AAA} \to \unit_k^{0,0}(\Spec{k})$ denote the homomorphism defined by stabilization and note that there is a canonical isomorphism $KO_k^{0,0}(\Spec{k}) \cong \GW(k)$. 
We then define the homomorphism $\delta$ as the following composition:  
\begin{align*}
\delta \colon [\PP,\PP]^{\AAA} \xrightarrow{\stable} \unit_k^{0,0}(\Spec{k}) 
\xrightarrow{\unitm_k} KO_k^{0,0}(\Spec{k}) 
\cong \GW(k). 
\end{align*}
We claim that the homomorphism $\delta$ can be identified with the motivic Brouwer degree $\deg_k^{\AAA}$ over $k$. 
To prove the claim we follow the argument of Levine and Raksit in \cite[proof of Theorem 8.6, page 1845]{LevineRaksit}. 
By Morel's computation \cite[Theorem 6.40]{Morel12}, the isomorphism 
$\GW(k) \cong \unit_k^{0,0}(\Spec{k})$ sends $\langle u \rangle \in \GW(k)$, for $u \in k^\times$, to $\stable(\nu_{\PP}[X/u])$, the image of the class of $X/u \colon \PP_k \to \PP_k$, $[x_0:x_1] \mapsto [x_0:ux_1]$, in $\unit_k^{0,0}(\Spec{k})$.  
Hence the classes $\stable(\nu_{\PP}[X/u])$ for all $u\in k^\times$ generate $\unit_k^{0,0}(\Spec{k})$. 
Thus, in order to prove the claim it suffices to show that $\delta(\nu_{\PP}[X/u]) = \langle u \rangle$ in $\GW(k)$, since $\deg_k^{\AAA}([X/u]) = \langle u \rangle \in \GW(k)$. 
That is, we need to show $\unitm_k(\stable(\nu_{\PP}[X/u])) = \langle u \rangle$. 
This follows from \cite[Corollary 6.2]{Ananyevskiy} which proves the claim.

In \cite[\S 3.8.3]{BHeta} Bachmann and Hopkins construct a motivic spectrum $KO'_{\bbZ} \in \mathcal{SH}(\bbZ)$ with a unit morphism $\unitm'_{\bbZ} \colon \unit_{\bbZ} \to KO'_{\bbZ}$, 
and write $KO'_k$ and $\unitm'_k$ for the pullback of $KO'_{\bbZ}$ and $\unitm'_{\bbZ}$ to $\mathcal{SH}(k)$ along the canonical morphism $\Spec{k} \to \Spec{\bbZ}$. 
Since the characteristic of $k$ is not $2$, there is an equivalence of ring spectra $KO'_k \simeq KO_k$ by \cite[Lemma 3.38 (3)]{BHeta}, which induces an isomorphism $(KO'_k)^{0,0}(\Spec{k}) \cong KO^{0,0}_k(\Spec{k})$. 
Thus, there is an isomorphism $(KO'_k)^{0,0}(\Spec{k}) \cong KO_k^{0,0}(\Spec{k})$ which 
fits into the following commutative diagram. 
\begin{align*}
\xymatrix{
\unit_k^{0,0}(\Spec{k}) \ar[dr]_-{\unitm_k} \ar[r]^-{\unitm'_k} & 
(KO'_k)^{0,0}(\Spec{k}) \ar[d]^-{\cong} \\
& KO_k^{0,0}(\Spec{k})
}
\end{align*}
By the above, we may therefore identify $\deg_k^{\AAA}$ over $k$ with the composed homomorphism 
\begin{align*}
[\PP,\PP]^{\AAA} \xrightarrow{\stable} \unit_k^{0,0}(\Spec{k}) \xrightarrow{\unitm'_k} (KO'_k)^{0,0}(\Spec{k}) \cong KO_k^{0,0}(\Spec{k}) \cong \GW(k). 
\end{align*}
Furthermore, by \cite[Lemma 3.38 (2)]{BHeta}, there is an isomorphism $\pi_{0,0}(KO'_{\bbZ}) \cong \GW(\bbZ)$, where $\GW(\bbZ)$ denotes the Grothendieck--Witt group over $\bbZ$ defined in \cite[Chapter II]{MilnorHusemoller}. 
Let $[\PP,\PP]^{\AAA}_{\bbZ}$ denote the set of endomorphisms of $\PP$ in the pointed unstable $\AAA$-homotopy category over $\Spec{\bbZ}$. 
We now define the homomorphism 
$\deg^{\AAA}_{\bbZ} \colon [\PP,\PP]^{\AAA}_{\bbZ} \to \GW(\bbZ)$ 
as the composition 
\begin{align*}
\deg^{\AAA}_{\bbZ} \colon 
[\PP,\PP]^{\AAA}_{\bbZ} \xrightarrow{\stable_{\bbZ}} \unit_{\bbZ}^{0,0}(\Spec{\bbZ}) \xrightarrow{\unitm'_{\bbZ}} (KO'_{\bbZ})^{0,0}(\Spec{\bbZ}) \cong \GW(\bbZ). 
\end{align*} 
The canonical homomorphism $\bbZ \to k$ then induces the following commutative square   
\begin{align*}%\label{eq:diagram_deg_over_Z_is_in_image}
\xymatrix{
[\PP,\PP]^{\AAA}_{\bbZ} \ar[d] \ar[rr]^-{\deg^{\AAA}_{\bbZ}} & & \GW(\bbZ) \ar[d]^-{\mathfrak{b}_k} \\
[\PP,\PP]^{\AAA}_{k} \ar[rr]_-{\deg^{\AAA}_k} & & \GW(k)
}
\end{align*}
where $\mathfrak{b}_k \colon \GW(\bbZ) \to \GW(k)$ denotes the change of coefficients homomorphism. 
As a consequence we see that if $[f] \in [\PP,\PP]^{\AAA}_k$ is in the image of the homomorphism 
$[\PP,\PP]^{\AAA}_{\bbZ} \to [\PP,\PP]^{\AAA}_k$,  
then 
\begin{align}\label{eq:deg_over_Z_is_in_image}
\deg^{\AAA}_k([f]) = 
\mathfrak{b}_k(\deg^{\AAA}_{\bbZ}([f])). 
\end{align}
By \cite[Lemma 5.6]{BWEuler} (see also \cite[Theorem II.4.3]{MilnorHusemoller}),  $\GW(\bbZ)$ is generated over $\Z$ by the classes $\langle 1 \rangle$ and $\langle -1 \rangle$. 
For a class $\qq \in \GW(\bbZ)$, let $\qq_{\bbC}$ and $\qq_{\bbR}$ denote the images of $\qq$ in $\GW(\bbC)$ and $\GW(\bbR)$, respectively. 
Then %similar to the arguments in \cite[\S 5]{BWEuler}  
$\qq \in \GW(\bbZ)$ is uniquely determined by the integers $\mathrm{rank}(\qq_{\bbC})$ and $\mathrm{sgn}(\qq_{\bbR})$, given by the rank and signature of $\qq_{\bbC}$ and $\qq_{\bbR}$, respectively, 
via the formula 
\begin{align}\label{eq:compute_q_via_dC_and_dR}
\qq = \frac{\mathrm{rank}(\qq_{\bbC}) + \mathrm{sgn}(\qq_{\bbR})}{2} \langle 1 \rangle +  \frac{\mathrm{rank}(\qq_{\bbC}) - \mathrm{sgn}(\qq_{\bbR})}{2} \langle -1 \rangle
\in \GW(\Z). 
\end{align}

Since $\J$ and both morphisms $\pi$ and $\tilde{\pi}$ are defined over $\Spec{\bbZ}$, 
we can now apply the above observations to prove the assertion of the proposition.  
Since $\pi$ is an $\AAA$-weak equivalence over $\Spec{\bbZ}$ as well, 
we can form the pointed $\AAA$-homotopy class $\cc_{\bbZ}\left([\tilde{\pi}]\right) := [\tilde{\pi} \circ \pi^{-1}]^{\AAA} \in [\PP,\PP]^{\AAA}_{\bbZ}$ defined by the zig-zag $\PP_{\bbZ} \xleftarrow{\pi} \calJ_{\bbZ} \xrightarrow{\tilde{\pi}}\PP_{\bbZ}$.  
The class  $\cc_{\bbZ}\left([\tilde{\pi}]\right)$ is sent  to $\cc\left([\tilde{\pi}]\right)$ under base change. 
Thus, by the above arguments, to determine $\deg^{\AAA}_k\left(\cc[\tilde{\pi}]\right)$ in $\GW(k)$ it suffices to compute the rank and signature of $\deg^{\AAA}_{\bbZ}\left(\cc_{\bbZ}[\tilde{\pi}]\right)$ after base change to $\bbC$ and $\bbR$, respectively. 
Since the right-hand side of Diagram \eqref{eq:diagram_rank_and_deg_and_GW} commutes, the fact that we have $\deg\left([\tilde{\pi}]\right) = -1$ implies $\mathrm{rank}(\deg_{\bbC}^{\AAA}\left(\cc[\tilde{\pi}]\right)) = -1$. 
In Appendix \ref{sec:evidence} and Example \ref{ex:realization_pi_tilde} we show that the signature of 
$\deg^{\AAA}_{\bbR}\left(\cc[\tilde{\pi}]\right)$ over $\bbR$ is $-1$. 
Thus, by Formula \eqref{eq:compute_q_via_dC_and_dR}, we get $\deg^{\AAA}_{\bbZ}\left(\cc_{\bbZ}[\tilde{\pi}]\right) = - \langle 1 \rangle$ in $\GW(\bbZ)$. 
By Equation \eqref{eq:deg_over_Z_is_in_image} we can therefore conclude that  
$\deg_k^{\AAA}\left(\cc[\tilde{\pi}]\right) = - \langle 1 \rangle$ in $\GW(k)$. 
\end{proof}

%%%%%%%%%%%%%%%%%%%%%%%%%%

\subsection{Group completion of naive homotopy classes}
\label{sec:group_completion_of_naive_classes}

We will now describe the homomorphism $\psi \colon \left([\PP, \PP]^{\AAA},\oplus^{\AAA}\right) \to \left([\J, \PP]\naif, \oplus\right)$ induced by the universal property of the group completion $\nu_{\PP}$ in more detail.   
By \cite[Lemma 3.13]{Caz}, $[\PP,\PP]\naif$ is generated by elements in degree $1$, i.e., it is generated by the set of classes $[X/u]$ for all $u\in k^{\times}$. 
Hence, since $\nu_{\PP}$ is a group completion, every element in $[\PP,\PP]^{\AAA}$ of degree $0$ can be written as a sum of the differences $\gamma_{u,v}:=\nu_{\PP}([X/u]) - \nu_{\PP}([X/v])$ for suitable $u,v \in k^{\times}$. 
Thus the set of classes $\gamma_{u,v}$ for all $u,v \in k^{\times}$ generates the subgroup $[\PP,\PP]_0^{\AAA}$ of degree $0$ elements. 
Because of this we would like to understand the image of the $\gamma_{u,v}$ under $\psi$. 
Since $\psi$ is a group homomorphism, we know  $\psi(\gamma_{u,v}) \oplus \psi(\nu_{\PP}([X/v])) = \psi(\nu_{\PP}([X/u]))$. 
Since $\psi \circ \nu_{\PP} = \pi_{\mathrm{N}}^*$, this implies $\psi(\gamma_{u,v}) \oplus \pi_{\mathrm{N}}^*([X/v]) = \pi_{\mathrm{N}}^*([X/u])$. 
By Lemma \ref{lem:g_uv_oplus_pi_v}, the map $g_{u,v}$ satisfies $[g_{u,v}] \oplus \pi_{\mathrm{N}}^*([X/v])) = \pi_{\mathrm{N}}^*([X/u])$. 
Hence, since $[\J,\PP]\naif$ is a group, we get 
\[
\psi(\gamma_{u,v}) = [g_{u,v}] ~ \text{in} ~ [\J,\PP]\naif.
\]

This motivates the following definition of the subgroup $\GG \subseteq [\J,\PP]\naif$. 

\begin{definition}\label{def:subgroup_GG}
Let $\GG_0:=\langle [g_{u,v}] \st u,v\in k^{\times} \rangle \subseteq [\J,\ato]\naif$ denote the subgroup generated by the homotopy classes of the maps $g_{u,v}$. 
Let $\GG \subseteq [\J,\PP]\naif$ be the subgroup generated by $\GG_0$ and $[\pi]$. 
\end{definition}

We note that by Lemma \ref{lemma:guv_plus_gvu}, we have $-[g_{u,v}] = [g_{v,u}]$, while $[g_{u,u}]$ is the neutral element, and we therefore have $-([g_{u,v}] \oplus \pm n[\pi]) = [g_{v,u}] \oplus \mp n[\pi]$ in $\GG$. 

\begin{remark}\label{rem:MM0_completes_to_GG0}
By Lemma \ref{lemma:guv_plus_gvu}, we have $[g_{u,v}] \oplus [g_{v,1}] = [g_{u,1}]$ for all $u,v \in k^{\times}$. 
Thus, every element in $\GG_0$ is a sum of differences of elements in 
the submonoid of $[\J,\ato]\naif$ generated by the set of homotopy classes $[g_{u,1}]$ for all $u \in k^{\times}$. 
This implies that the inclusion of the submonoid of $[\J,\ato]\naif$ generated by the set of homotopy classes $[g_{u,1}]$ for all $u \in k^{\times}$ 
into $\GG_0$ is a group completion.     
\end{remark}

\begin{lemma}
The morphism of monoids $\pi_{\mathrm{N}}^* \colon [\PP,\PP]\naif \to [\J,\PP]\naif$ has image in $\GG$. 
\end{lemma}
\begin{proof}
By \cite[Lemma 3.13]{Caz}, $[\PP,\PP]\naif$ is generated by elements in degree $1$, i.e., it is generated by the set of classes $[X/u]$ for all $u\in k^{\times}$. 
Hence it suffices to show that $\pi_{\mathrm{N}}^*([X/u])$ is contained in $\GG$. 
This follows from Lemma \ref{lem:g_u1_oplus_pi}. 
\end{proof}

\begin{proposition}\label{prop:GG_is_group_completion}
The morphism of monoids $\pi_{\mathrm{N}}^* \colon [\PP,\PP]\naif \to \GG$ is a group completion. 
\end{proposition}
\begin{proof} 
Let $H$ be an abelian group and $\mu \colon [\PP,\PP]\naif \to H$ be a morphism of monoids. 
We will show that there is a unique homomorphism of groups $\widetilde{\mu} \colon \GG \to H$ such that $\widetilde{\mu} \circ \pi_{\mathrm{N}}^* = \mu$. 

We set $\widetilde{\mu}([\pi]):=\mu([X/1])$.  
By Lemma \ref{lem:g_uv_oplus_pi_v}, we have $[g_{u,v}] \oplus \pi_{\mathrm{N}}^*([X/v]) = \pi_{\mathrm{N}}^*([X/u])$ in $\GG \subseteq [\J,\PP]\naif$. 
Hence compatibility with $\mu$ forces the definition 
\begin{align*}
\widetilde{\mu}([g_{u,v}]):=\mu([X/u]) - \mu([X/v]).
\end{align*}
By definition of $\GG$ this induces a unique group homomorphism $\widetilde{\mu}$, once we have shown that it is well-defined.

Now we show that $\widetilde{\mu}$ is well-defined. 
Because $\GG \cong \GG_0 \oplus \Z$, all relations in $\GG$ amongst the generators arise from relations of the classes $[g_{u,v}]$. 
Consider a relation of the form
\begin{equation}\label{eq:relation_guv_group_completion_proof}
[g_{u_1,v_1}] \oplus \ldots \oplus [g_{u_s,v_s}] = 0.
\end{equation} 
We must then show that $\sum_i\widetilde{\mu}\left([g_{u_i,v_i}]\right)=0$ in $H$. 
Since $\GG$ is a group and by Lemma \ref{lem:g_uv_oplus_pi_v}, we have 
\begin{equation*}
\sum_i g_{u_i,v_i} \oplus \pi_{\mathrm{N}}^*([X/v_i]) 
= \sum_i \pi_{\mathrm{N}}^*([X/u_i]).
\end{equation*} 
Since $\GG$ is abelian, this implies 
\begin{equation*}
\sum_i [g_{u_i,v_i}] = \sum_i \pi_{\mathrm{N}}^*([X/u_i]) - \pi_{\mathrm{N}}^*([X/v_i]) 
= \sum_i \pi_{\mathrm{N}}^*([X/u_i]) - \sum_i \pi_{\mathrm{N}}^*([X/v_i])=0.
\end{equation*} 
Hence 
\begin{equation*}
\sum_i \pi_{\mathrm{N}}^*([X/u_i]) = \sum_i \pi_{\mathrm{N}}^*([X/v_i])
\end{equation*} 
in $[\PP,\PP]\naif$. 
Since $\pi_{\mathrm{N}}^*$ is an injective monoid morphism, in $[\PP,\PP]\naif$ we have the equation 
\begin{equation*}
\sum_i [X/u_i] = \sum_i [X/v_i].
\end{equation*} 
It thus follows that
\begin{equation*}
\mu\left(\sum_i [X/u_i]\right) = \mu\left(\sum_i [X/v_i]\right) ~ \text{in} ~ H. 
\end{equation*}
We calculate 
\begin{equation*}
\widetilde{\mu}\left(\sum_i [g_{u_i,v_i}]\right) = \sum_i \mu([X/u_i]) - \mu([X/v_i]) = 0,
\end{equation*}
as desired.
This shows that $\widetilde{\mu}$ is well-defined. 

It remains to show $\widetilde{\mu} \circ \pi_{\mathrm{N}}^* = \mu$. 
By \cite[Lemma 3.13]{Caz}, $[\PP,\PP]\naif$ is generated by the set of classes $[X/u]$ for all $u\in k^{\times}$.  
Hence $\mu$ is completely determined by the images of $[X/u]$ for all $u\in k^{\times}$. 
Thus, in order to show $\widetilde{\mu} \circ \pi_{\mathrm{N}}^* = \mu$, it suffices to show 
$\widetilde{\mu}\left(\pi_{\mathrm{N}}^* ([X/u])\right) = \mu([X/u])$ for every $u \in k^{\times}$. 
This is now immediate from the definition of $\widetilde{\mu}$ and Lemma \ref{lem:g_uv_oplus_pi_v}: 
\begin{align*}
\widetilde{\mu}\left(\pi_{\mathrm{N}}^* ([X/u])\right)  
& = \widetilde{\mu}\left([g_{u,1}]\oplus \pi_{\mathrm{N}}^* ([X/1])\right) \\
& = \widetilde{\mu}\left([g_{u,1}]\right)+ \widetilde{\mu}\left([\pi]\right)  \\
& = \mu\left([X/u]\right) - \mu\left([X/1]\right) + \mu\left([X/1]\right) \\
& = \mu\left([X/u]\right).
\end{align*}
This shows that $\pi_{\mathrm{N}}^* \colon [\PP,\PP]\naif \to \GG$ has the universal property of a group completion and finishes the proof.  
\end{proof}

\begin{theorem}\label{thm:canonical_isomorphism_from_GG}
There is a unique isomorphism 
of groups  
\begin{equation*}
\nc \colon \GG \to [\PP,\PP]^{\AAA}    
\end{equation*} 
such that $\nc \circ \pi_{\mathrm{N}}^* = \nu_{\PP}$.  
The homomorphism $\nc$ sends $[g_{u,v}]$ to the unique element $\gamma_{u,v}$ that satisfies $\nu_{\PP}^*([X/u]) = \gamma_{u,v} \oplus^{\AAA} \nu_{\PP}([X/v])$ 
and $[\pi]$ to $[\id]$. 
Moreover, $\nc$ and the homomorphism 
\[
\psi \colon [\PP,\PP]^{\AAA} \to \GG \subseteq [\J,\PP]\naif
\]
are mutual inverses to each other.  
\end{theorem}
\begin{proof}
The existence of $\nc$ and its definition is a consequence of Proposition \ref{prop:GG_is_group_completion} and its proof. 
The assertion that $\nc$ is the inverse of $\psi$ follows from the fact that $\nu_{\PP}$ is a group completion proven by Cazanave in \cite[Theorem 3.22]{Caz} and the universal property of group completion.  
\end{proof}

As a particular consequence of Theorem \ref{thm:canonical_isomorphism_from_GG} we get the following result. 

\begin{proposition}\label{prop:GG_0_maps_onto_PP_ato}
The restriction $\nc_0$ of the homomorphism $\nc$ to $\GG_0$ defines an isomorphism of groups 
\begin{equation*}
\nc_0 \colon \GG_0 \xrightarrow{\cong} [\PP,\ato]^{\AAA}. 
\end{equation*}
\end{proposition}
\begin{proof}
The elements $\gamma_{u,v}$ are of degree $0$, and hence they lie in the subgroup $[\PP,\ato]^{\AAA}$. As a consequence, the homomorphism $\nc$ and its restriction $\nc_0$ fit in the following commutative diagram of abelian groups. 
\begin{align*}
\xymatrix{
1 \ar[r] & \GG_0 \ar[d]_-{\nc_0} \ar[r] & \GG \ar[d]_-{\nc}^-{\cong} \ar[r] & \bbZ \ar[d]^-{\cong} \ar[r] & 1 \\
1 \ar[r] & [\PP, \ato]^{\AAA} \ar[r] & [\PP, \PP]^{\AAA}   \ar[r]  & \bbZ \ar[r] & 1 
}
\end{align*} 
Since the middle and right-most maps are isomorphisms, the assertion follows from the five-lemma. 
\end{proof}

The existence of the isomorphisms $\GG \xrightarrow{\nc} [\PP,\PP]^{\AAA} \xleftarrow{\ff} [\J,\PP]\naif$ does not imply that $\GG$ equals $[\J,\PP]\naif$. 
However, we make the following conjecture on the a priori subgroups $\GG_0$ and $\GG$. 

\begin{conjecture}\label{conj:GG_0_and_GG_are_the_whole_group}
The inclusions $\GG_0 \subseteq [\J,\ato]\naif$ and $\GG \subseteq [\J,\PP]\naif$ are equalities.  
\end{conjecture}

We will show in Theorem \ref{thm:GG0_generates_whole_group_for_finite_fields} in Section \ref{sec:Milnor_Witt_K_theory_and_degree_0} that Conjecture \ref{conj:GG_0_and_GG_are_the_whole_group} is true whenever $k=\bbF_q$ is a finite field.  
This follows from an explicit computation of  $\GG_0$ and $K_1^{\mathrm{MW}}(\bbF_q)$, the first Milnor--Witt K-theory of $\bbF_q$.

\begin{remark}\label{rem:proving_GG0_is_enough_for_conjecture}
It follows from the structure of the group $\GG$ as a product of $\GG_0$ and $\{n[\pi] \st n \in \bbZ\}$ that in order to prove Conjecture \ref{conj:GG_0_and_GG_are_the_whole_group} it suffices to show that the inclusion $\GG_0 \subseteq [\J,\ato]\naif$ is an equality, i.e., 
that the set of homotopy classes $[g_{u,v}]$ for all $u,v \in k^{\times}$ generates the group $[\J,\ato]\naif$.  
\end{remark}

\begin{remark}\label{rem:unique_isos_if_completion_holds}
If Conjecture     \ref{conj:GG_0_and_GG_are_the_whole_group} is true, then the group homomorphism $\psi \colon [\PP,\PP]^{\AAA} \to [\J,\PP]\naif$, induced by the fact that $\nu_{\PP}$ is a group completion, is an isomorphism 
and it agrees with $\ff^{-1}$, the inverse of the isomorphism of Theorem \ref{thm:group_structure}.  
We note, however, that this does not yet imply that the bijection $\cc = (\pi^*_{\AAA})^{-1} \circ \nu \colon  [\J, \PP]\naif \to [\PP, \PP]^{\AAA}$ is a group homomorphism. 
\end{remark}

%%%

\subsection{Milnor--Witt K-theory and morphisms in degree 0}
\label{sec:Milnor_Witt_K_theory_and_degree_0}

Our final goal is to prove Conjecture \ref{conj:GG_0_and_GG_are_the_whole_group} for finite fields.  
For the proof we use the Milnor--Witt K-theory  of a field which we now recall from \cite[Definition 3.1]{Morel12}. 

\begin{definition}\label{def:Milnor_Witt_K}
The Milnor--Witt K-theory of the field $k$, denoted $K_*^{\mathrm{MW}}(k)$, is the graded associative ring generated by symbols $[u]$ in degree $1$ for $u \in k^\times$ and the  symbol $\eta$ in degree $-1$ subject to the following relations:
\begin{enumerate}
    \item For each $u \in k^\times \setminus \{1\}$, $[u].[1-u] = 0$.
    \item For each pair $u,v \in (k^\times)^2$, $[uv] = [u] + [v] + \eta.[u].[v]$. 
    \item For each $u \in k^\times$, $\eta.[u]=[u].\eta$.
    \item Let $h:= \eta.[-1]+2$. 
    Then $\eta.h=0$. 
\end{enumerate}
\end{definition}

\begin{remark}
\label{rem:kmw_properties}
It follows directly from the defining relations for $K_*^{\mathrm{MW}}(k)$ that $[1]=0$ and $\eta.[u^2]=0$ for each $u \in k^\times$. See \cite[\S3.1]{Morel12} for a proof and other basic properties of Milnor--Witt K-theory.
\end{remark}

Recall that $\GG_0:=\langle [g_{u,v}] \st u,v\in k^{\times} \rangle \subseteq [\J,\ato]\naif$ denotes the subgroup generated by the homotopy classes of  the maps $g_{u,v}$. 
In this subsection we write $\GG_0(k)$ and $\GG(k)$ for the groups $\GG_0$ and $\GG$, respectively, to emphasize the dependency of the base field $k$.

\begin{proposition}\label{prop:GG0_is_isomorphic_to_K1MW_abstract}
For every field $k$, there is an isomorphism $\GG_0(k) \cong K_1^{\mathrm{MW}}(k)$.
\end{proposition}
\begin{proof}
By Proposition \ref{prop:GG_0_maps_onto_PP_ato}, we have an isomorphism $\nc_0 \colon \GG_0(k) \xrightarrow{\cong} [\PP,\ato]^{\AAA}$. 
As recalled in Remark \ref{rem:Morel_shows_K1MW_gives_degree_0_maps}, the work of Morel in \cite[\S7.3]{Morel12} implies that there is an isomorphism of groups $[\PP,\ato]^{\AAA} \cong K_1^{\mathrm{MW}}(k)$.  
The composition of these isomorphisms yields the assertion. 
\end{proof}

\begin{lemma}\label{lemma:calculation_of_K1MW_finite_fields_q_odd}
Let $k=\bbF_q$ be a finite field of odd characteristic. Let $v_1,v_2$ be non-squares in $\bbF_{q}^{\times}$. 
Then we have $[v_1v_2] = [v_1] + [v_2]$ in $K_1^{\mathrm{MW}}(\bbF_q)$.  
\end{lemma}
\begin{proof}
Since $q$ is odd, the kernel of the squaring homomorphism has two elements, $-1$ and $1$, 
i.e., $\bbF_{q}^\times / \bbF_{q}^{\times2} \cong \Z/2\Z$.
Because $1$ is a square, the set $\bbF_{q}\setminus \{0,1\}$ contains more non-squares than squares. 
Construct pairs $(t,1-t)$ from elements $t\in \bbF_{q}\setminus \{0,1\}$, and observe that there must exist at least one non-square $t$ such that $1-t$ is also a non-square. 
For the rest of the proof we fix $t$ to be one such non-square. 
By relation (1) in $K_1^{\mathrm{MW}}(\bbF_{q})$, we have $[t].[1-t]=0$. 
Let $v_1,v_2$ be non-squares in $\bbF_{q}^{\times}$. 
By relation (2) of Definition \ref{def:Milnor_Witt_K} we have  
\begin{align*}
[v_1v_2]&=[v_1]+[v_2]+\eta.[v_1].[v_2] 
\end{align*}
Hence to prove the assertion of the lemma it suffices to prove $\eta.[v_1].[v_2]=0$ in $K_1^{\mathrm{MW}}(\bbF_{q})$. 
Since $\bbF_{q}^\times / \bbF_{q}^{\times2} \cong \Z/2\Z$, there exist units $c_1$ and $c_2$ such that $c_1^2t = v_1$ and $c_2^2(1-t) = v_2$. 
Hence we get 
\begin{align*}
\eta.[v_1].[v_2] & = \eta.[c_1^2t].[c_2^2(1-t)] \\
&= \eta.([c_1^2]+[t]).([c_2^2] + [1-t])
\end{align*}
where the second equality follows from the fact, for every non-square $v$ and every square $c^2$, we have $[vc^2]=[v]+[c^2]$. 
Since we have $\eta.[c_i^2]=0$ by Remark \ref{rem:kmw_properties} and $[t].[1-t]=0$, 
we can conclude $\eta.[v_1].[v_2]=0$ which proves the lemma.    
\end{proof}

\begin{proposition}\label{prop:calculation_of_K1MW_and_GG0_finite_fields}
Let $k=\bbF_q$ be a finite field. 
Then $K_1^{\mathrm{MW}}(\bbF_q)$ is a finite cyclic group of order $q-1$. 
\end{proposition}

\begin{proof}
First we assume that $q$ is even. 
Then the squaring homomorphism is surjective, and hence every unit is a square. 
Fix $u$ to be a multiplicative generator of $\bbF_q^\times$. 
It follows from \cite[Lemma 3.6 (1)]{Morel12} that $K_1^{\mathrm{MW}}(\bbF_q)$ is generated by the elements $[v]$ for $v\in k^{\times}$, which are subject to the relation $[vv']=[v]+[v']$ for all $v,v' \in \bbF_q^\times$. 
The fact that $u^{q-1}=1$ yields the result that $K_1^{\mathrm{MW}}(\bbF_q)$ is cyclic of order $q-1$ generated by the symbol $[u]$. 

Now we assume that $q$ is odd. 
Let $u$ be a multiplicative generator of $\bbF_{q}^\times$. 
Using induction and Lemma \ref{lemma:calculation_of_K1MW_finite_fields_q_odd} we get $[u^n] = n[u]$ for all $n\ge 1$. 
Since $\bbF_q^{\times}$ is cyclic of order $q-1$, this shows that $K_1^{\mathrm{MW}}(\bbF_{q})$ is cyclic of order $q-1$ generated by the symbol $[u]$. 
\end{proof}

\begin{theorem}\label{thm:GG0_generates_whole_group_for_finite_fields}
Let $k=\bbF_q$ be a finite field. 
Then Conjecture \ref{conj:GG_0_and_GG_are_the_whole_group} is true, i.e., the inclusions $\GG_0(\bbF_q) \subseteq [\J,\ato]\naif$ and $\GG(\bbF_q) \subseteq [\J,\PP]\naif$ are equalities.
\end{theorem}
\begin{proof}
By Remark \ref{rem:proving_GG0_is_enough_for_conjecture} it suffices to prove the assertion for $\GG_0(\bbF_q)$. 
By Propositions  \ref{prop:GG0_is_isomorphic_to_K1MW_abstract} and  \ref{prop:calculation_of_K1MW_and_GG0_finite_fields}, both $\GG_0(\bbF_q)$ and $K_1^{\mathrm{MW}}(\bbF_q)$ are finite groups of the same cardinality. 
Since $[\PP,\ato]^{\AAA}$ and $K_1^{\mathrm{MW}}(\bbF_q)$ are isomorphic and since $\cc_0 \colon [\J,\ato]\naif \to [\PP,\ato]^{\AAA}$ is an isomorphism, $[\J,\ato]\naif$ is a finite group of the same cardinality as $\GG_0(\bbF_q)$ as well. 
Hence $\GG_0(\bbF_q) \subseteq [\J,\ato]\naif$ is an inclusion of finite groups of the same cardinality. 
This implies that the inclusion $\GG_0(\bbF_q) \subseteq [\J,\ato]\naif$ is an equality. 
\end{proof}

We conclude this section with the following observation. 
While Proposition \ref{prop:GG0_is_isomorphic_to_K1MW_abstract} shows that there is an isomorphism between $\GG_0(k)$ and $K_1^{MW}(k)$, the proof of Proposition \ref{prop:calculation_of_K1MW_and_GG0_finite_fields} suggests that the following map provides a concrete isomorphism. 
We consider this an interesting observation about $K_1^{\mathrm{MW}}(k)$ that arises from our work on maps $\J \to \ato$.

\begin{proposition}\label{prop:mapping_to_K1MW}
Let $k$ be one of the following fields: a quadratically closed field, a finite field, or $\bbR$. 
Then the assignment $[u] \mapsto [g_{u,1}]$ defines an isomorphism $\kappa \colon K_1^{\mathrm{MW}}(k) \to \GG_0(k)$.  
\end{proposition}
\begin{proof}
We will show that $K_1^{\mathrm{MW}}(k)$ and $\GG_0(k)$ are generated by the classes $[u]$ and $[g_{u,1}]$, respectively, and that these generators satisfy exactly the same type of relations. 
This implies that $\kappa$ is both a well-defined homomorphism and an isomorphism. 
We will prove the claim by looking at each type of field separately. 

First we assume that $k$ is a quadratically closed field, i.e., a field where every unit is a square. 
The group $K_1^{\mathrm{MW}}(k)$ is generated by the elements $[u]$ and the relation $[uv]=[u]+[v]$ 
for all $u,v \in k^{\times}$. 
To prove that $\kappa$ is an isomorphism we need to show that $\GG_0(k)$ is generated by the classes $[g_{u,1}]$ subject to the  relation $[g_{uv,1}]=[g_{u,1}] \oplus [g_{v,1}]$. 
Since every unit in $k$ is a square, we get $[g_{u,v}]=[g_{u/v,1}]$ by Lemma \ref{lem:square_scale}. 
Hence $\GG_0(k)$ is generated by elements $[g_{u,1}]$. 
By Proposition \ref{prop:gu1_distinct} we know that $[g_{u,1}] \neq [g_{v,1}]$ for $u \neq v\in k^{\times}$.   
By Lemma \ref{lem:gu1_summing} we get the relation $[g_{u,1}] \oplus [g_{v,1}] = [g_{uv,1}]$. 
Thus, the map sending $[u]$ to $[g_{u,1}]$ induces a homomorphism which is surjective and injective. 
Hence $\kappa$ is an isomorphism. 

For $k=\bbF_q$, Proposition \ref{prop:calculation_of_K1MW_and_GG0_finite_fields} shows 
that $K_1^{\mathrm{MW}}(\bbF_q)$ is generated by the symbol $[u]$ for a multiplicative generator $u \in \bbF_q^\times$. 
When $q$ is even, every unit is a square, and in this case $\kappa$ is an isomorphism. 
So we assume that $q$ is odd, and will now show that every element in $\GG_0(\bbF_q)$ can be written in the form $m[g_{u,1}]\oplus [g_{u^{2m'},1}]$ for some $m,m'\in \bbZ$. 
We will use that every square in $\bbF_q$ is equal to an even power of the generator $u\in \bbF_q^\times$, and distinguish three cases: 
Assume first $v_1,v_2 \in \bbF_q^{\times}$ are squares. 
By Lemma \ref{lem:square_scale} we then have $[g_{v_1,v_2}]=[g_{v_1/v_2,1}]=[g_{u^{2m}},1]$ for some $m$. 
Second, if $v_1$ is not a square and $v_2$ is a square, then $v_1/v_2 = u^{2m+1}$ for some $m\in \Z$. Then by Lemma \ref{lem:square_scale} and \ref{lem:gu1_summing} we know 
$[g_{v_1,v_2}]= [g_{u^{2m+1},1}]=[g_{u,1}]\oplus [g_{u^{2m},1}]$. 
Note that $[g_{v_2,v_1}] = -[g_{v_1,v_2}] = -[g_{u,1}]\oplus [g_{u^{-2m},1}]$. 
Third, assume that both $v_1$ and $v_2$ are non-squares in $\bbF_q^{\times}$. 
Since $\bbF_{q}^\times / \bbF_{q}^{\times2} \cong \Z/2\Z$, 
we can find an $m\in \bbZ$ such that $v_1/v_2= u^{2m}$. 
We have $[g_{v_1,v_2}]=[g_{u\cdot u^{2m},u}]$ by Lemma \ref{lem:square_scale} and scaling by the square $u/v_2$. 
We can now apply Lemma \ref{lemma:guv_plus_gvu} and then Lemma \ref{lem:gu1_summing} to get 
\begin{equation*}
[g_{u\cdot u^{2m},u}]=[g_{u\cdot u^{2m},1}]\oplus [g_{1,u}]= [g_{u,1}]\oplus [g_{ u^{2m},1}]\oplus [g_{1,u}] = [g_{ u^{2m},1}].
\end{equation*}

To conclude the argument we note that, for $v_1,v_2 \in \bbF_q^{\times}$ with $v_1+v_2\neq 1$, there is the relation $\langle v_1 \rangle + \langle v_2 \rangle = \langle v_1+v_2 \rangle  + \langle (v_1+v_2)v_1v_2 \rangle $ in $\GW(\bbF_q)$. 
For $s$ and $1-s$ in $\bbF_q^{\times}$, this gives $\langle s \rangle + \langle 1-s \rangle = \langle 1 \rangle + \langle s(1-s) \rangle = \langle 1,1 \rangle$. 
In particular, since $u, s,$ and $1-s$ all differ by squares and hence $\langle u \rangle = \langle s \rangle = \langle 1-s \rangle$ in $\GW(\bbF_q)$, 
we have $\langle u \rangle + \langle u \rangle = \langle 1,1 \rangle = \langle u^2,1 \rangle$ in $\GW(\bbF_q)$.     
By \cite[Corollary 3.10]{Caz} this relation implies  
$[X/u] \oplus\naif [X/u] = [X/u^2] \oplus\naif [X/1]$ in $[\PP,\PP]\naif$. 
By Proposition \ref{prop:pi*_Xu1_to_X_un_and_oplus} and Lemma \ref{lem:g_uv_oplus_pi_v},  this implies  the equality $[g_{u,1}]\oplus [g_{u,1}] = [g_{u^2,1}]$ in $\GG_0(\bbF_q)$. 
Iterating this argument, we get $(q-1)[g_{u,1}]=[g_{u^{q-1},1}]=[g_{1,1}]$. 
Since $[g_{v_1},1] \neq [g_{v_2},1]$ for $v_1 \neq v_2 \in \bbF_q^{\times}$ by Proposition \ref{prop:gu1_distinct}, this implies that  $\GG_0(\bbF_q)$ is cyclic of order $q-1$ generated by $[g_{u,1}]$. 
Hence the map $[u] \mapsto [g_{u,1}]$ is a well-defined homomorphism which is surjective and injective. 
Thus $\kappa$ is an isomorphism in this case as well. 

Finally, we assume $k=\bbR$. 
First we determine the generators and relations for $K_1^{\mathrm{MW}}(\bbR)$. 
For $u> 0$, we have $[-u]=[-1]+[u] + \eta.[-1].[u] = [-1]+[u]$ and $-[u]=[1/u]$.  
Thus every element in $K_1^{\mathrm{MW}}(\bbR)$ can be written as $n[-1]+[u]$ with $u>0$ for some $n\in \bbZ$ subject to the relation $(n[-1]+[u])+(m[-1]+[v])=(n+m)[-1]+[uv]$. 
Next we show that $\GG_0(\bbR)$ has analogous generators and relations.  
Assume $u,v>0$. 
Since $v$ is a square, Lemmas \ref{lem:square_scale} and \ref{lem:gu1_summing} imply the identities
\begin{equation*}
[g_{u,v}]  = [g_{u/v,1}], ~ \text{and} ~ [g_{-u,v}] = [g_{-u/v,1}] 
= [g_{-1,1}]  \oplus [g_{u/v,1}]. 
\end{equation*}
Using that $v$ is a square, Lemma \ref{lem:square_scale} yields the following identity
\begin{equation*}
[g_{u,-v}] 
= [g_{u/v,-1}] 
= [g_{u/v,1}] \oplus [g_{1,-1}] 
= -[g_{-1,1}] \oplus [g_{u/v,1}]
\end{equation*}
where we have $[g_{1,-1}] = -[g_{-1,1}]$ by Lemma \ref{lemma:guv_plus_gvu}.
Finally, using Lemma \ref{lem:square_scale} and \ref{lem:gu1_summing} we get 
\begin{equation*}
[g_{-u,-v}] 
= [g_{-u/v,-1}] 
= [g_{-u/v,1}]\oplus [g_{1,-1}] 
= [g_{-1,1}] \oplus [g_{u/v,1}] \oplus [g_{1,-1}] 
= [g_{u/v,1}].
\end{equation*}
This implies every element of $\GG_0(\bbR)$ can be written as a sum $n[g_{-1,1}]\oplus [g_{u,1}]$ with $u>0$ and $n\in \bbZ$. 
By Proposition \ref{prop:gu1_distinct} we know that $[g_{u,1}]\neq [g_{v,1}]$ for $u\neq v \in \bbR^\times$. 
By Lemma \ref{lem:gu1_summing} we get the relation $\left(n[g_{-1,1}]\oplus [g_{u,1}]\right) \oplus \left(m[g_{-1,1}]\oplus [g_{v,1}]\right)= (n+m)[g_{-1,1}] \oplus [g_{uv,1}]$ when $u,v > 0$. 
Hence the map $[u] \mapsto [g_{u,1}]$ is a well-defined homomorphism which is surjective and injective. 
Thus $\kappa$ is an isomorphism in this case. 
This finishes the proof.
\end{proof}

%%%
\appendix 

\section{Affine representability for pointed spaces and homotopies}\label{sec:appendix}

In this section, we discuss the results of Asok, Hoyois, and Wendt in \cite{AHW1}, \cite{AHW2}, and how we apply them. 
While the definition of our proposed group operation on $[\J, \PP]\naif$ in Definition \ref{def:full_group} is independent of motivic homotopy theory and the results of \cite{AHW2},  
we use the affine representability results of \cite{AHW2} to compare our group operation with the conventional group structure on $[\PP, \PP]^{\AAA}$. 
A minor technical point to overcome is that Asok, Hoyois, and Wendt work in the unpointed motivic homotopy category, whereas we need the analogous results in the pointed setting. 
The purpose of this appendix is to explain how the pointed analogs can be deduced. 
To keep the presentation brief, we use the conventions and notation of the papers \cite{AHW1} and \cite{AHW2}. 
We thank Marc Hoyois for helpful comments. 

Recall that we denote by $\Sm_k$ the category of smooth finite type $k$-schemes. 
Let $\sPrek$ denote the category of simplicial presheaves on $\Sm_k$. 
Let $\sPrepk$ denote the category of pointed simplicial presheaves on $\Sm_k$. 
We refer to an object in $\sPrek$ (respectively in $\sPrepk$) as a (pointed) motivic space.  
We equip with the Nisnevich-local $\AAA$-model structure as in \cite{AHW2}. 
For a motivic space $\calY$, let $\Sing^{\AAA} \calY$ denote the singular functor defined in \cite[\S 4.1]{AHW1},  see also \cite[page 88]{MV99}. 
If $\calY$ is pointed by a morphism $y \colon \Spec{k} \to \calY$,  
then the pointed singular functor $\Singp^{\AAA} \calY$ is defined as the fiber over $y$. 
More precisely, let $X$ be a pointed smooth $k$-scheme pointed by the morphism $x \colon \Spec{k} \to X$, then 
$\Singp^{\AAA} \calY$ is determined by the pullback square of simplicial sets 
\begin{equation}\label{eq:pullback_square_sSets_appendix}
\xymatrix{
(\Singp^{\AAA} \calY)(X) \ar[r] \ar[d] &  (\Sing^{\AAA} \calY)(X) \ar[d]^{x^*} \\
\mathrm{point} = (\Sing^{\AAA} \Spec{k})(\Spec{k}) \ar[r]^-{y_*} & (\Sing^{\AAA} \calY)(\Spec{k}).
}
\end{equation}
We note that, on $0$-simplices, $x$ induces induces a map of sets $x^* \colon \sPrek(X,\calY) \to \sPrek(\Spec{k},\calY)$. 
Hence, $(\Singp^{\AAA} \calY)(X)_0$ is the set  $\sPrepk(X,\calY)$ of pointed morphisms $X \to \calY$. On $1$-simplices, $x$ induces a map of sets $x^* \colon \sPrek(X \times \AAA,\calY) \to \sPrek(\AAA,\calY)$. 
Hence, $(\Singp^{\AAA} \calY)(X)_1$ is the set of pointed naive $\AAA$-homotopies of pointed morphisms $X \to \calY$. 

\begin{remark}\label{rem:naive_htpy_classes_and_pointed_pi0_appendix}
In particular, if $\calY=Y$ is represented by a pointed smooth finite type $k$-scheme $Y$, the set $\pi_0((\Singp^{\AAA} Y)(X))$ is the set of pointed naive homotopy classes of pointed morphisms $X \to Y$ described in Section \ref{sec:naive_htpy_relation}, that is,
\begin{align*}
\pi_0((\Singp^{\AAA} Y)(X)) = [X,Y]\naif.
\end{align*}
\end{remark}

We recall the following definition from \cite{AHW2}: 

\begin{definition}\label{def:A1_naive_appendix} 
\cite[Definition 2.1.1]{AHW2}
Let $\calF \in \sPrek$ and let $\calF \to \widetilde{\calF}$ be a fibrant replacement in the $\AAA$-model structure on $\sPrek$. 
There is a canonical map $\Sing^{\AAA}\calF \to \widetilde{\calF}$ that is well-defined up to simplicial homotopy equivalence. 
Then $\calF \in \sPrek$ is called \emph{$\AAA$-naive} if the map $(\Sing^{\AAA}\calF)(X) \to \widetilde{\calF}(X)$ is a weak equivalence of simplicial sets for every affine smooth finitely presented $k$-scheme $X$. 
\end{definition}

We will now show how the unpointed notion of $\AAA$-naivity of Definition \ref{def:A1_naive_appendix} translates to the pointed setting. 

\begin{proposition}\label{prop:appendix_main_result}
Let $\calY \in \sPrepk$ be a pointed motivic space. 
Assume that the underlying unpointed motivic space $\calY$ is $\AAA$-naive. 
Then, for every affine pointed smooth finitely presented $k$-scheme $X$, the canonical map $\pi_0((\Singp^{\AAA}\calY)(X)) \xrightarrow{\cong} [X,\calY]^{\AAA}$ is a bijection.   
\end{proposition}
\begin{proof}
Let $(X,x)$ be a pointed smooth $k$-scheme, and  
let $p \colon X \to \Spec{k}$ denote the canonical morphism. 
Then $p$ induces a map $p^* \colon (\Sing^{\AAA}\calY)(\Spec{k}) \to (\Sing^{\AAA}\calY)(X)$ of simplicial sets such that $x^* \circ p^*$ is the identity on $(\Sing^{\AAA}\calY)(\Spec{k})$. 
This shows that the map $x^*$ is a Kan fibration. Since the Kan--Quillen model structure on simplicial sets is right proper, this implies that 
$(\Singp^{\AAA}\calY)(X)$ is the homotopy fiber of $x^*$. 

Let $\calY \to R_{\AAA}\calY$ be a fibrant replacement of $\calY$ in the $\AAA$-model structure on $\sPrepk$. 
After forgetting the basepoint, $R_{\AAA}\calY$ is fibrant in the $\AAA$-model structure on the category $\sPrek$ of unpointed motivic spaces. 
Since the singular functor preserves $\AAA$-fibrations, $\Sing^{\AAA}R_{\AAA}\calY$ is fibrant and we may assume $\widetilde{\calY} = \Sing^{\AAA}R_{\AAA}\calY$. 
Moreover, we get the following commutative diagram of simplicial sets which, by the above argument, is a morphism of homotopy fiber sequences for every pointed smooth $k$-scheme $(X,x)$. 
\begin{equation}\label{eq:homotopy_fiber_diagram_appendix}
\xymatrix{
(\Singp^{\AAA}\calY)(X) \ar[r] \ar[d] & (\Singp^{\AAA}R_{\AAA}\calY)(X) \ar[d] \\
(\Sing^{\AAA}\calY)(X) \ar[r] \ar[d]_-{x^*} & (\Sing^{\AAA}R_{\AAA}\calY)(X) \ar[d]^-{x^*} \\
(\Sing^{\AAA}\calY)(\Spec{k}) \ar[r] & (\Sing^{\AAA}R_{\AAA}\calY)(\Spec{k}) 
}
\end{equation}

Now we assume that the underlying simplicial presheaf of $\calY$ is $\AAA$-naive and that  $X$ is affine.  
Since $\calY$ is $\AAA$-naive and both $X$ and $\Spec{k}$ are affine, the horizontal maps in the middle and at the bottom are weak equivalences of simplicial sets. 
Thus, since Diagram \eqref{eq:homotopy_fiber_diagram_appendix} is a morphism of homotopy fiber sequences, the top horizontal map is a weak equivalence of simplicial sets. 
Hence it induces a bijection on $\pi_0$. 
Since $\pi_0((\Singp^{\AAA}R_{\AAA}\calY)(X)) = [X,\calY]^{\AAA}$, this proves the assertion. 
\end{proof}

\begin{lemma}\label{lemma:J_is_A1_naive_appendix}
The smooth $k$-schemes $\J$ and $\PP$ are $\AAA$-naive. 
\end{lemma}
\begin{proof}
Let $Q_2$ be the smooth affine quadric
over $\bbZ$ defined by $xy = z(z + 1)$. 
By \cite[Theorem 4.2.2]{AHW2}, $Q_2$ is $\AAA$-naive. 
The scheme endomorphism of $\Spec{\bbZ[x,y,z]}$ given by the ring homomorphism defined by sending $x \mapsto z,\, y \mapsto -y,\, z \mapsto -x$ induces an isomorphism $Q_2 \cong \calJ$. 
Hence $\calJ$ is $\AAA$-naive. 
By \cite[Lemma 4.2.4]{AHW2} an affine torsor bundle over a base space is $\AAA$-naive if and only if the base space is $\AAA$-naive.   
Since $\calJ $ is $\AAA$-naive and an affine torsor bundle over $\PP$, it follows that $\PP$ is $\AAA$-naive.
\end{proof}

\begin{proposition}\label{prop:pointed_naive_htpy_conclusion_appendix}
The canonical map $\nu \colon [X,\PP]\naif \xrightarrow{\cong} [X,\PP]^{\AAA}$ is a bijection  for every affine pointed finitely presented smooth $k$-scheme $X$.   
\end{proposition}

\begin{proof}
The proposition follows from Remark \ref{rem:naive_htpy_classes_and_pointed_pi0_appendix} and Proposition \ref{prop:appendix_main_result}, since Lemma \ref{lemma:J_is_A1_naive_appendix} shows $\PP$ is $\AAA$-naive. 
\end{proof}

For $X=\J$, the Proposition \ref{prop:pointed_naive_htpy_conclusion_appendix} yields the comparison of the sets $[\J,\PP]\naif$ and $[\J,\PP]^{\AAA}$ of pointed homotopy classes that we wanted.

%%%

\section{Facts about the resultant}\label{sec:appendix_resultants}

In Sections \ref{sec:pointed_maps_p1_to_p1_and_j_to_p1} and  \ref{sec:computations_for_degree_0_maps} we used the following facts about the resultant for which we now provide references or proofs.  
All results can  be found or deduced from \cite[Chapter IV]{Bourbaki2003}. 

Throughout this section we let $S$ be an integral domain and let $A=a_nX^n + a_{n-1}X^{n-1} + \ldots + a_0$ and $B=b_nX^n + b_{n-1}X^{n-1} + \ldots + b_0$ be polynomials over $S$ in the indeterminate $X$.

\begin{definition}
\label{def:sylvesterAndRes}
    The Sylvester matrix $\Syl(A,B)$ is defined as follows
    \begin{equation*}
        \begin{pmatrix}
            a_{n} & 0 & \ldots & 0 & b_n & 0 & \ldots & 0\\ 
            a_{n-1} & a_{n} & & \vdots & b_{n-1} & b_n & & \vdots\\
            \vdots & \vdots & \ddots & \vdots & \vdots & \vdots & \ddots \\
            a_1 & a_2 & \hdots & a_{n} & b_1 & b_2 & \hdots & b_n \\
            a_0 & a_1 & \hdots & a_{n-1} & b_0 & b_1 & \hdots & b_{n-1}  \\
            0 & a_0 & \hdots & a_{n-2} & 0 & b_0 & \hdots & b_{n-2} \\ 
            \vdots & & \ddots & \vdots & \vdots & & \ddots & \vdots \\ 
        0 & 0 & \hdots  & a_0 & 0 & 0 & \hdots  & b_0
            
        \end{pmatrix}.
    \end{equation*}
    and we define $\res(A,B) := \det(\Syl(A,B)).$
\end{definition}

\begin{lemma}[Remark 4 on page IV.76 in \cite{Bourbaki2003}]
\label{lem:resultant_bezout}
Assume $\res(A,B)\in S^\times$. 
Then there exist polynomials $U,V \in S[X]$ such that $AU+BV=1$.
\end{lemma}

\begin{lemma}[Remark 1 on page IV.76 in \cite{Bourbaki2003}] \label{lem:reverse_order_resultant}
Let $\tilde A = a_n + a_{n-1}X + \ldots + a_0X^n$ and $\tilde B = b_n + b_{n-1}X + \ldots + b_0X^n$ be the reversed polynomials of $A$ and $B$. 
If $\res(A,B)\in S^\times$, then $\res(\tilde A, \tilde B)= (-1)^n \res(A,B) \in S^\times$.
\end{lemma}

\begin{lemma}[Remark 5 on page IV.77 in \cite{Bourbaki2003}]
\label{lem:resultant_conservation_elementary_action} 
Assume $\res(A,B)\in S^\times$. 
Let $C \in S[X]$ be a polynomial such that $\deg(A)\geq \deg(BC)$. 
Then we have  $\res(A + BC, B)=\res(A,B)$.
\end{lemma}

\begin{lemma}\label{lem:resultantconservation}
Assume $\res(A,B)\in S^\times$ and that $A$ is monic. 
Then we have 
\begin{equation*}
\res\left(AX - \frac{1}{u}B, uA\right)=-u\cdot \res(A,B) ~ \text{for all} ~ u \in S^\times.
\end{equation*}
\end{lemma}
\begin{proof}
The strategy of the proof is as follows: We determine the Sylvester matrix for the pair $\left(AX - \frac{1}{u}B, uA\right)$ and will then use elementary row and column operations to confirm that it has the determinant claimed. 

First note that $\res\left(AX - \frac{1}{u}B, uA\right) = u^{n+1}\res\left(AX - \frac{1}{u}B, A\right)$. 
Let $A=\sum_{i=0}^n a_iX^i$, and $B=\sum_{i=0}^n b_iX^i$. 
Let $c_i = a_{i-1}-\frac{1}{u}b_i$ for $0\leq i \leq {n+1}$ and set $a_{-1} = b_{n+1} = 0$. Then $AX - \frac{1}{u}B = \sum_{i=0}^{n+1}c_iX^i$. 
The Sylvester matrix for the pair  $\left(AX - \frac{1}{u}B, A\right)$ is 
    \begin{equation*}
        \begin{pmatrix}
            c_{n+1} & 0 & \ldots & 0 & 0 & 0 & \ldots & 0\\ 
            c_n & c_{n+1} & & \vdots & a_n & 0 & & \vdots\\
            \vdots & \vdots & \ddots & \vdots & \vdots & \ddots & \\
            c_1 & c_2 & \hdots & c_{n+1} & a_1 & \hdots & a_n & 0 \\
            c_0 & c_1 & \hdots & c_n & a_0 & \hdots & a_{n-1} & a_n  \\
            0 & c_0 & \hdots & c_{n-1} & 0 & a_0 & \hdots & a_{n-1} \\ 
            \vdots & & \ddots & \vdots & \vdots & & \ddots & \vdots \\ 
        0 & 0 & \hdots  & c_0 & 0 & 0 & \hdots  & a_0
            
        \end{pmatrix}.
    \end{equation*}
Since $c_{n+1}=1$, we can remove the first row and first column to obtain a submatrix with the same determinant, namely
    \begin{equation*}
        \begin{pmatrix}
            c_{n+1} & 0 & \ldots & 0 & a_{n} & 0 & \ldots &0& 0\\ 
            c_n & c_{n+1} & & \vdots & a_{n-1} & a_n & \ddots& & \vdots\\
            \vdots &\vdots  & \ddots & \vdots & \vdots & \ddots &\ddots& \vdots& \\
            c_2 & c_3 & \hdots & c_{n+1} & a_1 & a_2 &\hdots & a_n & 0 \\
            c_1 & c_2 & \hdots & c_n & a_0 & a_1 & \hdots &a_{n-1} & a_n  \\
            c_0 & c_1 & \hdots & c_{n-1} & 0 & a_0 & \hdots & a_{n-2} & a_{n-1}  \\ 
            \vdots & c_0 & \ddots & \vdots & \vdots & & \ddots & \vdots &\vdots \\  
            0 & 0 &  \ddots & c_1 & 0 & 0 & \hdots  & a_0 & a_1\\
            0 & 0 & \hdots  & c_0 & 0 & 0 & \hdots  & 0& a_0
            
        \end{pmatrix}.
    \end{equation*}
Subtracting column $1$ from column $n+1$ yields
    \begin{equation*}
        \begin{pmatrix}
            c_{n+1} & 0 & \ldots & 0 & 0 & 0 & \ldots &0& 0\\ 
            c_n & c_{n+1} & & \vdots & \frac{1}{u}b_n & a_n & \ddots&\ddots& 0\\
            \vdots &\vdots  & \ddots & \vdots & \vdots & \vdots &\ddots& \\
            c_2 & c_3 & \hdots & c_{n+1} & \frac{1}{u}b_2& a_2 &\hdots & a_n & 0 \\
            c_1 & c_2 & \hdots & c_n & \frac{1}{u}b_1 & a_1 & \hdots &a_{n-1} & a_n  \\
            c_0 & c_1 & \hdots & c_{n-1} & \frac{1}{u}b_0 & a_0 & \hdots & a_{n-2} & a_{n-1}  \\ 
            \vdots & c_0 & \ddots & \vdots & \vdots & & \ddots & \vdots &\vdots \\  
            0 & 0 &  \ddots & c_1 & 0 & 0 & \hdots  & a_0 & a_1\\
            0 & 0 & \hdots  & c_0 & 0 & 0 & \hdots  & 0& a_0
            
        \end{pmatrix}.
    \end{equation*}
Once again, the determinant of this matrix is the same as that of the submatrix where the first row and first and column removed. We remove them and obtain
    \begin{equation*}
        \begin{pmatrix}
            c_{n+1} & 0 & \ldots & 0& \frac{1}{u}b_n & a_{n} & 0 & \ldots &0& 0\\ 
            c_n & c_{n+1} & & \vdots & \frac{1}{u}b_{n-1}& a_{n-1} & a_n & \ddots& & \vdots\\
            \vdots &\vdots  & \ddots &  \vdots &\vdots& \vdots & \ddots &\ddots& \vdots& \\
            c_3 & c_4 & \hdots & c_{n+1}&\frac{1}{u}b_2 & a_2 & a_3 &\hdots & a_n & 0 \\
            c_2 & c_3 & \hdots & c_n & \frac{1}{u}b_1& a_1 & a_2 & \hdots &a_{n-1} & a_n  \\
            c_1 & c_2 & \hdots & c_{n-1} & \frac{1}{u}b_0&a_0& a_1 & \hdots & a_{n-2} & a_{n-1}  \\ 
            c_0 & c_1 & \hdots & c_{n-2} & 0&0 & a_0 &\hdots  & a_{n-3} & a_{n-2}\\
            \vdots & c_0 & \ddots & \vdots & \vdots & && \ddots & \vdots &\vdots \\  
            0 & 0 &  \ddots & c_1 & 0 & 0 &0& \hdots  & a_0 & a_1\\
            0 & 0 & \hdots  & c_0 & 0 & 0 &0& \hdots  & 0& a_0
        \end{pmatrix}.
    \end{equation*}
We subtract column $n+i$ from column $i$ for each $i<n$, and the result is that at each entry $c_i$, we get instead $c_i-a_{i-1}= -\frac{1}{u}b_i$. 
    \begin{equation*}
        \begin{pmatrix}
            -\frac{1}{u}b_{n+1} & 0 & \ldots & 0& \frac{1}{u}b_n & a_{n} & 0 & \ldots &0& 0\\ 
            -\frac{1}{u}b_n & -\frac{1}{u}b_{n+1} & & \vdots & \frac{1}{u}b_{n-1}& a_{n-1} & a_n & \ddots& & \vdots\\
            \vdots &\vdots  & \ddots &  \vdots &\vdots& \vdots & \ddots &\ddots& \vdots& \\
            -\frac{1}{u}b_3 & -\frac{1}{u}b_4 & \hdots & -\frac{1}{u}b_{n+1}&\frac{1}{u}b_2 & a_2 & a_3 &\hdots & a_n & 0 \\
            -\frac{1}{u}b_2 & -\frac{1}{u}b_3 & \hdots & -\frac{1}{u}b_n & \frac{1}{u}b_1& a_1 & a_2 & \hdots &a_{n-1} & a_n  \\
            -\frac{1}{u}b_1 & -\frac{1}{u}b_2 & \hdots & -\frac{1}{u}b_{n-1} & \frac{1}{u}b_0&a_0& a_1 & \hdots & a_{n-2} & a_{n-1}  \\ 
            -\frac{1}{u}b_0 & -\frac{1}{u}b_1 & \hdots & -\frac{1}{u}b_{n-2} & 0&0 & a_0 &\hdots  & a_{n-3} & a_{n-2}\\
            \vdots & -\frac{1}{u}b_0 & \ddots & \vdots & \vdots & && \ddots & \vdots &\vdots \\  
            0 & 0 &  \ddots & -\frac{1}{u}b_1 & 0 & 0 &0& \hdots  & a_0 & a_1\\
            0 & 0 & \hdots  & -\frac{1}{u}b_0 & 0 & 0 &0& \hdots  & 0& a_0
        \end{pmatrix}
    \end{equation*}
    
Then multiplying the first $n$ columns by $u$ and applying a cyclic permutation of the $n$ first columns yields the Sylvester matrix of the pair $(B,A)$. 
The sign of the permutation is $(-1)^{n-1}$. 
Interchanging column $i$ with $n+i$ for all $i\leq n$ yields $(A,B)$, and this needed another permutation of sign $(-1)^{n-1}$, so the signs cancel out.
\end{proof}

%%%

\section{Testing compatibility via real realization and signatures} 
\label{sec:evidence}

Now we provide the additional evidence for Conjecture \ref{conj:main_conjecture} and a positive answer to Question \ref{question:tilde_pi} referred to in Section \ref{sec:proposed_group_structure}.

We assume that $k$ is a subfield of $\R$. 
Let $\Hop(k)$ denote the homotopy category of pointed smooth $k$-schemes and let $\Hop$ be the homotopy category of pointed topological spaces. 
By \cite{MV99} sending a smooth $k$-scheme $X$ to the topological space $X(\R)$ equipped with its usual structure of a smooth manifold extends to a functor $\Re \colon \Hop(k) \fd \Hop$, see also  \cite[page 14]{asok2020motivic} and \cite[Section 5.3]{dugger2004topological}.

For a smooth map $f$ between oriented compact smooth manifolds of the same dimension, let $\degt(f) \in \Z$ denote the topological Brouwer degree of $f$. 
In \cite{Morel_ICM} Morel describes the analog of the topological degree map in $\AAA$-homotopy theory. 
For endomorphisms of $\PP$ it defines a homomorphism 
\[
\deg^{\AAA} \colon [\bbP^1, \bbP^1]^{\AAA} \to \GW(k).
\]
Let $f \colon \PP \to \PP$ be a morphism. 
Since we assume that $k$ is a subfield of $\R$, 
we can form the real realization $\Re(f) \colon \PP(\R) \to \PP(\R)$. 
Following Morel, the signature, denoted $\mathrm{sgn}$, of the quadratic form given by the $\AAA$-Brouwer degree of $f$ equals the topological Brouwer degree of $\Re(f)$, i.e., 
\begin{align}\label{eq:signature_degA1_deg_of_real}
\mathrm{sgn}\left(\deg^{\AAA}(f)\right) = \degt(\Re(f)).
\end{align}
We note that, in some form, this was also shown by Eisenbud, Levine, and  Teissier in \cite[Theorem 1.2]{eisenbud1977algebraic} for the local degree of maps between real affine spaces. 
The latter approach has been incorporated into the motivic theory by Kass and Wickelgren \cite{KassWickelgren}.

The motivic Brouwer degree map $\deg^{\AAA}$ is a homomorphism for the conventional group structure $\oplus^{\AAA}$ on $[\bbP^1, \bbP^1]^{\AAA}$, and the signature is additive. 
Hence, for morphisms $f, g \colon \PP \to \PP$ and their sum $f \oplus^{\AAA} g$ in $[\bbP^1,\bbP^1]^{\AAA}$, Equation
\eqref{eq:signature_degA1_deg_of_real} implies 
\begin{align}\label{eq:singature_degree_is_additive}
\degt\left(\Re\left(f \oplus^{\AAA} g\right)\right) 
& = \mathrm{sgn}\left(\deg^{\AAA}\left(f \oplus^{\AAA} g \right) \right)\\
\nonumber & = \degt(\Re(f)) + \degt(\Re(g)).
\end{align}
We will now use this fact to test the compatibility of the action of Definition \ref{def:group_action_homotopy} and thereby of Definition \ref{def:full_group} with the conventional group structure in the following way. 

The real points $\J(\R)$ of $\J$ form a surface in $\R^3$ given by the equation $x(1-x)-yz=0$.  
The intersection with the plane defined by $y=z$ is the circle given by the set of points satisfying $x(1-x)-y^2=0$. 
Its center is the point $(1/2,0,0)\in \R^3$. 
We parameterize this circle via the map $\gamma\colon  \mathbb{S}^1 \to \J(\R)$ given  by 
\[
\gamma\colon  \theta \td \Big( 1/2+ \cos (\theta)/2, \sin (\theta)/2, \sin (\theta)/2 \Big).
\]

The real realization of $\PP$ is the topological real projective line $\mathbb{RP}^1$. 
Hence, for a morphism $f \colon \J \to \PP$, we may form the composition $\Re(f) \circ \gamma$ which is a smooth map $\mathbb{S}^1 \to \mathbb{RP}^1$. We can then apply the topological Brouwer degree to the composition $\Re(f) \circ \gamma$. 
Since the real realization of a naive homotopy induces a homotopy of maps between  topological spaces, 
this induces a well-defined map 
\begin{align*}
\degt (\Re(-) \circ \gamma) \colon [\J, \PP]\naif & \longrightarrow \Z \\
f ~ ~ ~ & \longmapsto \degt(\Re(f) \circ \gamma). 
\end{align*}

\begin{lemma}\label{lemma:top_degree_signature_commute}
The following diagram commutes. 
\begin{align}\label{eq:top_degree_signature_commute}
\xymatrix{
\ar@/_2.5pc/[dd]_-{\cc} [\J, \PP]\naif 
\ar[d]_-{\nu} \ar@/^1pc/[drr]^-{~ ~ ~ \degt\left(\Re(-)\circ \gamma\right)} & & \\
[\J, \PP]^{\AAA} \ar[d]^-{(\pi^*)^{-1}} \ar[rr]^-{\degt\left(\Re(-)\circ \gamma\right)} & & \bbZ \\ 
[\bbP^1, \bbP^1]^{\AAA} \ar@/^/[u]^-{\pi^*} \ar@/_1pc/[urr]_-{~ ~ ~ \degt\left(\Re(-)\right)} & & 
}   
\end{align}
\end{lemma}
\begin{proof}
To prove the assertion it suffices to show that both parts of the diagram commute. 
The functor $\Re$ commutes with the canonical map $\nu \colon [\J, \PP]\naif \to [\J, \PP]^{\AAA}$. 
This implies that the upper part commutes. 
We verify in Example \ref{ex:realization pi} that for $\pi \colon \J \to \PP$ the composite map $\Re(\pi) \circ \gamma \colon \mathbb{S}^1 \to \mathbb{RP}^1$ is an orientation preserving diffeomorphism. 
Now let $f \colon \PP \to \PP$ be a morphism.  
Since the composition with $\Re(\pi) \circ \gamma$ preserves degrees, we obtain the identity 
\begin{align*}
\degt(\Re(f \circ \pi) \circ \gamma) = \degt(\Re(f) \circ \Re(\pi) \circ \gamma) = \degt(\Re(f)).   
\end{align*}
This implies that the lower part of Diagram \eqref{eq:top_degree_signature_commute} commutes and finishes the proof. 
\end{proof}

This implies the following necessary condition for the compatibility of the operations $\oplus$ and $\oplus^{\AAA}$: 

\begin{proposition}\label{prop:nec_condition}
Assume that $\cc$ is a group homomorphism.  
Then we have 
\begin{align*}
\degt(\Re(\cc(f \oplus g))) 
= \degt(\Re(f)\circ \gamma) + \degt(\Re(g)\circ \gamma). 
\end{align*}
\end{proposition}
\begin{proof}
The assumption that $\cc$ is a group homomorphism implies  
 \begin{align*}
\degt(\Re(\cc(f \oplus g))) 
= \degt(\Re(\cc(f) \oplus^{\AAA} \cc(g))). 
\end{align*}
Identity \eqref{eq:singature_degree_is_additive} implies 
 \begin{align*}
\degt(\Re(\cc(f) \oplus^{\AAA} \cc(g))) 
= \degt(\Re(\cc(f))) + \degt(\Re(\cc(g))).  
\end{align*}
Commutativity of Diagram \eqref{eq:top_degree_signature_commute} implies 
 \begin{align*}
\degt(\Re(\cc(f))) + \degt(\Re(\cc(g)))  = \degt(\Re(f)\circ \gamma) + \degt(\Re(g)\circ \gamma). 
\end{align*}
Putting these identities together yields the assertion. 
\end{proof}

As a special case, we get the following necessary condition for the compatibility of $\oplus$ with the conventional group structure:   
\begin{corollary}\label{cor:concrete_obstruction}
Let $F \colon \J \to \PP$ be a pointed morphism. 
Then, if $\cc$ is a group homomorphism, 
we must have 
\begin{align*}
\degt(\Re(F\oplus \pi)\circ \gamma)  
= \degt(\Re(F)\circ \gamma) + 1.
\end{align*}
\end{corollary}

In the following section we will apply Corollary \ref{cor:concrete_obstruction} in a concrete case in Example \ref{ex:realization_signature_3_map}. 

Moreover, we exclude a potential alternative to the operation $\oplus$ of Definition \ref{def:group_action_homotopy} in Example \ref{ex:group_action_justification}.    

\begin{remark}\label{rem:concrete_obstruction_tilde_pi} 
Let $F \colon \J \to \PP$ again be a pointed morphism. 
Assume that Question \ref{question:tilde_pi} has a positive answer, i.e., assume that $\tilde{\pi}$ is naively homotopic to $-\pi$. 
Then Proposition \ref{prop:nec_condition} shows that, if $\cc$ is a group homomorphism, then we must expect to get 
\begin{align}\label{eq:top_degree_and_tilde_pi}
\degt(\Re(F\oplus \tilde{\pi})\circ \gamma) 
= \degt(\Re(F)\circ \gamma) - 1. 
\end{align}
Note that, since we do not know whether $\tilde{\pi}$ is naively homotopic to $-\pi$, Equation \eqref{eq:top_degree_and_tilde_pi} may fail to hold for some $F$ even though $\cc$ is a group homomorphism. 

However, we confirm Formula  \eqref{eq:top_degree_and_tilde_pi} in a concrete case in Example \ref{ex:realization_F_acts_on_pi_tilde}. 
\end{remark} 

We will now compute the topological degrees and thereby the signatures of several maps and apply the previous observations. 
For the following computations we will often identify the ring $R$ with the ring $k[x,y,z]/(x(1-x)-yz)$ where it is convenient.

\begin{example}\label{ex:degree_of_real_of_F}
Consider the morphism $g_{1,-1} \colon  \J \fd \PP$ defined by the unimodular row $(2x-1, 2y)$. 
Its real realization is the map $\Re (g_{1,-1})\colon \Re(\J) \fd \Re(\PP)$ defined by
\begin{align*}
\Re (g_{1,-1})\colon  (x, y, z) \td [2x-1: 2y]. 
\end{align*}
Precomposing with $\gamma$ gives
\begin{align*} 
\Re (g_{1,-1}) \circ \gamma \colon \theta \td 
\left[\cos (\theta):\sin (\theta)\right],
\end{align*}
which is the usual double cover of $\mathbb{RP}^1$ by $\mathbb{S}^1$ and has topological Brouwer degree 2. 
\end{example}

As explained in Section \ref{sec:pointed_maps_p1_to_p1_and_j_to_p1}, a morphism $f\colon  \J \to \PP$ may be described by gluing together partially defined maps on open subsets. 
In the following examples we will define a morphism
$\Re(f)\circ \gamma\colon  \mathbb{S}^1\to \mathbb{RP}^1$ 
by gluing $\Re(f|_{D(x)})\circ \gamma \colon \gamma\inv(\Re(D(x))) \to \mathbb{RP}^1$ and 
$\Re(f|_{D(1-x)})\circ \gamma \colon \gamma\inv(\Re(D(1-x))) \to \mathbb{RP}^1$ 
on their overlaps in the respective domains. 

\begin{example}
\label{ex:realization pi}
The real realization of $\pi \colon  \J \fd \PP$ is defined on $\Re(D(x))$ by
\begin{align*}
\Re (\pi|_{D(x)})\colon  (x, y, z) &\td [x:y], 
\end{align*}
and on $\Re(D(1-x))$ by
\begin{align*}
\Re (\pi|_{D(1-x)})\colon  (x, y, z) &\td [z:1-x]. 
\end{align*}
Precomposing with $\gamma$ gives
\begin{align*} 
\Re (\pi|_{D(x)}) \circ \gamma \colon \theta
&\td 
[ 1/2+ \cos (\theta)/2 : \sin (\theta)/2], \\
\Re (\pi|_{D(1-x)}) \circ \gamma \colon \theta
&\td 
[ \sin (\theta)/2: 1/2- \cos (\theta)/2],
\end{align*}
which glue together to give a map of degree 1:  
\begin{equation*}
    (\Re (\pi) \circ \gamma) (\theta) = \begin{cases}[1+\cos(\theta):\sin(\theta)] & \theta \neq \pi, \\
    [0:1] & \theta = \pi. \end{cases}
\end{equation*}
This shows that $\Re (\pi) \circ \gamma$ is an orientation preserving diffeomorphism and has topological Brouwer degree $1$.  
\end{example}

In the following example we test the necessary condition of Corollary \ref{cor:concrete_obstruction} in a concrete case. 

\begin{example}\label{ex:realization_signature_3_map}
Recall that the unimodular row $g_{1,-1} = (2x-1, 2y)$ can be augmented to the following matrix with determinant 1:
\begin{equation*}
m_{1,-1} = \begin{pmatrix}
2x-1 &-2z\\
2y & 2x-1
\end{pmatrix}.
\end{equation*} 
The group action of Definition \ref{def:group_action_homotopy} yields the map $F := g_{1,-1} \oplus \pi = (2x-1,-2z:2y,2x-1)_1$. 

Taking real realization and 
precomposing with $\gamma$ yields the map $\Re (F) \circ \gamma \colon \mathbb{S}^1 \to \mathbb{RP}^1$  
given by 
\begin{equation*}
    (\Re (F) \circ \gamma) (\theta) = \begin{cases}[\cos(\theta)+\cos(2\theta):\sin(\theta)+\sin(2\theta)] & \theta \neq \pi, \\
    [0:1] & \theta = \pi. \end{cases}
\end{equation*}
The topological degree of this map is $3$.  

Hence our computation confirms that 
\begin{align*}
\degt(\Re(F) \circ \gamma) = 3 = 2 + 1 = \degt(\Re(g_{-1,1})\circ \gamma) + \degt(\Re(\pi)\circ \gamma),     
\end{align*}
as required for the compatibility of $\oplus$ with $\oplus^{\AAA}$. 
\end{example}

The next example confirms that the signature of the motivic Brouwer degree of $\tilde{\pi}$ has the value $-1$ as expected if Question \ref{question:tilde_pi} has a positive answer. 

\begin{example}
\label{ex:realization_pi_tilde}
The real realization of the morphism $\tilde\pi = (1,0:0,-1)_{-1} \colon \J \to \PP$ is defined on $\Re(D(x))$ and $\Re(D(1-x))$ respectively by
$\Re(D(x))$ by
\begin{align*}
\Re (\tilde\pi|_{D(x)})\colon  (x, y, z) &\td [x:-z], \\
\Re (\tilde\pi|_{D(1-x)})\colon  (x, y, z) &\td [y:-1+x]. 
\end{align*}
Precomposing with $\gamma$ gives
\begin{align*} 
\Re (\tilde\pi|_{D(x)}) \circ \gamma \colon \theta
&\td 
[ 1/2+ \cos (\theta)/2 : -\sin (\theta)/2], \\
\Re (\tilde\pi|_{D(1-x)}) \circ \gamma \colon \theta
&\td 
[ \sin (\theta)/2: -1/2+ \cos (\theta)/2],
\end{align*}
which glue together to give a map of topological Brouwer degree $-1$.  
\end{example}

Now we confirm that Identity \eqref{eq:top_degree_and_tilde_pi} of Remark \ref{rem:concrete_obstruction_tilde_pi} does hold in an example. 

\begin{example}
\label{ex:realization_F_acts_on_pi_tilde} 
Consider the unimodular row $F=(2x-1,2z)$ which can be augmented to the following matrix   
\begin{equation*}
    M = \begin{pmatrix}
2x-1 &-2y\\
2z & 2x-1
\end{pmatrix}
\end{equation*} 
with determinant $1$.  
Note that $F$ is homotopic to $g_{1,-1}$ by Lemma \ref{lem:ABeqUV}. 
We let $F$ act on $\tilde \pi$ via the action of Definition \ref{def:group_action_homotopy}. 
This yields the map $L := F \oplus \tilde{\pi} = (2x-1,2y:2z,-2x+1)_{-1}$. 
Precomposing its real realization with $\gamma$ yields
the same map as in Example \ref{ex:realization pi} where we showed it has topological Brouwer degree $1$. 

Hence our computation confirms  
\begin{align*}
\degt(\Re(L) \circ \gamma) = 1 = 2 - 1  =  \degt(\Re(F)\circ \gamma) +  \degt(\Re(\tilde{\pi})\circ \gamma), 
\end{align*}
as required in Remark \ref{rem:concrete_obstruction_tilde_pi}. 
\end{example}

\begin{example}
\label{ex:group_action_justification} 
Consider now an alternative action $\boxplus$ of $[\J, \PP]\naif_0$ on $[\J, \PP]\naif$ defined as follows. For $[(A,B)] \in [\J, \ato]\naif$ and $[s_0,s_1] \in [\J, \PP]\naif_n$, extend the unimodular row $(A,B)$ to a matrix $M$ in $\SL_2(R)$ and define
\[
[(A,B)] \boxplus [s_0,s_1] := [M^T\cdot (s_0,s_1)^T].
\]

Again we look at the unimodular row $F=(A,B) = (2x-1, 2z)$ and the matrix $M$ of Example \ref{ex:realization_F_acts_on_pi_tilde}. 
The action $\boxplus$ of $F$ on $\pi$ yields the morphism $H := F \boxplus \pi = (2x-1, 2z : -2y, 2x-1)_1=(1,0:0,-1)_1$. 
Taking real realization and precomposing with $\gamma$ yields 
the map $\Re (H) \circ \gamma \colon \mathbb{S}^1 \to \mathbb{RP}^1$ given by 
\begin{equation*}
(\Re (H) \circ \gamma) (\theta) = 
\begin{cases}[1+\cos(\theta):-\sin(\theta)] & \theta \neq \pi, \\
[0:1] & \theta = \pi. 
\end{cases}
\end{equation*}
This map has topological Brouwer degree $-1$. Hence our computation shows 
\begin{align*}
\degt(\Re(F \boxplus \pi)\circ \gamma) = -1 \neq 3 =\degt(\Re(F) \circ \gamma) + \degt(\Re(\pi) \circ \gamma).     
\end{align*}
Thus, by the analogous statement of Corollary \ref{cor:concrete_obstruction} for $\boxplus$, we see that $\boxplus$ cannot be used to define an operation compatible with the conventional group structure. 
\end{example}

%%%%%%%%%%%%%%%%%%%%%%%%%%%%

%Remember to update/double-check directory of bib-arxiv when we're done editing and make a ArXivVersion3 folder. 
\bibliography{bibliography}
\bibliographystyle{gtart}

\end{document}